\RequirePackage{fix-cm}
\documentclass[smallextended]{svjour3}       
\smartqed  

\usepackage{graphicx}      

\usepackage{lmodern,enumerate,rotating}
\usepackage{amsmath,amsfonts,amssymb,subfigure,mathtools,blkarray}
\usepackage{tikz-cd}
\usetikzlibrary{cd}
\usetikzlibrary{shapes,shapes.geometric,arrows,fit,calc,positioning,automata}

\DeclarePairedDelimiter\ceil{\lceil}{\rceil}

\def\nUparrow{\rotatebox[]{90}{$\nRightarrow$}}

\DeclareMathOperator{\Acc}{Acc}
\DeclareMathOperator{\CCa}{CC_A}
\DeclareMathOperator{\CCn}{CC_N}
\DeclareMathOperator{\CCne}{CC_N^E}
\DeclareMathOperator{\Obs}{Obs}

\DeclareMathOperator{\Bifur}{Bifur}

\allowdisplaybreaks

\newcommand{\Z}{\mathbb{Z}}

\newcommand{\N}{\mathbb{N}}

\newcommand{\LM}{\mathcal{L}}

\newcommand{\Tt}{\mathcal{T}}

\newcommand{\ep}{\epsilon}

\newcommand{\Scal}{\mathcal{S}}

\newcommand{\Sig}{\Sigma}
\newcommand{\s}{\sigma}

\newcommand{\Rt}{\mathcal{R}}
\newcommand{\Mt}{\mathcal{M}}

\newtheorem{assumption}{Assumption}

\newtheorem{remark}{Remark}

\newtheorem{example}{Example}
\usepackage{times,amsmath,amssymb,graphicx,tikz,algorithm,algorithmic,url,hhline,booktabs}
\usetikzlibrary{automata,arrows,positioning}
\usetikzlibrary{petri}

\makeatletter
\newcommand{\subalign}[1]{%
  \vcenter{%
    \Let@ \restore@math@cr \default@tag
    \baselineskip\fontdimen10 \scriptfont\tw@
    \advance\baselineskip\fontdimen12 \scriptfont\tw@
    \lineskip\thr@@\fontdimen8 \scriptfont\thr@@
    \lineskiplimit\lineskip
    \ialign{\hfil$\m@th\scriptstyle##$&$\m@th\scriptstyle{}##$\crcr
      #1\crcr
    }%
  }
}
\makeatother

\usepackage[colorlinks]{hyperref}

\definecolor{green}{rgb}{0.1,0.7,0.1}

\begin{document}

\tikzset{elliptic state/.style={draw,ellipse}}

\title{On detectability of labeled Petri nets and finite automata\thanks{The main results
on labeled Petri nets shown in Section~\ref{sec:WeakDetect} were presented at WODES'18
\cite{Zhang2018WODESDetectabilityLPS}. Compared to the original conference version
\cite{Zhang2018WODESDetectabilityLPS},
the current journal version has been significantly expanded. Especially the results in
Section~\ref{sec:StrongDetectabilityLPN} (which have not been submitted anywhere yet)
cannot be obtained by using techniques used in Section~\ref{sec:WeakDetect}.
This work was partially supported by National Natural Science Foundation of China (No. 61603109),
and Natural Science Foundation of Heilongjiang Province of China (No. LC2016023).
}
}


\author{Kuize Zhang \and
        Alessandro Giua
}


\institute{K. Zhang \at
			  School of Electrical Engineering and Computer Science, KTH Royal Institute of Technology, 10044 Stockholm, Sweden\\
			  College of Automation, Harbin Engineering University, Harbin 150001, P.R. China \\
              \email{zkz0017@163.com, kuzhan@kth.se}           
           \and
           A. Giua \at
			  Department of Electrical and Electronic Engineering, University of Cagliari, 09123 Cagliari, Italy\\
			  \email{giua@diee.unica.it}
}

\date{Received: date / Accepted: date}

\maketitle

\begin{abstract}
		Detectability is a basic property of dynamic systems:
	when it holds an observer can use the \emph{current} and \emph{past} values of the observed output signal produced by a system to reconstruct its \emph{current} state.

	In this paper, we consider properties of this type in the framework of discrete-event systems
	modeled by labeled Petri nets and finite automata.
	We first study {\it weak approximate detectability}.
	This property implies that there exists an infinite observed output sequence of
	the system such tcheck each prefix of the output sequence with length
	greater than a given value allows an observer to determine if the current state belongs to a given set.
	We prove that the problem of verifying this property is undecidable
	for labeled Petri nets,
	and PSPACE-complete for finite automata.

	We also consider two new concepts called {\it instant strong detectability} and
	{\it eventual strong detectability}.
	The former property implies that for each possible infinite observed output sequence
	each prefix of the output sequence allows reconstructing the current state.
	The latter implies that for each possible infinite observed output sequence,
	there exists a value such that
	each prefix of the output sequence with length greater than that value allows reconstructing
	the current state. We prove that for labeled Petri nets,
	the problems of verifying instant strong detectability and eventual strong detectability
	are decidable and EXPSPACE-hard, where the decidability result for eventual strong detectability
	holds under a mild promptness assumption, but the result for instant strong detectability
	holds without any assumption.
	For finite automata, we give polynomial-time verification algorithms for 
	both properties. We also give a polynomial-time verification algorithm for strong detectability
	of finite automata, which strengthens the corresponding result given by [Shu and Lin 2011] based on 
	the usual assumptions of deadlock-freeness and promptness
	(collected in Assumption \ref{assum1_Det_PN}).
	In addition, we prove that strong detectability is strictly stronger
	than eventual strong detectability, but strictly weaker than
	instant strong detectability, for labeled Petri nets and even for deterministic finite automata
	satisfying Assumption \ref{assum1_Det_PN}.
	In particular, for deterministic finite automata 
	such that every event can be directly observed, 
	we prove that eventual strong detectability is equivalent to strong detectability.

\end{abstract}

\section{Introduction}

{\it Detectability} is a basic property of dynamic systems: when it holds an observer can use the \emph{current} and \emph{past} values of the observed output signal produced by a system to reconstruct its \emph{current} state
\cite{Giua2002ObservabilityPetriNets,Shu2007Detectability_DES,Shu2011GDetectabilityDES,Shu2013DelayedDetectabilityDES,Fornasini2013ObservabilityReconstructibilityofBCN,Xu2013ObserverFA_STP,Zhang2015WPGRepresentationReconBCN,RuHadjicostis2010OptimalStructrualObservabilityPetriNets,Yin2017DetectabilityModularDES,Masopust2018ComplexityDetectabilityDES,Keroglou2015DetectabilityStochDES}.
This property plays a fundamental role in many related
control problems such as observer design and controller synthesis. Hence for different applications,
it is meaningful to characterize different notions of detectability.
This property also has different terminologies, e.g.,
in \cite{Giua2002ObservabilityPetriNets,Xu2013ObserverFA_STP,RuHadjicostis2010OptimalStructrualObservabilityPetriNets},
it is called ``observability'' while  in
\cite{Fornasini2013ObservabilityReconstructibilityofBCN,Zhang2015WPGRepresentationReconBCN},
it is called ``reconstructibility''. In this paper, we uniformly call this property ``detectability'', and call
another similar property ``observability'' implying that the {\it initial state} can be determined
by the observed output signal produced by a system (e.g., \cite{Yin2017InitialStateDetectabilityStoDES,Shu2013IDetectabilityDES,Zhang2017ObservabilityFLTS,Zhang2014ObservabilityofBCN}).

\subsection{Literature review}

\subsubsection*{Finite automata}

For {\it discrete-event systems} (DESs) modeled by
{\it finite automata}, the detectability problem has been widely studied
\cite{Shu2007Detectability_DES,Shu2011GDetectabilityDES,Zhang2017PSPACEHardnessWeakDetectabilityDES,Masopust2018ComplexityDetectabilityDES,Yin2017DetectabilityModularDES} 
in the context of {\it $\omega$-languages}, i.e., taking into account all output sequences
of infinite length
generated by a DES. These results are usually based on two assumptions that a system is deadlock-free and that it cannot generate an infinitely long subsequence of unobservable events. These requirements are collected in Assumption~\ref{assum1_Det_PN} formally stated in the following sections: when it holds, a system will always run and generate an infinitely long observation.

Two fundamental definitions are those of {\it strong detectability} and {\it weak detectability}
\cite{Shu2007Detectability_DES}.
Strong detectability implies\footnote{Formal definitions of strong and weak detectability are given later in
Definitions \ref{def1_Det_PN} and \ref{def4_Det_PN}.} that:
\begin{quote}
  (A)  there exists a positive integer $k$ such that for all infinite output sequences $\sigma$ generated by a system, all  prefixes of $\sigma$ of length greater than $k$  allow reconstructing the current states.
\end{quote}
Weak detectability implies that:
\begin{quote}
  (B)  there exists a positive integer $k$ and some infinite output sequence $\sigma$ generated by a system such that all prefixes of $\sigma$ of length greater than $k$ allow reconstructing the current states.
\end{quote}

Weak detectability is strictly weaker than strong detectability. Consider
the finite automaton shown in Fig. \ref{fig3_Det_PN}, where events $a$ and $b$ can be directly observed.
It is weakly detectable but not strongly detectable.
The automaton can generate infinite event sequences $a^{\omega}$ and $b^{\omega}$, where
$(\cdot)^{\omega}$ denotes the concatenation of infinitely many copies of $\cdot$.
When any number of $a$'s were observed but no $b$ was observed,
the automaton could be only in state $s_0$. Hence it is weakly detectable. When
any number of $b$'s were observed but no $a$ was observed, it could be in states $s_1$ or $s_2$.
Hence it is not strongly detectable.

Strong detectability can be verified in polynomial time while weak detectability can
be verified in exponential time \cite{Shu2007Detectability_DES,Shu2011GDetectabilityDES}
the usual Assumption \ref{assum1_Det_PN}. 

In addition, checking weak detectability is PSPACE-complete in the numbers of states and events
for finite automata, where the hardness result holds for deterministic finite automata
whose events can be directly observed 
\cite{Zhang2017PSPACEHardnessWeakDetectabilityDES}. The hardness result even holds for
more restricted deterministic finite automata
having only two events that can be directly observed \cite{Masopust2018ComplexityDetectabilityDES}.

	\begin{figure}
        \centering
\begin{tikzpicture}[>=stealth',shorten >=1pt,auto,node distance=1.5 cm, scale = 1.0, transform shape,
	->,>=stealth,inner sep=2pt,state/.style={shape=circle,draw,top color=red!10,bottom color=blue!30}]

	\node[initial,state] (1) {$s_0$};
	\node[state] (2) [above right of =1] {$s_1$};
	\node[state] (3) [below right of =1] {$s_2$};
	
	\path [->]
	(1) edge [loop below] node {$a$} (1)
	(1) edge node {$b$} (2)
	(1) edge node {$b$} (3)
	(2) edge [loop right] node {$b$} (2)
	(3) edge [loop right] node {$b$} (3)
	;

        \end{tikzpicture}
		\caption{A finite automaton.}
	\label{fig3_Det_PN}
	\end{figure}
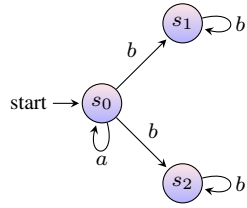

\subsubsection*{Petri nets}


Detectability of free-labeled Petri nets with unknown initial {\it markings}
(i.e., states)
has been studied in \cite{Giua2002ObservabilityPetriNets}, where several types of detectability
called ``(strong) marking observability'', ``uniform (strong) marking observability'',
and ``structural (strong) marking observability'' are proved to be decidable\footnote{In the sequel,
we will always use the expression ``a property is decidable/undecidable'' instead of
``the problem of verifying the property
is decidable/undecidable.''} by reducing them
to several decidable home space properties \cite{Escrig1989DecidabilityHomeSpaceProperty}
that are more general than the reachability problem
of Petri nets (with respect to a given marking).

Some detectability properties of {\it labeled Petri nets}\footnote{More precisely  {\it labeled place/transition nets} or \emph{labeled P/T nets} for short} have also been studied.
In \cite{RuHadjicostis2010OptimalStructrualObservabilityPetriNets}, a notion of detectability called
``structural observability'' is characterized. This property implies that for every initial marking,
each observed {\it label} (i.e., output) sequence determines the current marking.
It is pointed out that the ``structural observability'' is important, because ``the majority of existing control
schemes for Petri nets rely on complete knowledge of the system state at any given time step''
\cite{RuHadjicostis2010OptimalStructrualObservabilityPetriNets}.
It is shown that
structural observability can be verified in polynomial time
\cite{RuHadjicostis2010OptimalStructrualObservabilityPetriNets}.
In the same paper, in order to make a labeled Petri net structurally observable,
the problem of placing the minimal number
of sensors on places and the problem of placing the minimal number of sensors on transitions are studied,
respectively. The former problem
is proved to be NP-complete, while the latter is shown to be solvable in polynomial time, both in the numbers of
places and transitions.

In \cite{Jancar1994BisimilarityPetriNet}, for labeled Petri nets,
a concept of determinism is characterized, where this concept implies that each label sequence
generated by a net can be used to determine the current marking. It is proved that verifying determinism
is as hard as verifying {\it coverability} for Petri nets \cite{Rackoff19782EXPSPACECoverabilityPetriNets,Lipton1976ReachabilityPetriNetsEXPSPACE-hard},
hence EXPSPACE-complete.
Note that the ``structural observability'' studied in
\cite{RuHadjicostis2010OptimalStructrualObservabilityPetriNets} requires a labeled Petri net
to satisfy the determinism property at each initial marking.

The above mentioned detectability results for labeled Petri nets apply to finite-length
languages of the nets,
i.e., the set of all words (of finite length) that a net can generate.
In the sequel, we always use terminology ``language'' to denote ``finite-length language'' for short,
and use ``$\omega$-language'' to denote a ``language'' consisting of several infinite-length label sequences.
However, a few authors have recently studied detectability properties of $\omega$-languages
extending to labeled Petri net models the notions of strong
and weak detectability which Shu and Lin have originally studied in the context of finite automata.



Weak detectability of labeled Petri nets with inhibitor arcs has been proved to be undecidable in
\cite{Zhang2018WODESDetectabilityLPS} by reducing the well known undecidable {\it language equivalence
problem} \cite[Theorem 8.2]{Hack1976PetriNetLanguage} of labeled Petri nets to the inverse problem of the weak
detectability problem, i.e., the non-weak detectability problem.

Decidability and complexity of strong detectability and weak detectability
for labeled Petri nets are also studied in \cite{Masopust2018DetectabilityPetriNet}.
Under \eqref{item11_Det_PN} of Assumption \ref{assum1_Det_PN} and another assumption that
a net cannot generate an infinite unobservable sequence which is actually equivalent to
\eqref{item12_Det_PN} of Assumption \ref{assum1_Det_PN} for Petri nets,
strong detectability has been
proved to be decidable with EXPSPACE-hard complexity
by reducing its negation to the satisfiability of a
{\it Yen's path formula} \cite{Yen1992YenPathLogicPetriNet,Atig2009YenPathLogicPetriNet}.
Weak detectability has been proved to be undecidable by reducing the undecidable
{\it language inclusion problem} \cite[Theorem 8.2]{Hack1976PetriNetLanguage} to the non-weak detectability problem,
thus improving the related result given in \cite{Zhang2018WODESDetectabilityLPS}.

\subsection{Contribution of the paper}

In this paper, we propose some new notions of detectability in the context of $\omega$-languages, and characterize the related decision problems (in terms of decidability or computational complexity) for both finite automata and labeled Petri nets.

To motivate the interest for this work, let us recall that the theory of $\omega$-languages is a rich and important domain of computer science \cite{Pin2004InfiniteWords}. We mention, in addition, that these languages have a practical interest in automatic control because they can describe the infinite behavior of a system: for this reason they find significant applications in the very active area of verification with discrete-event and hybrid systems --- in particular model checking with temporal logic.



\subsubsection*{Instant detectability}

The notions of strong and weak detectability considered in \cite{Shu2007Detectability_DES,Shu2011GDetectabilityDES,Masopust2018DetectabilityPetriNet}
assume that an observer may be able to reconstruct the current state of a system only after a transient period characterized by a number $k$ of generated outputs/labels. 
However, in many applications, e.g., those concerning safety-critical systems,
it may be necessary to reconstruct  the current value of the state at all times and thus 
this transient should have length $k=0$. We denote this notion by \emph{instant detectability}.

It may be possible to consider this notion in different settings. When languages are
considered,
the strong version of this property is equivalent to the classical notion of determinism 
\cite{Jancar1994BisimilarityPetriNet}. In the case of Petri nets, the further requirement that
the property holds for every initial marking leads to even stronger notion of structural 
observability \cite{RuHadjicostis2010OptimalStructrualObservabilityPetriNets}.

In the case of $\omega$-languages, the stronger version of this property is strictly weaker
than determinism as we will show in Fig. \ref{fig1_Det_PN}.
In this paper, we study \emph{instant strong detectability} which implies that all prefixes
of all infinite output sequences generated by a system allow reconstructing 
the current states. This notion has been studied in \cite{Shu2013DelayedDetectabilityDES} for finite automata satisfying Assumption \ref{assum1_Det_PN} and is called  $(0,0)$-detectability.
Actually, a more general $(k_1,k_2)$-detectability is characterized in 
\cite{Shu2013DelayedDetectabilityDES}
which describes strong detectability with computation delays, 
and a polynomial-time verification algorithm is given under Assumption \ref{assum1_Det_PN}.


We will prove that instant strong detectability of labeled P/T nets is decidable, by reducing its negation to the satisfiability of a Yen's path formula.
We will also prove that the corresponding decision problem is EXPSPACE-hard by reducing the coverability problem of Petri nets to the non-instant strong detectability problem.
For finite automata, we will give a polynomial-time verification 
algorithm for instant strong detectability without any 
assumption by using a concurrent-composition method, which strengthens the corresponding 
algorithm given in \cite{Shu2013DelayedDetectabilityDES} under Assumption \ref{assum1_Det_PN}.

We point out that it may also be possible to consider the dual notion of
\emph{instant weak detectability} which implies that there exists some generated infinite output
sequence such that all its prefixes allow reconstructing the current states.
However, we are not going to study this property in this paper.

\subsubsection*{Eventual detectability}

Let us consider again the notion of strong dectability implied by condition (A) stated above. An alternative definition could be based on the following definition:
\begin{quote}
  (A')  for every infinite output sequence $\sigma$ generated by a system,
  there exists a positive integer $k_{\sigma}$ such that all  prefixes of $\sigma$ of 
  length greater than $k_{\sigma}$  allow reconstructing the current states,
\end{quote}
where the length $k_{\sigma}$ of the transient before the state can be reconstructed 
may depend on a particular output sequence $\sigma$.

Obviously, condition (A) implies condition (A') but the converse implication does not hold, because there may exist infinitely many strings of infinite length and thus a maximal value among all $k_{\sigma}$ may not be computed (this will be formally proved in Proposition~\ref{prop5_Det_PN}).

We point out some similarities with the notion of diagnosability introduced by
Lafortune and co-authors \cite{Sampath1995DiagnosabilityDES} which requires
the occurrence of a fault to be detected within a finite delay.
The original definition in \cite{Sampath1995DiagnosabilityDES} assumes this delay may
depend on the string that produces the fault, i.e., it is similar to condition (A') above.
A different condition, similar to condition (A) above and called $K$-step diagnosability,
is considered in \cite{Cabasino2012DiagnosabilityPetriNet}:
it assumes the length of the delay is bounded for all strings.
Note however a difference with respect to the detectability results we present here:
the two notions of diagnosability and $K$-step diagnosability are equivalent in the case
of finite automata, thanks to the well-known Myhill-Nerode characterization of a regular
language by the finiteness of its set of
residuals. They only differ for infinite-state systems, such as labeled Peri nets.



Based on condition (A'), we consider a new type of detectability, which we call {\it eventual strong detectability}.
Formally, eventual strong detectability implies that for every infinite
output sequence $\s$ generated by a system, there exists a positive integer $k_{\s}$
such that each
prefix $\s'$ of $\s$ with length greater than $k_{\s}$
allows reconstructing the current state.
We will prove that eventual strong detectability is strictly weaker than strong detectability
and strictly stronger than weak detectability, for labeled Petri nets and even for deterministic
finite automata satisfying Assumption \ref{assum1_Det_PN}.

We will also prove that eventual strong detectability can be verified in polynomial time for finite automata.
For labeled Petri nets, we show that the property is decidable and the corresponding decision problem is EXPSPACE-hard:
note that this decidability result holds  under the \emph{promptness} assumption
(collected in \eqref{item6_Det_PN} of Assumption \ref{assum2_Det_PN}) that is actually 
equivalent to condition \eqref{item12_Det_PN} of Assumption~\ref{assum1_Det_PN} for labeled Petri nets.

\subsubsection*{Approximate detectability}

State estimation is usually a preliminary step that a plant operator must address so that,
depending on the state value, a suitable action may be taken. Examples include computing a
control input in supervisory control, raising an alarm in fault diagnosis, inferring a secret
in an opacity problem, reacting to the detection of a cyber-attack, etc.
The number of these possible actions is usually finite and this naturally determines a
finite partition of the system's state space into equivalence classes, each one
corresponding to states for which the same action should be taken. In such a context,
it is not necessary to solve a detectability problem, i.e., determine the exact value of the state,
but just to solve an approximate version of it, i.e., determine to which class the state belongs.

The notion of approximate detectability applies to all previously defined detectability notions,
weak or strong, instant or eventual. Here we just study one of them, namely  
{\it weak approximate detectability} which implies that, given a finite partition of the 
state space, there exists an integer $k$ and an infinite output sequence generated by 
a system each of whose prefixes of
length greater than $k$ allows determining the partition cell to which the current state belongs.
In this paper, we will prove that
weak approximate detectability is undecidable for labeled P/T nets.
For finite automata, we will prove that deciding this property is PSPACE-complete.
The undecidable result is obtained by
reducing the undecidable language equivalence problem
for labeled P/T nets to negation  of the weak approximate detectability problem.
The result for finite automata is obtained by using related results for
weak detectability of finite automata 
\cite{Zhang2017PSPACEHardnessWeakDetectabilityDES,Shu2007Detectability_DES}.

\subsection{Paper structure}

To help the reader better understand the contribution of the paper, the relations among the different detectability properties studied in this work are shown in
Tabs. \ref{tab1:_Det_PN} and \ref{tab2:_Det_PN}. The table also includes known results on strong detectability
and weak detectability of finite automata and labeled Petri nets proved in
\cite{Masopust2018DetectabilityPetriNet,Zhang2017PSPACEHardnessWeakDetectabilityDES}.

\begin{table}
	\centering
	\begin{tabular}{ccccc}
		\hline
		\begin{tabular}{c}Instant\\ strong\\ detectability\\
			\hline
			decidable\\ (Thm. \ref{thm3_Det_PN})\\
			\hline
			EXPSPACE-hard\\ (Thm. \ref{thm3_Det_PN})
		\end{tabular}& $\begin{tabular}{c} $\Rightarrow$\\
			$\nLeftarrow$\\(Fig. \ref{fig8_Det_PN})\end{tabular}$ &
		\begin{tabular}{c}Strong\\ detectability \\
			\hline
			decidable\\ (\cite{Masopust2018DetectabilityPetriNet})\\
			\hline
			EXPSPACE-hard\\ (\cite{Masopust2018DetectabilityPetriNet})
		\end{tabular}& $\begin{tabular}{c} $\Rightarrow$\\
			$\nLeftarrow$\\(Fig. \ref{fig5_Det_PN})\end{tabular}$ &
		\begin{tabular}{c}Eventual\\ strong\\ detectability\\
			\hline
			decidable\\ (Thm. \ref{thm4_Det_PN})\\
			\hline
			EXPSPACE-hard\\ (Thm. \ref{thm4_Det_PN})
		\end{tabular}\\
		\hline
		&&&& \begin{tabular}{c}$\Downarrow$\\ $\nUparrow$ (Fig. \ref{fig9_Det_PN})\\
		\end{tabular}\\\hline &&
		\begin{tabular}{c}
			Weak\\ approximate\\ detectability\\\hline
			undecidable\\
			(Thm. \ref{thm2_Det_PN})
		\end{tabular} &$\begin{tabular}{c} $\Leftarrow$\\
			$\nRightarrow$\\(Fig. \ref{fig10_Det_PN})\end{tabular}$
		&
		\begin{tabular}{c}
			Weak \\ detectability\\\hline
			undecidable\\ (\cite{Masopust2018DetectabilityPetriNet})
		\end{tabular}\\
		\hline
	\end{tabular}
	\caption{Relationships among different detectability notions for labeled Petri nets, where
	$\Rightarrow$ means ``imply by definition'', $\nRightarrow$ means ``does not imply'',
	the decidability result for strong detectability proved in \cite{Masopust2018DetectabilityPetriNet}
	is based on Assumption \ref{assum1_Det_PN}, and can be strengthened to hold only based on the 
	promptness assumption which is actually \eqref{item12_Det_PN} of Assumption \ref{assum1_Det_PN} for 
	labeled Petri nets by using our proposed extended concurrent composition method
	similarly as in the proof of Theorem \ref{thm3_Det_PN}.
	The decidability result for eventual strong detectability is also based on the promptness 
	assumption.}
	\label{tab1:_Det_PN}
\end{table}

\begin{table}
	\centering
	\begin{tabular}{ccccc}
		\hline
		\begin{tabular}{c}Instant\\ strong\\ detectability\\
			\hline
			$O(s^2e)$\\ (Thm. \ref{thm6_Det_PN})
		\end{tabular}& $\begin{tabular}{c} $\Rightarrow$\\
			$\nLeftarrow$\\(Fig. \ref{fig8_Det_PN})\end{tabular}$ &
		\begin{tabular}{c}Strong\\ detectability \\
			\hline
			$O(s^4e^2)$\\ (Thm. \ref{thm7_Det_PN})
		\end{tabular}& $\begin{tabular}{c} $\Rightarrow$\\
			$\nLeftarrow$\\(Fig. \ref{fig5_Det_PN})\end{tabular}$ &
		\begin{tabular}{c}Eventual\\ strong\\ detectability\\
			\hline
			$O(s^4e^2)$\\ (Thm. \ref{thm8_Det_PN})
		\end{tabular}\\
		\hline
		&&&& \begin{tabular}{c}$\Downarrow$\\ $\nUparrow$ (Fig. \ref{fig9_Det_PN})\\
		\end{tabular}\\\hline &&
		\begin{tabular}{c}
			Weak\\ approximate\\ detectability\\\hline
			PSPACE-complete\\
			(Thm. \ref{thm5_Det_PN})
		\end{tabular} &$\begin{tabular}{c} $\Leftarrow$\\
			$\nRightarrow$\\(Fig. \ref{fig10_Det_PN})\end{tabular}$
		&
		\begin{tabular}{c}
			Weak \\ detectability\\\hline
			PSPACE-complete\\ (\cite{Zhang2017PSPACEHardnessWeakDetectabilityDES})
		\end{tabular}\\
		\hline
	\end{tabular}
	\caption{Relationships among different detectability notions for finite automata, where
	$s$ and $e$ are the numbers of states and events,
	$\Rightarrow$ means ``imply by definition'', $\nRightarrow$ means ``does not imply'';
	the polynomial-time verification algorithm for strong detectability 
	given in \cite{Shu2011GDetectabilityDES} applies to finite automata satisfying 
	Assumption \ref{assum1_Det_PN}, but generally does not apply to finite automata not satisfying
	Assumption \ref{assum1_Det_PN}; the exponential-time verification algorithm
	for weak detectability given in \cite{Shu2011GDetectabilityDES} actually applies to finite automata
	satisfying the assumption of non-emptiness of generated $\omega$-languages that is weaker than 
	Assumption \ref{assum1_Det_PN}, and in this paper we will characterize how to verify the weaker 
	assumption and how to deal with the case when the weaker assumption is not satisfied.}
	\label{tab2:_Det_PN}
\end{table}

The remainder of the paper is as follows. Section~\ref{sec:pre} introduces necessary
preliminaries, including finite automata, labeled Petri nets, the language equivalence problem,
and the coverability problem, together with necessary 
tools such as Dickson's lemma, Yen's path formulae, etc.
Section~\ref{sec:WeakDetect} collects the results on weak approximate detectability for 
finite automata and labeled Petri nets.
Section~\ref{sec:StrongDetectabilityLPN} consists of the results on instant strong detectability
and eventual strong detectability also for both models, and a new verification 
algorithm for strong detectability of finite automata. 
Section~\ref{sec:conc} ends up with a short conclusion.
We first study weak approximate detectability because fewer tools are needed than in studying
instant strong detectability and eventual strong detectability.


\section{Preliminaries}\label{sec:pre}





\subsection{Labeled state-transition systems}

In order to formulate detectability notions in a uniform manner, we introduce
{\it labeled state-transition systems} (LSTSs) as follows, which contain finite automata
and labeled Petri nets as special cases.
An LSTS is formulated as a sextuple $${\mathcal S}=(X,T,X_0,\to,\Sig,\ell),$$
where $X$ is a set of states, $T$ a set of events,  $X_0\subset X$ a set of initial states,
$\to\subset X\times T\times X$ a transition relation, $\Sig$ a set of outputs (labels),
and $\ell:T\to\Sig\cup\{\epsilon\}$
a {\it labeling} function, where $\epsilon$ denotes the empty word.
As usual, we use $\ell^{-1}(\sigma)$ to denote the {\it preimage} $\{t\in T|\ell(t)=\sigma\}$ 
of an output $\sigma\in\Sig$.
A state $x\in X$ is called {\it deadlock} if 
$(x,t,x')\notin \to$ for any $t\in T$ and $x'\in X$.
$\mathcal{S}$ is called {\it deadlock-free} if it has no deadlock state.
Events with label $\epsilon$ are called
{\it unobservable}. Other events are called {\it observable}.
Denote $T=:T_o\dot{\cup} T_{\ep}$, where $T_o$ and $T_{\ep}$ are the sets of 
observable events, and unobservable events, respectively.
For an observable event $t\in T$, we say $t$ {\it can be directly observed} if $\ell(t)$
differs from $\ell(t')$ for any other $t'\in T$. 
Labeling function $\ell:T\to\Sig\cup\{\ep\}$ can be recursively
extended to $\ell:T^*\cup T^{\omega}\to\Sig^*\cup\Sig^{\omega}$ as $\ell(t_1t_2\dots)=\ell(t_1)
\ell(t_2)\dots$ and $\ell(\ep)=\ep$.
For all $x,x'\in X$ and $t\in T$,
we also denote $x\xrightarrow[]{t}x'$ if $(x,t,x')\in\to$. More generally, we denote all transitions
$x\xrightarrow[]{t_1}x_1$, $x_1\xrightarrow[]{t_2}x_2$, $\dots$, $x_{n-1}\xrightarrow[]{t_n}x_n$
by $x\xrightarrow[]{t_1 \dots t_n}x_n$ for short, where $n$ is a positive integer.
We say {\it a state $x'\in X$ is reachable from a state}
$x\in X$ if there exist $t_1,\dots,t_n\in T$ such that
$x\xrightarrow[]{t_1\dots t_n}x'$, where $n$ is a positive integer. 
We say {\it a subset $X'$ of $X$ is reachable from a state} $x\in X$ if some state of $X'$ is 
reachable from $x$. Similarly {\it a state $x\in X$ is reachable from a subset $X'$ of $X$}
if $x$ is reachable from some state of $X'$.
We call a state $x\in X$
{\it reachable} if either $x\in X_0$ or it is reachable from an initial state.
For an LSTS $\Scal$, we call the new LSTS the {\it accessible part} 
(denoted by $\Acc(\Scal)$) of $\Scal$
that is obtained from $\Scal$ by removing all non-reachable states.
An LSTS $\cal S$ is called {\it deterministic} if for all $x,x',x''\in X$ and all $t\in T$,
if $(x,t,x')\in\to$ and $(x,t,x'')\in\to$ then $x'=x''$.

Next we introduce necessary notions that will be used throughout this paper.
Symbols $\N$ and $\Z_{+}$ denote the sets of natural numbers and positive integers, respectively.
For a set $S$, $S^*$ and $S^{\omega}$ are used to denote the sets of finite sequences
(called {\it words}) of elements of $S$ including the empty word $\epsilon$
and infinite sequences (called {\it configurations}) of elements of $S$,
respectively. As usual, we denote $S^{+}=S^*\setminus\{\epsilon\}$.
For a word $s\in S^*$,
$|s|$ stands for its length, and we set $|s'|=+\infty$ for all $s'\in S^{\omega}$.
For $s\in S$ and natural number $k$, $s^k$ and $s^{\omega}$ denote the $k$-length word and configuration
consisting of copies of $s$'s, respectively.
For a word (configuration) $s\in S^*(S^{\omega})$, a word $s'\in S^*$ is called a {\it prefix} of $s$,
denoted as $s'\sqsubset s$,
if there exists another word (configuration) $s''\in S^*(S^{\omega})$ such that $s=s's''$.
For two natural numbers $i\le j$, $[i,j]$ denotes the set of all integers between $i$ and $j$ including $i$ and $j$;
and for a set $S$, $|S|$ its cardinality and $2^S$ its power set.
For a word $s\in S^*$, where $S=\{s_1,\dots,s_n\}$,
$\sharp(s)(s_i)$ denotes the number of $s_i$'s occurrences in $s$, $i\in[1,n]$.

For each $\s\in\Sig^*$, we denote by $\Mt({\cal S},\s)$ the set of
states that the system can be in after $\s$ has been observed, i.e., 
$\Mt({\cal S},\s):=
\{x\in X|(\exists x_0\in X_0)(\exists s\in T^+)[
(\ell(s)=\s)\wedge(x_0\xrightarrow[]{s}x)]\}$.
In addition, we set $\Mt({\cal S},\epsilon):=\Mt({\cal S},\epsilon)\cup X_0$.
Particularly, for all $X'\subset X$ we denote $\Mt(X',\ep):=X'\cup\{x\in X|(\exists x'\in X')
(\exists s\in T^+)[(\ell(s)=\ep)\wedge(x'\xrightarrow[]{s}x)]\}$;
and for all $\sigma\in\Sig^+$, we denote  $\Mt(X',\sigma):=\{
x\in X|(\exists x'\in X')(\exists s\in T^+)[(\ell(s)=\sigma)
\wedge(x'\xrightarrow[]{s}x)]\}$.
$\LM({\cal S})$ denotes the {\it language generated} by system $\cal S$,
i.e., $\LM({\cal S}):=\{\s\in\Sig^*|\Mt({\cal S},\s)\ne\emptyset\}$.
An infinite event sequence $t_1t_2\dots$$\in T^{\omega}$ is called {\it generated by}
$\cal S$ if there exist states $x_0,x_1,\dots$$\in X$ with $x_0\in X_0$ such that
for all $i\in\N$, $(x_i,t_{i+1},x_{i+1})\in\to$.
We use $\LM^{\omega}({\cal S})$ to denote the $\omega$-{\it language generated} by $\cal S$,
i.e., $\LM^{\omega}(\mathcal{ S}):=\{\s\in\Sigma^{\omega}|(\exists t_1t_2\dots$$\in T^{\omega}
\text{ generated by }\mathcal{S})[\ell(t_1t_2\dots)$$=\s]\}$.

\subsection{Finite automata}

A DES can be modeled by a finite automaton or a labeled Petri net.
In order to represent a DES, we consider a {\it finite automaton} as a finite LSTS
${\mathcal S}=(X,T,X_0,\to,\Sig,\ell)$, i.e.,
when $X,T,\Sig$ are finite.
Such a finite automaton is also obtained from
a standard finite automaton \cite{Sipser2006TheoryofComputation} by removing all accepting states,
replacing a unique initial state
by a set $X_0$ of initial states, and adding a labeling function $\ell$.
{\it In the sequel, a finite automaton always means a finite LSTS.}
Transitions $x\xrightarrow[]{t}x'$
with $\ell(t)=\ep$ are called {\it $\ep$-transitions} (or {\it unobservable transitions}), and other transitions are called
{\it observable transitions}.

\subsection{Labeled Petri nets}

A {\it net} is a quadruple $N=(P,T,Pre,Post)$, where $P$ is a finite set of {\it places}
graphically represented by circles;
$T$ is a finite set of {\it transitions} graphically represented by bars; $P\cup T\ne\emptyset$, $P\cap  T = \emptyset$;
$Pre: P \times T\to \N$ and $Post : P \times T \to \N$ are the {\it pre-} and {\it post-incidence functions} 
that specify the arcs directed from places to transitions, and vice versa.
Graphically $Pre(p,t)$ is the weight of
the arc $p\to t$ and $Post(p,t)$ is the weight of the arc $t\to p$ for all $(p,t)\in P\times T$.
The {\it incidence function} is defined as $C = Post - Pre$.

A {\it marking} is a map $M: P \to \N$ that assigns to each place of a net a natural number of tokens,
graphically represented by black dots. For a marking $M\in \N^P$, the restriction of $M$ to a subset $P'$
of $P$ is denoted by $M|_{P'}$.
For a marking $M\in \N^ P$, a transition $t\in T$ is called {\it enabled} at $M$ if $M(p)\ge Pre(p,t)$
for all $p\in P$, and is denoted by $M[t\rangle $, where as usual $\N^P$ denotes the set
of maps from $P$ to $\N$. 
An enabled transition $t$ at $M$ may {\it fire} and yield a new making $M'(p)=M(p)+C(p,t)$
for all $p\in P$,
written as $M[t\rangle M'$. As usual, we assume that at each marking and each time step, at most one transition fires.
For a marking $M$, a sequence $t_1\dots t_n$ of transitions is called enabled at $M$ if
$t_1$ is enabled at $M$,
$t_2$ is enabled at the unique $M_2$ satisfying $M[t_1\rangle M_2$, \dots,
$t_n$ is enabled at the unique $M_{n-1}$ satisfying $M[t_1\rangle\cdots[t_{n-1}\rangle M_{n-1}$.
We write the firing of $t_1\dots t_n$ at $M$ as $M[t_1\dots t_{n}\rangle$ for short,
and similarly denote the firing of $t_1\dots t_n$ at $M$ yielding $M'$ by $M[t_1\dots t_{n}\rangle M'$.
$\Tt(N,M_0) := \{s\in  T^*|M_0[s\rangle\}$ is used to denote the set of transition sequences enabled at $M_0$.
Particularly we have $M_0[\epsilon\rangle M_0$.
A pair $(N,M_0)$ is called a {\it Petri net} or a {\it place/transition net (P/T net)},
where $N=(P,T,Pre,Post)$ is a net,
$M_0:P\to \N$ is called the {\it initial marking},
and the Petri net evolves initially at $M_0$ as transition sequences fire. Denote the set of {\it reachable
markings} of the Petri net by $\Rt(N,M_0):=\{M\in \N^P|\exists s\in T^*,M_0[s\rangle M\}$.

A {\it labeled P/T net} is a quadruple $(N,M_0,\Sigma,\ell)$,
where $N$ is a net, $M_0$ is an initial marking,
$\Sig$ is an {\it alphabet} (a finite set of labels),
and $\ell: T\to \Sig \cup\{\epsilon\}$ is a {\it labeling function} that assigns to each transition
$t\in T$ a symbol of $\Sig$ or the empty word $\epsilon$, which means when a transition $t$
fires, its label $\ell(t)$ can be observed if $\ell(t)\in \Sig$; and nothing can be observed if
$\ell(t)=\epsilon$.
A transition $t\in T$ is called {\it observable} if $\ell(t)\in\Sig$, and called
{\it unobservable} otherwise.
Particularly, a labeling function $\ell:T\to \Sig$ is called {\it $\epsilon$-free}, and a P/T net
with an $\epsilon$-free labeling function is called an {\it $\epsilon$-free labeled P/T net}.
A Petri net is actually an $\epsilon$-free labeled P/T net with an injective labeling function.
For a labeled P/T net $G=(N,M_0,\Sig,\ell)$, the {\it language
generated by $G$} is denoted by $\mathcal{L}(G):=\{\s\in \Sig^*|\exists s\in T^*,M_0[s\rangle,\ell(s)=\s\}$, 
i.e., the set of labels of finite transition sequences enabled at the initial marking $M_0$.
We also say for each $\s\in\LM(G)$, $G$ generates $\s$. For $\s\in\Sig^{\omega}$, we say $G$ generates
$\s$ if an infinite event sequence $t_1t_2\dots$$\in T^{\omega}$ is enabled at $M_0$ (denoted $M_0[t_1t_2\dots
\rangle$) and $\ell(t_1t_2\dots)=\s$.
The set of infinite label sequences generated by $G$
is denoted by $\LM^{\omega}(G)$  (which is an {\it$\omega$-language}).

Note that for a labeled P/T net $G=(N,M_0,\Sig,\ell)$, when we observe a label sequence
$\s\in \Sig^*$, there may exist infinitely many firing transition sequences labeled by $\s$.
However, for an $\epsilon$-free labeled P/T net, when we observe a label sequence $\s$, there exist at most finitely
many firing transition sequences labeled by $\s$.
Denote by $\Mt(G,\s):=\{M\in \N^P|\exists s\in T^*,M_0[s\rangle M,\ell(s)=\s\}$,
the set of markings in which $G$ can be when $\s$ is observed.
Then for each $\s\in \Sig^*$, $\Mt(G,\s)$ is finite
for an $\epsilon$-free labeled P/T net $G$.

\subsection{The language equivalence problem}

The undecidable result proved in this paper is obtained by using the following language equivalence problem.

\begin{proposition}\label{prop1_Det_PN}\cite[Theorem 8.2]{Hack1976PetriNetLanguage}
	It is undecidable to verify whether two $\epsilon$-free labeled P/T nets with the same alphabet generate
	the same language.
\end{proposition}

\subsection{Dickson's lemma}

Let $P$ be a finite set. For every two elements $x$ and $y$ of $\N^{P}$, we say $x\le y$
if and only if $x(p)\le y(p)$ for all $p$ in $P$. We write $x<y$ if $x\le y$ and $x\ne y$.
For a subset $S$ of $\N^{P}$, an element $x\in S$
is called {\it minimal} if for all $y$ in $S$, $y\le x$ implies $y=x$. Dickson's lemma
\cite{Dickson1913DicksonLemma}
shows that for each subset $S$ of $\N^{P}$, there exist at most finitely many distinct minimal elements.
This lemma follows from the fact that every infinite sequence with all elements
in $\N^{P}$ has an increasing infinite
subsequence, where such an increasing subsequence can be chosen component-wise
\cite[Theorem 2.5]{Reutenauer1990MathematicsPetriNets}.
We will use Dickson's lemma to prove some decidable results for labeled P/T nets.

\subsection{The coverability problem}

We also need the following Proposition \ref{prop4_Det_PN} on the coverability problem
to obtain some main results on complexity.

\begin{proposition}\label{prop4_Det_PN}\cite{Rackoff19782EXPSPACECoverabilityPetriNets,Lipton1976ReachabilityPetriNetsEXPSPACE-hard} 
	It is EXPSPACE-complete to decide for a Petri net $G=(N,M_0)$ and a {\it destination marking}
	$M\in\N^{P}$ whether $G$ covers $M$, i.e., whether
	there exists a marking $M'\in\Rt(N,M_0)$ such that $M\le M'$.
\end{proposition}

	In \cite{Lipton1976ReachabilityPetriNetsEXPSPACE-hard}, it is proved that deciding coverability
	for Petri nets requires
	at least $2^{cn}$ space infinitely often for some constant $c>0$, where $n$ is the number of transitions.
	In \cite{Rackoff19782EXPSPACECoverabilityPetriNets}, it is shown that deciding this property
	for a Petri net requires at most space $2^{cm\log m}$ for some constant $c$,
	where $m$ is the {\it size} of the set of all transitions.
	For a Petri net $((P,T,Pre,Post),M_0)$, each transition $t\in T$ corresponds to a $|P|$-length vector
	$Post(\cdot,t)-Pre(\cdot,t)=:c(t)$ whose components are integers. The size of $t$ is the sum of
	the lengths of the binary representations of the components of $c(t)$ (where the length of $0$ is $1$).
	The size of $T$ is the sum of the sizes of all transitions of $T$, and is set to be the above $m$.

The coverability problem belongs to EXPSPACE \cite{Rackoff19782EXPSPACECoverabilityPetriNets}.
Proposition \ref{prop4_Det_PN} has been used to prove the EXPSPACE-hardness of checking diagnosability
\cite{Yin2017DiagnosabilityLabeledPetriNets} and prognosability \cite{Yin2017PrognosabilityLabeledPetriNets}
of labeled Petri nets.

\subsection{Infinite graphs}

Let $(V,E)$ be a {\it directed graph}, where $V$ is the {\it vertex} set, and $E\subset V\times V$ the {\it edge} set.
For each edge $(v,v')\in E$, also denoted by $v\to v'$, $v$ and $v'$ are called the {\it tail} and the
{\it head} of the edge, respectively, $v$ is called a {\it parent} of $v'$ and $v'$
is called a {\it child} of $v$. A directed graph is called {\it infinite} if it has infinitely many vertices. 
A {\it path} is a sequence of vertices connected by edges with the same direction,
i.e., a path is of one of the forms: (1) $\cdots \to v_{-1}\to v_0\to v_1\to\cdots$ (bi-infinite),
(2) $v_0\to v_1\to\cdots$ (infinite),
(3) $\cdots \to v_{-1}\to v_0$ (anti-infinite), or (4) $v_1\to \cdots \to v_n$ (finite).
For each finite path $v_1\to\cdots\to v_n$, $v_1$ is called an {\it ancestor} of $v_n$, and $v_n$
is called a {\it descendant} of $v_1$.
A directed graph $(V,E)$ is called a {\it tree} if there is a vertex $v_0$
without any parent (called {\it root}), any other vertex is a descendant of $v_0$ and 
the head of exactly one edge. A tree is called {\it locally finite} if each vertex has at most
finitely many children.


\subsection{Yen's path formulae for Petri nets}

The final tool that we will use to prove some decidable results is Yen's path formula
\cite{Yen1992YenPathLogicPetriNet,Atig2009YenPathLogicPetriNet} for Petri nets.
In \cite{Yen1992YenPathLogicPetriNet}, a concept of Yen's path formulae is proposed and
some upper bounds for verifying the satisfiability of the formulae are studied.
In addition, it is shown that many problems, e.g., the boundedness problem,
the coverability problem for Petri nets,
can be reduced to the satisfiability problem of some Yen's path formulae.
In \cite{Atig2009YenPathLogicPetriNet}, a special class of Yen's path formulae called {\it increasing}
Yen's path formulae is proposed. The main results of \cite{Atig2009YenPathLogicPetriNet} are stated as follows.
\begin{proposition}[\cite{Atig2009YenPathLogicPetriNet}]\label{prop6_Det_PN}
	The reachability problem for Petri nets can be
	reduced to the satisfiability problem of some Yen's path formula,
	and the satisfiability problem of each Yen's path formula can be reduced to the reachability problem
	for Petri nets with respect to the marking with all places empty, all in polynomial time.
	In addition, the satisfiability of each increasing Yen's path formula can be verified in EXPSPACE.
\end{proposition}

For a Petri net $(N,M_0)$, where $N=(P,T,Pre,Post)$ is a net, each Yen's path formula consists of the following
elements:

\begin{enumerate}
	\item {\it Variables}. There are two types of variables, namely, {\it marking variables}
		$M_1,M_2,\dots$ and {\it variables for transition sequences} $s_1,s_2,\dots$,
		where each $M_i$ denotes an indeterminate function in $\Z^{P}$ and
		each $s_i$ denotes an indeterminate finite sequence of
		transitions, $\Z$ is the set of integers.
	\item {\it Terms}. Terms are defined recursively as follows.
		\begin{enumerate}
			\item $\forall$ constant $c\in \N^{P}$, $c$ is a term.
			\item $\forall j>i$, $M_j-M_i$ is a term, where $M_i$ and $M_j$ are marking variables.
			\item $T_1+T_2$ and $T_1-T_2$ are terms if $T_1$ and $T_2$ are terms.
		\end{enumerate}
	\item {\it Atomic Predicates}. There are two types of atomic predicates, namely {\it transition
		predicates} and {\it marking predicates}.
		\begin{enumerate}
			\item Transition predicates.
				\begin{itemize}
					\item $y\odot\sharp(s_i)<c$, $y\odot\sharp(s_i)=c$, and $y\odot\sharp(s_i)>c$
						are predicates, where $i>1$, constant $y$ $\in \Z^{T}$, constant
						$c\in\N$, and $\odot$ denotes the inner product (i.e., $(a_1,\dots,a_{|T|})
						\odot(b_1,\dots,b_{|T|})=\sum_{i=1}^{|T|}a_kb_k$).
					\item $\sharp(s_1)(t)\le c$ and $\sharp(s_1)(t)\ge c$ are predicates, where
						constant $c\in \N$, $t\in T$.
				\end{itemize}
			\item Marking predicates.
				\begin{itemize}
					\item Type 1. $M(p)\ge c$ and $M(p)>c$ are predicates, where $M$ is a marking variable
						and $c\in \Z$ is constant.
					\item Type 2. $T_1(i)=T_2(j)$, $T_1(i)<T_2(j)$, and $T_1(i)>T_2(j)$ are predicates,
						where $T_1,T_2$ are terms and $i,j\in T$.
				\end{itemize}
		\end{enumerate}
	\item $F_1\vee F_2$ and $F_1\wedge F_2$ are predicates if $F_1$ and $F_2$ are predicates.
\end{enumerate}

A Yen's path formula $f$ is of the following form (with respect to Petri net $(N,M_0)$, where $N=(P,T,Pre,Post)$):
\begin{equation}\label{eqn12_Det_PN}
	\begin{split}
		&(\exists M_1,\dots,M_n\in\N^{P})(\exists s_1,\dots,s_n\in T^{*})
		[ (M_0[s_1\rangle M_1[s_2\rangle\cdots[s_n\rangle M_n)\\
		&\quad\wedge F(M_1,\dots,M_n,s_1,\dots,s_n)],
	\end{split}
\end{equation}
where $F(M_1,\dots,M_n,s_1,\dots,s_n)$ is a predicate.

Given a Petri net $G$ and a Yen's path formula $f$, we use $G\models f$ to denote that $f$ is true
in $G$. The satisfiability problem is the problem of determining, given a Petri net $G$ and a Yen's path
formula $f$, whether $G\models f$.

A Yen's path formula \eqref{eqn12_Det_PN} is called {\it increasing} if $F$ does not contain
transition predicates and implies $M_n\ge M_1$. When $n=1$, it naturally holds $M_n\ge M_1$,
then in this case an increasing Yen's path formula is
$(\exists M_1)(\exists s_1)[ (M_0[s_1\rangle M_1)\wedge F(M_1)]$.

The unboundedness problem can be formulated as the satisfiability of the increasing Yen's
path formula
$(\exists M_1,M_2)(\exists s_1,s_2)[ (M_0[s_1\rangle M_1[s_2\rangle M_2)\wedge(M_2>M_1)]$.

The coverability problem can be formulated as 
the satisfiability of the increasing Yen's
path formula
$(\exists M_1)(\exists s_1)[ (M_0[s_1\rangle M_1)\wedge(M_1\ge M)]$,
where $M$ is the destination marking.

\section{Weak approximate detectability}\label{sec:WeakDetect}


The concept of weak detectability is formulated as follows.

\begin{definition}[WD]\label{def1_Det_PN}
	Consider an LSTS ${\mathcal S}=(X,T,X_0,\to,\Sig,\ell)$.
	System $\cal S$ is called {\it weakly detectable} if there exists a label sequence
	$\s\in \LM^{\omega}(\cal S)$ such that for some positive integer $k$, $|\Mt({\cal S},\s')|=1$ 
	for every prefix $\s'$ of $\s$ satisfying $|\s'|\ge k$.
\end{definition}

Sometimes, we do not need to determine the current state of an LSTS,
but only need to know whether the current state belongs
to some prescribed subset of reachable states.
Then the concept of weak approximate detectability is formulated as below.

\begin{definition}[WAD]\label{def2_Det_PN}
	Consider an LSTS ${\mathcal S}=(X,T,X_0,\to,\Sig,\ell)$.
	Given a positive integer $n>1$ and 
	a partition $\{R_1,\dots,R_n\}$ of the set of its reachable states,
	$\cal S$ is called {\it weakly approximately detectable} with respect to partition $\{R_1,\dots,R_n\}$
	if there exists a label sequence
	$\s\in \LM^{\omega}(\cal S)$
	such that for some positive integer $k$, for every prefix $\s'$
	of $\s$ satisfying $|\s'|\ge k$,  $\emptyset\ne\Mt({\cal S},\s')\subset R_{i_{\s'}}$ for some
	$i_{\s'}\in [1,n]$.
\end{definition}

\subsection{Labeled Petri nets}

One directly sees that if an LSTS is weakly detectable, then it is weakly approximately
detectable with respect to every finite partition of its state space. However,
if it is weakly approximately detectable with respect to some finite
partition of its state space, then it is not necessarily weakly detectable.
See the following example.

\begin{example}
	Consider a labeled Petri net $G$  in Fig. \ref{fig10_Det_PN}. We have
	$\LM^{\omega}(G)=\{a^{\omega},b^{\omega}\}$. We also have for all $k\in\Z_{+}$,
	$\Mt(G,a^{k})=\{(0,1,0,0,0),(1,0,0,0,0)\}$,
	$\Mt(G,b^{k})=\{(0,0,0,1,0),(0,0,0,0,1)\}$,
	where the components of a marking is in the order $(p_{-2},p_{-1},p_0,
	p_1,p_2)$.
	These observations show that the net is not weakly detectable.
	It is weakly approximately detectable with respect to the partition:
	\begin{equation}\label{eqn19_Det_PN}
		\begin{split}
			&R_1=\{(0,0,1,0,0)\},\\
			&R_2=\{(0,0,0,1,0),(0,0,0,0,1)\},\\
			&R_3=\{(0,1,0,0,0),(1,0,0,0,0)\}
		\end{split}
	\end{equation}
	of the set of its reachable markings. Also, this net is actually a deterministic finite automaton
	if we regard labels $a$ and $b$ as labels of events,
	and $(0,0,1,0,0)$ as the unique initial state.
	Similarly we have the automaton is also weakly approximately detectable with respect to partition 
	\eqref{eqn19_Det_PN} but not weakly detectable.
	
		\begin{figure}[htbp]
		\tikzset{global scale/.style={
    scale=#1,
    every node/.append style={scale=#1}}}
		\begin{center}
			\begin{tikzpicture}[global scale = 1.0,
				>=stealth',shorten >=1pt,thick,auto,node distance=1.5 cm, scale = 0.8, transform shape,
	->,>=stealth,inner sep=2pt,
				every transition/.style={draw=red,fill=red,minimum width=1mm,minimum height=3.5mm},
				every place/.style={draw=blue,fill=blue!20,minimum size=7mm}]
				\tikzstyle{emptynode}=[inner sep=0,outer sep=0]
				\node[place, label=above:$p_0$, tokens=1] (p1) {};
				\node[transition, label=above:$b$,right of = p1] (t3) {}
				edge[pre] (p1);
				\node[transition, label=above:$b$,above of = t3] (t2) {}
				edge[pre] (p1);
				\node[place, label=above:$p_{1}$,  right of = t2] (p2) {}
				edge[pre] (t2);
				\node[place, label=above:$p_{2}$,  right of = t3] (p3) {}
				edge[pre] (t3);
				\node[transition, label=above:$b$,right of = p2] (t4) {}
				edge[pre, bend left] (p2)
				edge[post, bend right] (p2);
				\node[transition, label=above:$b$,right of = p3] (t5) {}
				edge[pre, bend left] (p3)
				edge[post, bend right] (p3);
				\node[transition, label=above:$a$,left of = p1] (t3'){}
				edge[pre] (p1);
				\node[transition, label=above:$a$,above of = t3'] (t2') {}
				edge[pre] (p1);
				\node[place, label=above:$p_{-1}$,  left of = t2'] (p2') {}
				edge[pre] (t2');
				\node[place, label=above:$p_{-2}$,  left of = t3'] (p3') {}
				edge[pre] (t3');
				\node[transition, label=above:$a$,left of = p2'] (t4') {}
				edge[pre, bend left] (p2')
				edge[post, bend right] (p2');
				\node[transition, label=above:$a$,left of = p3'] (t5') {}
				edge[pre, bend left] (p3')
				edge[post, bend right] (p3');
			\end{tikzpicture}
			\caption{A labeled P/T net $G$, where letters beside transitions
			denote their labels, each arc is with weight $1$.}
			\label{fig10_Det_PN}
		\end{center}
	\end{figure}
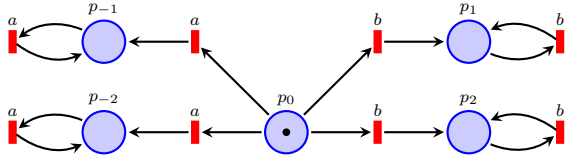

\end{example}

For the weak approximate detectability of labeled P/T nets, the following result holds.

\begin{theorem}\label{thm2_Det_PN}
	Let $n>1$ be a positive integer. 
	It is undecidable to verify for an $\epsilon$-free labeled P/T net
	and a partition $\{R_1,\dots,R_n\}$ of the set of its reachable markings, whether the labeled P/T net is 
	weakly approximately detectable with respect to $\{R_1,\dots,R_n\}$.
\end{theorem}



\begin{proof}
	We prove this result by reducing the language equivalence problem of labeled Petri nets
	(Proposition \ref{prop1_Det_PN}) to
	the problem under consideration. The proof is divided into three cases:
	$n=2$, $n=3$, and $n>3$.

	Let $l\ge 3$ be an integer.
	Arbitrarily given two $\epsilon$-free labeled P/T nets $G_i=(N_i,M_0^i,\Sig,\ell_i)$,
	where $N_i=(P_i,T_i,Pre_i,Post_i)$,
	$i=1,2$, $P_1\cap P_2=\emptyset$, $T_1\cap T_2=\emptyset$, we next construct a new $\epsilon$-free labeled P/T net
	$G=(N_G,M_0^G,\Sig\cup\{\s_G\},\ell_G)$ from $G_1$ and $G_2$.
	$G$ is specified as follows: (1) Add $l+2$ places $p_0,p_1^1,p_1^2,p_2,\dots,p_l$ to $G_1$ and $G_2$, where
	initially $p_0$ has one token, and all the other places have no token. (2) Add $l+3$ transitions $t_0^1,t_0^2,
	t_1^1,t_1^2,t_2,\dots,t_l$, and arcs $p_0\to t_0^1\to p_1^1\to t_1^1\to p_2\to t_2\to \cdots \to p_l\to t_l\to p_2$,
	and $p_0\to t_0^2\to p_1^2\to t_1^2\to p_l$, where these transitions are labeled by $\s_G\notin \Sig$.
	(3) For each transition $t\in T_i$, add arcs $p_1^i\to t\to p_1^i$, $i=1,2$.
	(4) All these new added arcs are with weight $1$. See Fig. \ref{fig2_Det_PN} as a sketch.

	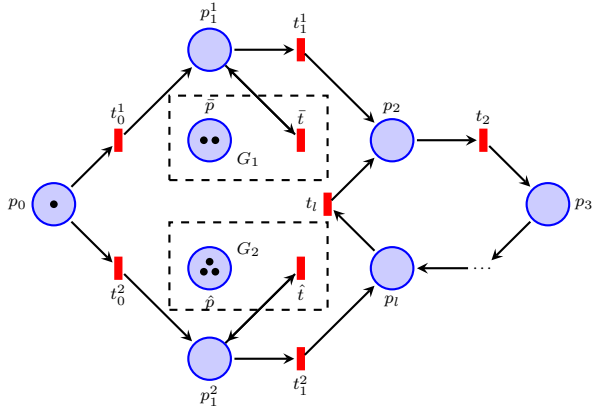
\begin{figure}[htbp]
		\tikzset{global scale/.style={
    scale=#1,
    every node/.append style={scale=#1}}}
		\begin{center}
			\begin{tikzpicture}[global scale = 1.0,
				>=stealth',shorten >=1pt,thick,auto,node distance=1.5 cm, scale = 0.8, transform shape,
	->,>=stealth,inner sep=2pt,
				every transition/.style={draw=red,fill=red,minimum width=1mm,minimum height=3.5mm},
				every place/.style={draw=blue,fill=blue!20,minimum size=7mm}]
				\node[place,label=left:$p_0$,tokens=1] (p0) {};
				\node[transition,label=above:$t_0^1$, above right of = p0] (t01) {}
				edge[pre] (p0);
				\node[transition,label=below:$t_0^2$, below right of = p0] (t02) {}
				edge[pre] (p0);
				\node[place,label=above:$\bar p$,tokens=2,right of = t01] (barp) {};
				\node[transition,label=above:$\bar t$,right of = barp] (bart) {};
				\node[place,label=above:$p_2$, right of = bart] (p2) {}
				;
				\node[place,label=above:$p_1^1$,above of = barp] (p11) {}
				edge[pre] (bart)
				edge[post] (bart)
				edge[pre] (t01);
				\node[transition,label=above:$t_1^1$,above of = bart] (t11) {}
				edge[pre] (p11)
				edge[post] (p2);
				\node[place,label=below:$\hat p$,right of = t02,tokens=3] (hatp) {};
				\node[transition,label=below:$\hat t$,right of = hatp] (hatt) {};
				\node[place,label=below:$p_l$,right of = hatt] (pl) {}
				;
				\node[place,label=below:$p_1^2$,below of = hatp] (p12) {}
				edge[pre] (t02)
				edge[pre] (hatt)
				edge[post] (hatt);
				\node[transition,label=below:$t_1^2$, below of = hatt] (t12) {}
				edge[pre] (p12)
				edge[post] (pl)
				;
				\draw[dashed] (1.9,-1.75) rectangle (4.5,-0.3);
				\node at (3.20,0.75) {$G_1$};
				\draw[dashed] (1.9,1.8) rectangle (4.5,0.4);
				\node at (3.20,-0.75) {$G_2$};
				\node[transition,label=left:$t_l$,below left of = p2] (tl) {}
				edge[pre] (pl)
				edge[post] (p2);
				\node[transition,label=above:$t_2$,right of = p2] (t2) {}
				edge[pre] (p2);
				\node[place,label=right:$p_3$,below right of = t2] (p3) {}
				edge[pre] (t2);
				\node[below left of = p3] (dot) {\dots}
				edge[pre] (p3)
				edge[post] (pl);
			\end{tikzpicture}
			\caption{Sketch for the reduction in the proof of Theorem \ref{thm2_Det_PN},
			where all transitions outside $G_1\cup G_2$ are with the same label.}
			\label{fig2_Det_PN}
		\end{center}
	\end{figure}

	For net $G$, initially only transition $t_0^1$ or $t_0^2$ can fire. After $t_0^1$ ($t_0^2$) fires,
	the unique token in place $p_0$ moves to place $p_1^1$ ($p_1^2$), initializing net $G_1$ ($G_2$).
	While $G_1$ ($G_2$) is running, only transition $t_1^1$ ($t_1^2$) outside $T_1\cup T_2$ can fire.
	The firing of $t_1^1$ ($t_1^2$)
	moves the token in place $p_1^1$ ($p_1^2$) to place $p_2$ ($p_l$), and terminates the running of
	$G_1$ ($G_2$), yielding that the token in $p_2$ ($p_l$) can move
	along the direction $p_2\to\cdots\to p_l\to p_2$ periodically forever, but
	$G_1$ ($G_2$) will never run again. Hence net $G$ may fire only infinite transition sequences
	$t_0^1 s t_1^1(t_2\dots t_l)^{\omega}$, $t_0^1 s'$, $t_0^2 r t_1^2t_l(t_2\dots t_l)^{\omega}$,
	or $t_0^2 r'$, where $s\in(T_1)^*$, $s'\in(T_1)^{\omega}$, $r\in(T_2)^*$,
	$r'\in(T_2)^{\omega}$. So $G$ can generate only configurations
	$\s_G\s(\s_G)^{\omega}$ or $\s_G\s'$
	where $\s\in\Sig^*$, $\s'\in\Sig^{\omega}$.
	Note that for some nets $G_1$ and $G_2$, the corresponding net $G$
	never fires $t_0^1s'$ or $t_0^2r'$ as above, e.g., when $\LM(G_1)\cup\LM(G_2)$ is finite;
	but for all $G_1$ and $G_2$, the corresponding $G$ fires
	$t_0^1 s t_1^1(t_2\dots t_l)^{\omega}$ and $t_0^2 r t_1^2t_l(t_2\dots t_l)^{\omega}$
	as above.

	$n=2$:

	Let $l=3$. We partition the set $\Rt(N_G,M_0^G)$ of reachable markings of net $G$ as follows:
	\begin{equation}\label{eqn11_Det_PN}
		\begin{split}
			R_1 =& \{M\in\N^{P_G}|M(p_0)\text{ or }M(p_1^1)\text{ or }M(p_2)=1,
			M(p_1^2)=M(p_3)=0\}\\
			&\cap\Rt(N_G,M_0^G),\\
			R_2 =& \{M\in\N^{P_G}|M(p_{1}^2)\text{ or }M(p_3)=1,
			M(p_0)=M(p_1^1)=M(p_2)=0\}\\
			&\cap \Rt(N_G,M_0^G).
		\end{split}
	\end{equation}

	If $\LM(G_1)\ne\LM(G_2)$, without loss of generality, we assume that there exists $\s\in\LM(G_1)\setminus\LM(G_2)$.
	Then when $G$ generates configuration $\s_G\s(\s_G)^{\omega}$, it can fire only transition sequences
	$t_0^1 st_1^1(t_2t_3)^{\omega}$,
	where $s\in (T_1)^*$, $\ell_G(s)=\s$. It can be directly seen 
	for each positive integer $k$,
	$\emptyset\ne\Mt(G,\s_G\s(\s_G)^k)\subset R_{k\!\!\!\mod\!2+1}$, where
	$k\mod 2$ means the remainder of $k$ divided by $2$.
	That is, net $G$ is weakly approximately detectable with respect to partition \eqref{eqn11_Det_PN}.

	Next we assume that $\LM(G_1)=\LM(G_2)$. Note that net $G$ generates only configurations
	$\s_G\s'$ or $\s_G\s(\s_G)^{\omega}$, where $\s'\in\Sig^{\omega}$, $\s\in\Sig^*$.
	For the former case, for each prefix $\s''$ of $\s'$, there exist firing sequences $s\in(T_1)^*$ of
	net $G_1$ and $r\in(T_2)^*$ of net $G_2$ such that $\ell_G(s)=\ell_G(r)=\s''$, and markings
	$M_G,M_G'\in\N^{P_G}$ such that $M_0^G[t_0^1 s\rangle M_G$, $M_0^G[t_0^2 r\rangle M_G'$,
	$M_G(p_1^1)=1$, $M_G(p_1^2)=0$, $M_G'(p_1^1)=0$, and $M_G'(p_1^2)=1$,
	then we have $\Mt(G,\s'')\cap R_{1}\ne\emptyset$ and $\Mt(G,\s'')\cap R_{2}\ne\emptyset$.
	For the latter case,
	chosen an arbitrary prefix $\s_G\s(\s_G)^k$ of $\s_G\s(\s_G)^{\omega}$, where $k$ is an arbitrary positive
	integer, we have there exist firing sequences $s\in(T_1)^*$ of net $G_1$ and
	$r\in(T_2)^*$ of net $G_2$ such that $\ell_G(s)=\ell_G(r)=\s$
	and net $G$ can fire both $t_0^1 ss'$ and
	$t_0^2 rr'$, where $s'$ and $r'$ are $k$ length prefixes of $(t_2t_3)^{\omega}$
	and $(t_3t_2)^{\omega}$, respectively. Since $G$ will fire both
	$t_0^1 ss'$ and $t_0^2 rr'$, we have
	$\Mt(G,\s_G\s(\s_G)^k)\cap R_{1}\ne\emptyset$ and
	$\Mt(G,\s_G\s(\s_G)^k)\cap R_{2}\ne\emptyset$.
	Hence for each positive integer $k$,
	$\Mt(G,\s_G\s(\s_G)^k)$ intersects both $R_{1}$ and $R_{2}$. 
	We have checked all label sequences generated by $G$,
	hence $G$ is not weakly approximately detectable with respect to partition
	\eqref{eqn11_Det_PN}.

	$n=3$:

	Let $l=3$. We partition the set $\Rt(N_G,M_0^G)$ of reachable markings of net $G$ as follows:
	\begin{equation}\label{eqn10_Det_PN}
		\begin{split}
			R_1 =& \{M\in\N^{P_G}|M(p_0)\text{ or }M(p_1^1)=1,
			M(p_1^2)=M(p_2)=M(p_3)=0\}\\
			&\cap\Rt(N_G,M_0^G),\\
			R_2 =& \{M\in\N^{P_G}|M(p_{2})=1,
			M(p_0)=M(p_1^1)=M(p_1^2)=M(p_3)=0\}\\
			&\cap \Rt(N_G,M_0^G),\\
			R_{3} =& \{M\in\N^{P_G}|M(p_1^2)\text{ or }M(p_3)=1,
			M(p_0)=M(p_1^{1})=M(p_2)=0\}\\
			&\cap \Rt(N_G,M_0^G).
		\end{split}
	\end{equation}

	Similarly to the case $n=2$, we also have that $\LM(G_1)\ne\LM(G_2)$ if and only if net $G$
	is weakly approximately detectable with respect to partition \eqref{eqn10_Det_PN}.

	$n>3$:

	Let $l=n-1$. We partition the set $\Rt(N_G,M_0^G)$ of reachable markings of net $G$ as follows:
	\begin{equation}\label{eqn1_Det_PN}
		\begin{split}
			R_1 =& \{M\in\N^{P_G}|M(p_0)\text{ or }M(p_1^1)=1,
			M(p_1^2)=M(p_j)=0,j\in[2,l]\}\\
			&\cap\Rt(N_G,M_0^G),\\
			R_i =& \{M\in\N^{P_G}|M(p_0)=M(p_1^1)=M(p_1^2)=0,M(p_{i})=1,M(p_j)=0,\\
			&j\in[2,l]\setminus\{i\}\}
			\cap \Rt(N_G,M_0^G),\quad i\in[2,l],\\
			R_{l+1} =& \{M\in\N^{P_G}|M(p_1^2)=1,
			M(p_0)=M(p_1^{1})=M(p_j)=0,j\in[2,l]\}\\
			&\cap \Rt(N_G,M_0^G).
		\end{split}
	\end{equation}

	Similarly we also have that $\LM(G_1)\ne\LM(G_2)$ if and only if net $G$
	is weakly approximately detectable with respect to partition \eqref{eqn1_Det_PN}.

\end{proof}

\subsection{Finite automata}

Next, we study the complexity of deciding weak approximate detectability of finite automata.

An exponential-time algorithm for verifying weak detectability of a finite automaton
$\Scal$ under Assumption
\ref{assum1_Det_PN} is given in \cite{Shu2011GDetectabilityDES}, but the algorithm actually 
applies to every $\Scal$ satisfying $\LM^{\omega}(\Scal)\ne\emptyset$ which is weaker than Assumption 
\ref{assum1_Det_PN}. Automaton $\Scal$ such that $\LM^{\omega}(\Scal)=\emptyset$ is naturally weakly detectable
and weakly approximately detectable (with respect to very finite partition of its set of reachable states)
as well,
and the condition $\LM^{\omega}(\Scal)\ne\emptyset$ can be verified in polynomial time (see Proposition
\ref{prop8_Det_PN}). 
Note that in Assumption \ref{assum1_Det_PN},
\eqref{item12_Det_PN} is actually a little weaker than the counterpart in 
\cite{Shu2007Detectability_DES,Shu2011GDetectabilityDES}, as in these two papers, there is no requirement
``reachable from an initial state''. However, one easily sees that existence of a cycle not 
reachable from an initial state consisting of only unobservable events does not violate the 
verification results for weak detectability given in \cite{Shu2011GDetectabilityDES}.

\begin{assumption}\label{assum1_Det_PN}
	An LSTS ${\mathcal S}=(X,T,X_0,\to,\Sig,\ell)$ satisfies
	\begin{enumerate}[(i)]
		\item\label{item11_Det_PN}
			$\mathcal S$ is deadlock-free, 
		\item\label{item12_Det_PN} 	
			no cycle in $\mathcal S$ reachable from an initial state contains only unobservable events,
			i.e.,
			for every reachable state $x\in X$ and every
			nonempty unobservable event sequence $s$, there exists no transition
			sequence $x\xrightarrow[]{s}x$ in $\mathcal S$.
	\end{enumerate}
\end{assumption}

In Assumption \ref{assum1_Det_PN}, \eqref{item11_Det_PN} guarantees that the automaton never halts,
\eqref{item12_Det_PN} ensures that for each infinite event sequence generated by the automaton, the
corresponding label sequence is also of infinite length.

\begin{proposition}\label{prop8_Det_PN}
	The property $\LM^{\omega}(\Scal)=\emptyset$ for a finite automaton $\mathcal S$
	can be verified in linear time of the size of $\Scal$. 
\end{proposition}

\begin{proof}
	Consider a finite automaton $\mathcal S=(X,T,X_0,\to,\Sig,\ell)$,
	it is not difficult to see that $\LM^{\omega}(\mathcal{S})\ne\emptyset$ if and only if 
	there is an infinite transition sequence $x_0\xrightarrow[]{s_1}x_1\xrightarrow[]{s_2}\cdots$
	with $x_0\in X_0$ such that $s_i\in T^*$ and $\ell(s_i)\in\Sig^{+}$ for each $i\in\Z_{+}$ if and only if
	there exists a transition sequence $x_0\xrightarrow[]{s_1'}x_1'\xrightarrow[]{s_2'}x_1'$
	with $x_0\in X_0$ and $\ell(s_2')\in\Sig^+$. 

	Construct an {\it observation automaton}
	\begin{equation}\label{eqn70_Det_PN}
		\mathcal \Obs(\Scal)=(X,\{\varepsilon,\hat\epsilon\},X_0,\to',\{\hat\epsilon\},\ell')
	\end{equation}
	in linear time of the size of $\Scal$,
	where $\to'\subset X\times\{\varepsilon,\hat\epsilon\}\times X$, $\ell'(\varepsilon)=\epsilon$,
	$\ell'(\hat\epsilon)=\hat\epsilon$,
	for every two states $x,x'\in X$, $(x,\hat\epsilon,x')\in\to'$
	if there exists $t\in T$ such that $(x,t,x')\in\to$ and $\ell(t)\ne\epsilon$; $(x,\varepsilon,x')
	\in\to'$ if there exists $t\in T$ such that $(x,t,x')\in\to$ and for all $t'\in T$
	with $(x,t',x')\in\to$, $\ell(t')=\epsilon$. Here the label function $\ell'$ is also naturally
	extended to $\ell':\{\varepsilon,\hat\epsilon\}^*\cup\{\varepsilon,\hat\epsilon\}^{\omega}\to
	\{\hat\epsilon\}^*\cup\{\hat\epsilon\}^{\omega}$.
	One sees that $\LM^{\omega}(\mathcal S)\ne\emptyset$
	if and only if in $\Obs(\Scal)$ there is a transition sequence $x_0\xrightarrow[]{s}x
	\xrightarrow[]{s'}x$ such that $x_0\in X_0$, $s,s'\in\{\varepsilon,\hat\epsilon\}^*$,
	and $\ell'(s')\ne\epsilon$.
	Next, we show that this condition can be trivially verified in linear time of the size
	of $\Scal$. 

	Firstly, find the accessible part $\Acc(\Obs(\Scal))$, which takes linear time.
	Secondly, compute all strongly connected components of $\Acc(\Obs(\Scal))$.
	There are well-known algorithms
	for computing all strongly connected components of $\Acc(\Scal)$ in linear time,
	e.g., the slight variant of the depth-first search. Thirdly, observe that the condition holds
	if and only if in some strongly connected component, there is an observable transition,
	because each cycle belongs to only one strongly 
	connected component. This can also be checked trivially in linear time.
\end{proof}

\begin{theorem}\label{thm5_Det_PN}
	\begin{enumerate}\item
			The weak approximate detectability of finite automata 
	can be verified in PSPACE.
	\item
	Deciding weak approximate detectability of deterministic finite automata
	whose events can be directly observed is PSPACE-hard.
	\end{enumerate}
\end{theorem}

\begin{proof}
	Consider a finite automaton ${\mathcal S}=(X,T,X_0,\to,\Sig,\ell)$
	and a partition $R=\{R_1,\dots,R_n\}$ of $X$.
	If $\mathcal{S}$ satisfies that $\LM^{\omega}(\mathcal{S})=\emptyset$,
	then it is naturally weakly approximately detectable with respect to $R$. By Proposition \ref{prop8_Det_PN},
	the property $\LM^{\omega}(\mathcal{S})=\emptyset$ can be verified in polynomial time.
	Otherwise, continue the following procedure.

	Construct a new automaton ${\mathcal S}'=(R,T,R_0,\to',\Sig,\ell)$ in polynomial time, where
	$R_0=\{R_i\in R|R_i\cap X_0\ne\emptyset,i\in[1,n]\}$, for all $r,r'\in R$ and $t\in T$,
	$(r,t,r')\in\to'$ if and only if there exist $x\in r$ and $x'\in r'$ such that 
	$(x,t,x')\in\to$. One directly sees that $\mathcal S$ is weakly approximately detectable 
	with respect to $R$ if and only if $\mathcal S'$ is weakly detectable. Hence the weak approximate
	detectability of finite automata can be verified in PSPACE, since the weak detectability of
	finite automata can be verified in PSPACE \cite{Zhang2017PSPACEHardnessWeakDetectabilityDES}.

	To prove the hardness result, we consider a deterministic ${\mathcal S}$ whose events can be directly
	observed and the partition $R=\{\{x\}|x\in X\}$. For such an automaton, it is weakly approximately
	detectable with respect to $R$ if and only if it is weakly detectable. By the PSPACE-hardness
	result of deciding weak detectability of deterministic finite automata
	whose events can be directly observed
	\cite[Theorem 4.2]{Zhang2017PSPACEHardnessWeakDetectabilityDES}, we conclude the PSPACE-hardness
	of weak approximate detectability for the same model.
\end{proof}

%
%
%
%

\begin{remark}
	The notion of weak approximate detectability can be extended from a finite partition
	of the set of reachable states to a finite cover of that set.
	Such an extension may have potential applications in supervisor reduction of supervisory control theory.
	In supervisory control theory, the optimal solution to the control problem associated with
	a DES is the supremal supervisor (the supremal controllable sublanguage), and it is important 
	to reduce the size of the supremal supervisor together with preserving some corresponding control
	actions
	\cite{Cai2016SupervisorLocalization,Su2004SupervisorReductionDES,Vaz1986SupevisorReductionDES},
	where the reduction is done based on a notion of control cover that is actually a cover of 
	the state set.
	Under this extension, it is not difficult to see that the extended weak approximate detectability
	of finite automata can also be verified in PSPACE by the powerset construction 
	used to verify weak detectability in \cite{Shu2011GDetectabilityDES},
	and it is undecidable to verify
	this notion for labeled Petri nets (from Theorem \ref{thm2_Det_PN}).
\end{remark}

\section{Instant strong detectability and eventual strong detectability}\label{sec:StrongDetectabilityLPN}

\subsection{Instant strong detectability}

The concept of instant strong detectability is formulated as follows. It implies that
each prefix of each infinite label sequence generated by an LSTS allows reconstructing the current state.

\begin{definition}[ISD]\label{def3_Det_PN}
	Consider an LSTS $\Scal=(X,T,X_0,\to,\Sig,\ell)$.
	System $\cal S$ is called {\it instantly strongly detectable} if for each prefix $\s$
	of each infinite label sequence $\s'$ of $\LM^{\omega}({\cal S})$, $|\Mt({\cal S},\s)|=1$.
\end{definition}

Note that instant strong detectability is a weaker form of determinism \cite{Jancar1994BisimilarityPetriNet}.
In fact determinism implies that the condition $|\Mt(\mathcal S, \sigma')| = 1$ holds on all finite label
sequences $\sigma'$ generated by $\mathcal S$,
while the definition of instant strong detectability only requires that condition to
hold on the finite prefixes of infinite label sequences generated by $\mathcal S$.

It is trivial to see that instant strong detectability is strictly weaker than determinism.
Consider labeled Petri net $G=(N,M_0,\Sig,\ell)$ (Fig. \ref{fig1_Det_PN}),
where $N=(\{p_1,p_2,p_3\},\{t_1,t_2,t_3\},Pre,Post)$, $Pre$ and $Post$ are shown in 
Fig. \ref{fig1_Det_PN}, $M_0=(1,0,0)$ (in the order $(p_1,p_2,p_3)$),
$\Sig=\{a,b\}$, $\ell(t_1)=a$, $\ell(t_2)=\ell(t_3)=b$.
The language and $\omega$-language generated $G$ are $\LM(G)=a^*+a^*b=\{a^n|n\in\N\}\cup\{a^nb|n\in\N\}$
and $\LM^{\omega}(G)=\{a^{\omega}\}$, respectively. For all $n\in\N$,
$\Mt(G,a^n)=\{(1,0,0)\}$, $\Mt(G,a^nb)=\{(0,1,0),(0,0,1)\}$.
By definition the net is instantly strongly detectable but does not satisfy determinism.
Since $G$ has only finitely many markings, it is also a deterministic finite automaton. The automaton 
is also instantly strongly detectable but does not satisfy determinism.
		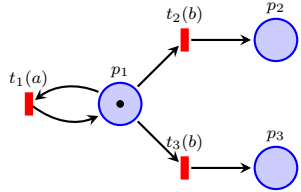
\begin{figure}[htbp]
		\tikzset{global scale/.style={
    scale=#1,
    every node/.append style={scale=#1}}}
		\begin{center}
			\begin{tikzpicture}[global scale = 1.0,
				>=stealth',shorten >=1pt,thick,auto,node distance=1.5 cm, scale = 0.8, transform shape,
	->,>=stealth,inner sep=2pt,
				every transition/.style={draw=red,fill=red,minimum width=1mm,minimum height=3.5mm},
				every place/.style={draw=blue,fill=blue!20,minimum size=7mm}]
				\tikzstyle{emptynode}=[inner sep=0,outer sep=0]
				\node[place,tokens=1,label=above:$p_1$] (p1) {};
				\node[transition,label=above:$t_2(b)$,above right of =p1] (t2) {}
				edge [pre] node[above, sloped] {} (p1);
				\node[transition,label=above:$t_1(a)$,left of =p1] {}
				edge [pre, bend left] node[above, sloped] {} (p1)
				edge [post, bend right] node[below, sloped] {} (p1);
				\node[place,label=above:$p_2$,right of = t2] (p2) {}
				edge [pre] (t2);
				\node[transition,label=above:$t_3(b)$,below right of =p1] (t3) {}
				edge [pre] (p1);
				\node[place,label=above:$p_3$,right of = t3] (p3) {}
				edge [pre] (t3);
			\end{tikzpicture}
			\caption{A labeled Petri net that is instantly strongly detectable but does not 
			satisfy determinism, where each arc is with weight $1$.}
			\label{fig1_Det_PN}
		\end{center}
	\end{figure}



\subsubsection{Finite automata}
\label{subsec:ISDofFA}

Consider a finite automaton $\mathcal{S}$, we next construct
its concurrent composition $\CCa(\mathcal{S})$.
Using $\CCa(\Scal)$ we will verify different notions of strong detectability for $\Scal$.
The proposed method applies to all finite automata even to those that do not satisfy
Assumption \ref{assum1_Det_PN}.

Consider a finite automaton ${\mathcal S}=(X,T,X_0,\to,\Sig,\ell)$.
We construct its concurrent composition 
\begin{equation}\label{eqn48_Det_PN}
	\CCa(\Scal)=(X',T',X_0',\to')
\end{equation} as follows:
\begin{enumerate}
	\item $X'=X\times X$;
	\item $T'=T_o'\cup T_{\epsilon}'$, where $T_o'=\{(\breve{t},\breve{t}')|\breve{t},\breve{t}'\in T,
		\ell(\breve{t})=\ell(\breve{t}')\in\Sig\}$,
		$T_{\epsilon}'=\{(\breve{t},\epsilon)|\breve{t}\in T,\ell(\breve{t})=\epsilon\}\cup
		\{(\epsilon,\breve{t})|\breve{t}\in T,\ell(\breve{t})=\epsilon\}$;
	\item $X_0'=X_0\times X_0$;
	\item for all $(\breve{x}_1,\breve{x}_1'),(\breve{x}_2,\breve{x}_2')\in X'$, $(\breve{t},\breve{t}')
		\in T_o'$, $(\breve{t}'',\epsilon)\in T_{\epsilon}'$,
		and $(\epsilon,\breve{t}''')\in T_{\epsilon}'$,
		\begin{itemize}
			\item $((\breve{x}_1,\breve{x}_1'),(\breve{t},\breve{t}'),(\breve{x}_2,\breve{x}_2'))\in\to'$ 
				if and only if $(\breve{x}_1,\breve{t},\breve{x}_2),(\breve{x}_1',\breve{t}',\breve{x}_2')\in\to$,
			\item $((\breve{x}_1,\breve{x}_1'),(\breve{t}'',\epsilon),(\breve{x}_2,\breve{x}_2'))\in\to'$ 
				if and only if $(\breve{x}_1,\breve{t}'',\breve{x}_2)\in\to$, $\breve{x}_1'=\breve{x}_2'$,
			\item $((\breve{x}_1,\breve{x}_1'),(\epsilon,\breve{t}'''),(\breve{x}_2,\breve{x}_2'))\in\to'$ 
				if and only if $\breve{x}_1=\breve{x}_2$, $(\breve{x}_1',\breve{t}''',\breve{x}_2')\in\to$.
		\end{itemize}
	\end{enumerate}

For an event sequence $s'\in (T')^{*}$, we use $s'(L)$ and $s'(R)$ to denote its left and right
components, respectively. Similar notation is applied to states of $X'$. 
In addition, for every $s'\in(T')^{*}$, we use $\ell(s')$
to denote $\ell(s'(L))$ or $\ell(s'(R))$, since $\ell(s'(L))=\ell(s'(R))$. 
In the above construction, $\mathcal{S}'$ aggregates every pair of transition sequences of 
$\mathcal{S}$ producing the same label sequence. In addition, $\Scal'$ has at most
$|X|^2$ states and at most $|X|^2(2|T_{\ep}||X|+\sum_{\sigma\in\Sig}|\ell^{-1}(\sigma)|^2
|X|^2)$ transitions, where the number does not exceed $|X|^2(2|T_{\ep}||X|+|T_o|^2|X|^2)$.
Hence it takes time $O(2|X|^3|T_\ep|+|X|^4\sum_{\sigma\in\Sig}|\ell^{-1}(\sigma)|^2)$ to construct $\CCa(\Scal)$.
For the special case when all observable events can be directly observed studied in \cite{Shu2011GDetectabilityDES},
the complexity reduces to $O(2|X|^3|T_\ep|+|X|^4|T_o|)$.
See the following example.

\begin{example}\label{exam3_Det_PN}
	A finite automaton ${\mathcal S}$ and its concurrent composition $\CCa(\Scal)$ are
	shown in Fig. \ref{fig19_Det_PN}.
			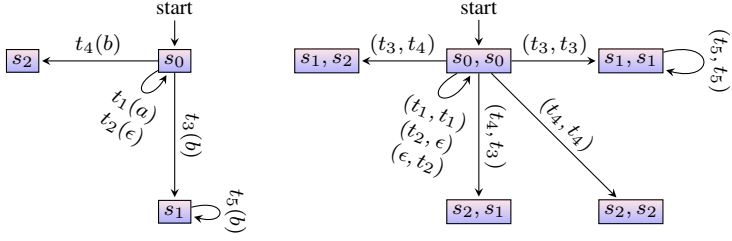
\begin{figure}
        \centering
\begin{tikzpicture}[>=stealth',shorten >=1pt,auto,node distance=2.0 cm, scale = 1.0, transform shape,
	->,>=stealth,inner sep=2pt,state/.style={shape=rectangle,draw,top color=red!10,bottom color=blue!30}]

	\node[initial, initial where =above, state] (s0) {$s_0$};
	\node[state] (s1) [below of =s0] {$s_1$};
	\node[state] (s2) [left of =s0] {$s_2$};
	
	\path [->]
	(s0) edge [out = 210, in = 240, loop] node [below, sloped] {$\begin{matrix}t_1(a)\\t_2(\epsilon)\end{matrix}$} (s0)
	(s0) edge node [above, sloped] {$t_3(b)$} (s1)
	(s0) edge node [above, sloped] {$t_4(b)$} (s2)
	(s1) edge [loop right] node [above, sloped] {$t_5(b)$} (s1)
	;

	\node[state] (s1s2) [right of =s0] {$s_1,s_2$};
	\node[state,initial,initial where=above] (s0s0) [right of =s1s2] {$s_0,s_0$};
	\node[state] (s2s1) [below of =s0s0] {$s_2,s_1$};
	\node[state] (s1s1) [right of =s0s0] {$s_1,s_1$};
	\node[state] (s2s2) [below of =s1s1] {$s_2,s_2$};

	\path [->]
	(s0s0) edge [out = 210, in = 240, loop] node [below, sloped] {$\begin{matrix}(t_1,t_1)\\(t_2,\epsilon)\\(\epsilon,t_2)\end{matrix}$} (s0s0)
	(s0s0) edge node [above, sloped] {$(t_3,t_4)$} (s1s2)
	(s0s0) edge node [above, sloped] {$(t_3,t_3)$} (s1s1)
	(s0s0) edge node [above, sloped] {$(t_4,t_3)$} (s2s1)
	(s0s0) edge node [above, sloped] {$(t_4,t_4)$} (s2s2)
	(s1s1) edge [loop right] node [above, sloped] {$(t_5,t_5)$} (s1s1)
	;

        \end{tikzpicture}
		\caption{A finite automaton (left) and its concurrent composition (right, only 
		the accessible part illustrated).}
		\label{fig19_Det_PN}
	\end{figure}
\end{example}

Consider a finite automaton $\Scal=(X,T,X_0,\to,\Sig,\ell)$.
In order to verify different notions of detectability for finite automata, we also need
to construct a {\it bifurcation automaton}
\begin{equation}\label{eqn71_Det_PN}
	\mathcal \Bifur(\Scal)=(X,\{\bar\epsilon,\check\epsilon\},X_0,\to',\{\bar\ep,\check\epsilon\},\ell')
\end{equation}
in linear time of the size of $\Scal$,
where $\to'\subset X\times\{\bar\epsilon,\check\epsilon\}\times X$, $\ell'(\bar\epsilon)=\bar\epsilon$,
$\ell'(\check\epsilon)=\check\epsilon$, $\ell'$ is also naturally
extended to $\ell':\{\bar\epsilon,\check\epsilon\}^*\cup\{\bar\epsilon,\check\epsilon\}^{\omega}\to
\{\bar\epsilon,\check\epsilon\}^*\cup\{\bar\epsilon,\check\epsilon\}^{\omega}$,
transitions $x\xrightarrow[]{\bar\ep}x'$ are called {\it fair transitions},
transitions $x\xrightarrow[]{\check\ep}x'$ are called {\it bifurcation transitions},
for every two states $i,j\in X$,
(1) $(j,\bar\ep,i),(j,\check\ep,i)\notin\to'$ if $\neg A_1$,
(2) $(x,\bar\ep,x')\in\to'$ if $A_1\wedge A_{2}\wedge A_{3}$,	
(3) $(x,\check\ep,x')\in\to'$ otherwise, where
\begin{align*}
	A_1 =& (\exists t\in T)[(j,t,i)\in\to],\\
	A_{2} =& (\nexists t\in T,j'\in X)[( (j,t,j')\in\to)\wedge(\ell(t)=\epsilon)
			 \wedge(j'\ne j)],\\
	A_{3} =& (\forall t\in T)[(( (j,t,i)\in\to)\wedge(\ell(t)\ne\epsilon))\implies\\
	      & (\nexists t'\in T,j'\in X)[( (j,t',j')\in\to)\wedge(\ell(t')=\ell(t))
		  \wedge(j'\ne i)]].
\end{align*}

Ones sees that both fair transitions and bifurcations transitions can be $\ep$-transitions
or observable transitions.
Next we explain the relation between $\Bifur(\Scal)$, the original automaton $\Scal$,
and the concurrent composition $\CCa(\Scal)$.
Here (1) holds if there is no transition from state $j$ to state $i$ in 
$\Scal$; (2) holds if there exists a transition from $j$ to $i$, and none of such transitions 
has a bifurcation in $\Scal$;
and (3) holds if there is a transition from $j$ to $i$ that has a bifurcation also in $\Scal$.
For the case that (3) holds, if $A_1$ holds but $A_2$ does not hold,
then for $\Scal$ one has $\{j\}\subsetneq\Mt(\{j\},\epsilon)$ and hence $|\Mt(\{j\},\epsilon)|>1$,
for $\CCa(\Scal)$ there is a transition $(j,j)\xrightarrow[]{(\ep,\tilde t)}(j,i')$
with $\ell(\tilde t)=\ep$ and $i'\ne j$;
if $A_1$ and $A_2$ hold but $A_3$ does not hold,
then for $\Scal$ one has $|\Mt(\{j\},\epsilon)|=1$, $\{i\}\subsetneq\Mt(\{j\},\ell(\tilde t'))$,
and hence $|\Mt(\{j\},\ell(\tilde t'))|>1$ for some $\tilde t'\in T$
with $\ell(\tilde t')\ne\epsilon$ and $(j,\tilde t',i)\in\to$; for $\CCa(\Scal)$ there is a
transition $(j,j)\xrightarrow[]{(\tilde t',\tilde t')}(i,i')$ with $i'\ne i$ for the above $\tilde t'$.

One also has that for all states $x$ and $x'$, there is a transition from $x$ to $x'$
in $\Scal$ if and only if there is a transition from $x$ to $x'$ in $\Obs(\Scal)$
if and only if there is a transition from $x$ to $x'$ in $\Bifur(\Scal)$. This obvious
observation is helpful in verify different notions of detectability for finite automata.

\begin{theorem}\label{thm6_Det_PN}
	The instant strong detectability of finite automata can be verified in linear time.	
\end{theorem}

\begin{proof}
	Consider a finite automaton $\Scal=(X,T,X_0,\to,\Sig,\ell)$ and its bifurcation 
	automaton $\Bifur(\Scal)$ defined by \eqref{eqn71_Det_PN}. If $\LM^{\omega}(\Scal)=\emptyset$,
	then $\Scal$ is naturally instantly strongly detectable. By the proof of Proposition 
	\ref{prop8_Det_PN}, it takes linear time of the size of $\Scal$ to check whether
	$\LM^{\omega}(\Scal)=\emptyset$.
	Next we assume that $\LM^{\omega}(\Scal)\ne\emptyset$. If additionally $|X_0|>1$,
	then by definition $\Scal$ is not instantly strongly detectable either. Next we additionally assume
	that there is a unique initial state.

	We claim that $\Scal$ is not instantly strongly detectable if 
	and only if in $\Scal$, there is a transition sequence 
	\begin{equation}\label{eqn72_Det_PN}
		x_0\xrightarrow[]{s_1}x_1\xrightarrow[]{t}x_2\xrightarrow[]{s_2}x_3\xrightarrow[]{s_3}x_3
	\end{equation}
	with $x_0\in X_0$,
	$x_1,x_2,x_3\in X$, $s_1,s_2,s_3\in T^*$, $t\in T$ such that $\ell(s_3)\in\Sig^+$
	and there is a bifurcation transition $x_1\xrightarrow[]{\check\ep}x_2$ in $\Acc(\Bifur(\Scal))$.

	``if'': This holds since the cycle $x_3\xrightarrow[]{s_3}x_3$ with positive-length label
	sequence can be extended to an infinite-length 
	transition sequence with infinite-length label sequence, and whether $|\Mt(\{x_0\},\ell(s_1))|>1$
	or $|\Mt(\{x_0\},\ell(s_1t'))|>1$ for some $t'\in T$ such that $\ell(t')\ne\ep$ and $(x_1,t',x_2)\in\to$
	by the notion of bifurcation automaton.

	``only if'': If $\Scal$ is not instantly strongly detectable, then there is an infinite transition
	sequence $x_0\xrightarrow[]{s_1}\bar x\xrightarrow[]{s_2}$ and a finite transition sequence
	$x_0\xrightarrow[]{s_1'}\bar x'$ such that $x_0\in X_0$, $\bar x,\bar x'\in X$, $\bar x
	\ne \bar x'$,
	$s_1,s_1'\in T^{*}$, $\ell(s_1)=\ell(s_1')$, $s_2\in T^{\omega}$, and $\ell(s_2)\in\Sig^{\omega}$.
	Then $s_1,s_1'\in T^+$ since at least one of $\bar x$ and $\bar x'$ differs from $x_0$.
	Moreover, $|\Mt(\{x_0\},\ell(s_1))|>1$.
	By the finiteness of $X$ and $\ell(s_2)\in\Sig^{\omega}$,
	in $\Scal$ there is a cycle with positive-length label sequence reachable from $\bar x$.

	We next check the above equivalent condition for instant strong detectability under the above
	two assumptions without loss of generality. See Fig. \ref{fig23_Det_PN} for a sketch.
	\begin{enumerate}
		\item Construct the accessible part $\Acc(\Obs(\Scal))$ of the observation automaton
			$\Obs(\Scal)$ of $\Scal$ defined by \eqref{eqn70_Det_PN}.
		\item Compute the set $X_c$ of all states of $\Acc(\Scal)$
			that belong to a cycle of $\Acc(\Scal)$ with positive-length label sequence.
			(Then we have $X_c\ne\emptyset$
			by the proof of Proposition \ref{prop8_Det_PN} since previously 
			we assume that $\LM^{\omega}(\Scal)\ne\emptyset$.)
		\item Compute $\Acc(\Bifur(\Scal))$. 
		\item Check whether there is a bifurcation transition $x_1\xrightarrow[]{\check\ep}x_2$ in $\Acc(\Bifur(\Scal))$ such that
			$X_c$ is reachable from $x_2$.
	\end{enumerate}
	
	The first step and the third step both take linear time of $\Scal$.

	For the second step, we firstly compute all strongly connected components of $\Acc(\Obs(\Scal))$
	in linear time of $\Scal$.
	Observe that for each strongly connected component, if it contains a transition,
	then it contains a cycle containing all its
	states and transitions. One then has that the set $X_c$ consists of all states
	of all strongly connected components of $\Acc(\Obs(\Scal))$, where each of these components
	has at least one observable transition.
	Hence $X_c$ can be computed in linear time.
	
	Recall that $\Acc(\Bifur(\Scal))$ and
	$\Acc(\Obs(\Scal))$ have the same set of states, and for every two states $x$ and $x'$,
	there is a transition from $x$ to $x'$ in $\Acc(\Bifur(\Scal))$ if and only if there is a
	transition also from $x$ to $x'$ in $\Acc(\Bifur(\Scal))$. Then the fourth step consumes
	linear time of $\Scal$ by traversing from $X_c$ all paths along the inverse direction
	of transitions. The bifurcation transition $x_1\xrightarrow[]{\check\ep}x_2$ in the fourth
	step exists if and only if transition sequence \eqref{eqn72_Det_PN} exists.

			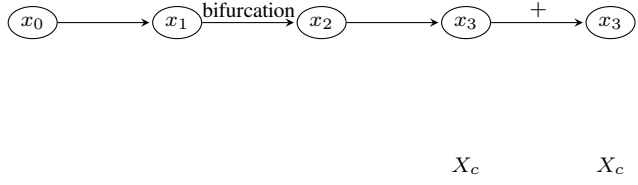
\begin{figure}
        \centering
\begin{tikzpicture}[>=stealth',shorten >=1pt,auto,node distance=1.9 cm, scale = 1.0, transform shape,
	->,>=stealth,inner sep=2pt,state/.style={shape=circle,draw,top color=red!10,bottom color=blue!30}]

	\node[elliptic state] (x0) {$x_0$};
	\node[elliptic state] (x1) [right of = x0] {$x_1$}; 
	\node[elliptic state] (x2) [right of = x1] {$x_2$};
	\node[elliptic state] (x3) [right of = x2] {$x_3$};
	\node[elliptic state] (x3') [right of = x3] {$x_3$};

	\node(empty) (x3_) [below of = x3] {$X_{c}$};
	\node(empty) (x3'_) [below of = x3'] {$X_{c}$};

	\path [->]
	(x0) edge (x1)
	(x1) edge node {bifurcation} (x2)
	(x2) edge (x3)
	(x3) edge node {$+$} (x3')
	;
        \end{tikzpicture}
		\caption{A sketch for verifying instant strong detectability of finite automata.}
		\label{fig23_Det_PN}
	\end{figure}

\end{proof}

\begin{example}\label{exam4_Det_PN}
	Reconsider the finite automaton $\Scal$ in Example \ref{exam3_Det_PN} (in the left part of Fig. 
	\ref{fig19_Det_PN}). Its observation automaton and bifurcation automaton are seen in Fig.
	\ref{fig24_Det_PN}. It has a unique initial state and generates a nonempty $\omega$-language.
	In addition, all its states are reachable. 
	According to the proof of Theorem \ref{thm6_Det_PN}, 
	one then has
	$X_c=\{s_0,s_1\}$, and in its bifurcation automaton there is a transition $s_0\xrightarrow
	[]{\check\ep}s_1$ such that $s_1$ in $X_c$ is reachable from $s_1$ in the transition. 
	Then $\Scal$ is not instantly strongly detectable.

	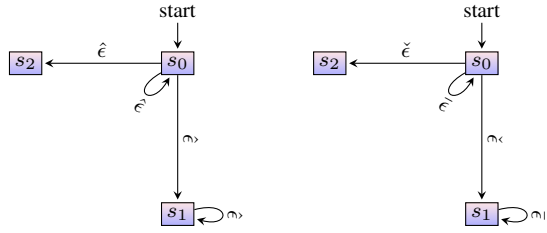
\begin{figure}
        \centering
\begin{tikzpicture}[>=stealth',shorten >=1pt,auto,node distance=2.0 cm, scale = 1.0, transform shape,
	->,>=stealth,inner sep=2pt,state/.style={shape=rectangle,draw,top color=red!10,bottom color=blue!30}]

	\node[initial, initial where =above, state] (s0) {$s_0$};
	\node[state] (s1) [below of =s0] {$s_1$};
	\node[state] (s2) [left of =s0] {$s_2$};
	
	\path [->]
	(s0) edge [out = 210, in = 240, loop] node [below, sloped] {$\hat\ep$} (s0)
	(s0) edge node [above, sloped] {$\hat\ep$} (s1)
	(s0) edge node [above, sloped] {$\hat\ep$} (s2)
	(s1) edge [loop right] node [above, sloped] {$\hat\ep$} (s1)
	;

	\node[state] (s2') [right of =s0] {$s_2$};
	\node[initial, initial where =above, state, right of = s2'] (s0') {$s_0$};
	\node[state] (s1') [below of =s0'] {$s_1$};

	\path [->]
	(s0') edge [out = 210, in = 240, loop] node [below, sloped] {$\bar\ep$} (s0')
	(s0') edge node [above, sloped] {$\check\ep$} (s1')
	(s0') edge node [above, sloped] {$\check\ep$} (s2')
	(s1') edge [loop right] node [above, sloped] {$\bar\ep$} (s1')
	;

        \end{tikzpicture}
		\caption{Observation automaton (left) and bifurcation automaton (right) of the automaton in
		the left part of Fig. \ref{fig19_Det_PN}.}
		\label{fig24_Det_PN}
	\end{figure}
\end{example}

\subsubsection{Labeled Petri nets}

In this subsection we discuss the decidability and complexity of instant strong detectability for labeled Petri nets.

If a labeled Petri net $G$ satisfies $\LM^{\omega}(G)=\emptyset$, then it is naturally
instantly strongly detectable. Actually whether the property $\LM^{\omega}(G)=\emptyset$ holds 
can be verified in 
EXPSPACE, and can also be guaranteed by the following Assumption \ref{assum2_Det_PN} that is weaker than
the widely used Assumption \ref{assum1_Det_PN} in detectability studies of DESs.

\begin{proposition}\label{prop7_Det_PN}
	Verifying whether a labeled Petri net $G$ satisfies $\LM^{\omega}(G)=\emptyset$ belongs to EXPSPACE.
\end{proposition}

\begin{proof}
	Consider a labeled Petri net $G=(N=(P,T,Pre,Post),M_0,\Sig,\ell)$.
	Observe that $\LM^{\omega}(G)\ne\emptyset$ if and only if there exists an infinite firing sequence
	\begin{equation}\label{eqn45_Det_PN}
		M_0[s_1\rangle M_1[s_2\rangle\cdots
	\end{equation}
	such that for each $i\in\Z_{+}$, $\ell(s_i)\in\Sig^+$.

	For $G$, a sequence \eqref{eqn45_Det_PN} exists if and only if $G$ satisfies the following Yen's 
	path formula 
	\begin{equation}\label{eqn46_Det_PN}
		(\exists \widetilde{M}_1,\widetilde{M}_2)(\exists \widetilde{s}_1,\widetilde{s}_2)
		[(M_0[\widetilde{s}_1\rangle \widetilde{M}_1[\widetilde{s}_2\rangle \widetilde{M}_2)\wedge(\widetilde{M}_2
		\ge \widetilde{M}_1)\wedge(\ell(\widetilde{s}_2)\in\Sig^+)].
	\end{equation}

	The ``if'' part follows from $\widetilde{M}_1[\widetilde{s}_2\rangle \widetilde{M}_2$ being a 
	repetitive firing sequence (hence can consecutively fire for infinitely many times) and 
	$|\ell(\widetilde{s}_2)|>0$.

	For the ``only if'' part: 
	Arbitrarily fix a sequence \eqref{eqn45_Det_PN}. 
	By Dickson's lemma, in the set $\{M_0,M_1,\dots\}$, there are totally finitely many distinct minimal 
	elements. Choose $k>0$ such that $\{M_0,\dots,M_k\}$ contains the maximal number of distinct minimal elements 
	of $\{M_0,M_1,\dots\}$, then there exist $0\le k'\le k< k''$ such that $M_{k'}\le M_{k''}$.
	Then the firing sequence $M_0[s_1\dots s_{k''}\rangle M_{k''}[s_{k'+1}\dots s_{k''}\rangle M'$ satisfies 
	$M_{k''}\le M'$ and $\ell(s_{k'+1}\dots s_{k''})\in\Sig^+$.

	The satisfiability of \eqref{eqn46_Det_PN} is actually a fair nondetermination
	problem and hence belongs to EXPSPACE \cite[Subsection 6.1]{Atig2009YenPathLogicPetriNet}.
\end{proof}

\begin{assumption}\label{assum2_Det_PN}
	\begin{enumerate}[(i)]
		\item\label{item5_Det_PN}
			A labeled P/T net $G$ does not {\it terminate}, i.e.,
			there exists an infinite firing sequence at the initial marking, and
		\item\label{item6_Det_PN}
			it is prompt, i.e., there exists no repetitive firing sequence labeled by the empty string.
	\end{enumerate}
\end{assumption}


Note that the deadlock-freeness assumption (see \eqref{item11_Det_PN}
of Assumption \ref{assum1_Det_PN}) implies \eqref{item5_Det_PN}
of Assumption \ref{assum2_Det_PN}, but not vice versa;
\eqref{item6_Det_PN} of Assumption \ref{assum2_Det_PN} is actually equivalent to \eqref{item12_Det_PN}
of Assumption \ref{assum1_Det_PN} for labeled Petri Petri nets.
Note also that for a labeled P/T net $G$,
$\LM^{\omega}(G)\ne\emptyset$ implies that $G$ does not terminate,
but not vice versa, because transitions could be labeled by $\epsilon$.
Verifying termination of Petri nets (the first part of Assumption \ref{assum2_Det_PN}) is
EXPSPACE-complete by the results of \cite{Rackoff19782EXPSPACECoverabilityPetriNets,Lipton1976ReachabilityPetriNetsEXPSPACE-hard}. 
Verifying promptness of labeled Petri nets belongs to EXPSPACE \cite{Atig2009YenPathLogicPetriNet}.
In addition, promptness is equivalent to all infinite firing sequences being labeled by infinite-length sequences.

In order to characterize instant strong detectability for labeled Petri nets,
we introduce the concurrent composition of a labeled Petri net.
	Given a labeled P/T net $G=(N=(P,T,Pre,Post),M_0,\Sig,\ell)$, we construct in polynomial time
	its concurrent composition as a Petri net
	\begin{equation}\label{eqn6_Det_PN}
		\CCn(G)=(N'=(P',T',Pre',Post'),M_0')
	\end{equation}
	which aggregates every pair of firing sequences of $G$ producing the same label sequence.
	Denote $P=\{\breve{p}_1,\dots,\breve{p}_{|P|}\}$ and $T=\{\breve{t}_1,\dots,\breve{t}_{|T|}\}$, duplicate them to
	$P_i=\{\breve{p}_1^i,\dots,\breve{p}_{|P|}^i\}$ and $T_i=\{\breve{t}_1^i,\dots,\breve{t}_{|T|}^i\}$, $i=1,2$,
	where we let $\ell(\breve{t}_i^1)=\ell(\breve{t}_i^2)=\ell(\breve{t}_i)$ for all $i$ in $[1,|T|]$.
	Then we specify $G'$
	as follows:
	\begin{enumerate}
		\item $P'=P_1\cup P_2$;
		\item $T'=T_o'\cup T_{\epsilon}'$, 
			where $T_o'=\{(\breve{t}_i^1,\breve{t}_j^2)\in T_1\times T_2|i,j\in[1,|T|],\ell(\breve{t}_i^1)=
			\ell(\breve{t}_j^2)\in\Sig\}$,
			$T_{\epsilon}'=\{(\breve{t}_1,\epsilon)|\breve{t}_1\in T_1,\ell(\breve{t}_1)=\epsilon\}
			\cup\{(\epsilon,\breve{t}_2)|\breve{t}_2\in T_2,\ell(\breve{t}_2)=\epsilon\}$;
		\item for all $k\in[1,2]$, all $l\in[1,|P|]$, and all $i,j\in[1,|T|]$ such that $\ell(\breve{t}_i^1)=
			\ell(\breve{t}_j^2)\in\Sig$,
			\begin{align*}
				Pre'(\breve{p}_l^k,(\breve{t}_i^1,\breve{t}_j^2)) &= \left\{
				\begin{array}[]{ll}
					Pre(\breve{p}_l^k,\breve{t}_i^1) &\text{if }k=1,\\
					Pre(\breve{p}_l^k,\breve{t}_j^2) &\text{if }k=2,
				\end{array}
				\right.\\
				Post'(\breve{p}_l^k,(\breve{t}_i^1,\breve{t}_j^2)) &= \left\{
				\begin{array}[]{ll}
					Post(\breve{p}_l^k,\breve{t}_i^1) &\text{if }k=1,\\
					Post(\breve{p}_l^k,\breve{t}_j^2) &\text{if }k=2;
				\end{array}
				\right.
			\end{align*}
		\item for all $l\in[1,|P|]$, all $i\in[1,|T|]$ such that $\ell(\breve{t}_i^1)=\ell(\breve{t}_i^2)=\epsilon$,
			\begin{align*}
				Pre'(\breve{p}_l^1,(\breve{t}_i^1,\epsilon)) &= Pre(\breve{p}_l^1,\breve{t}_i^1),\\
				Pre'(\breve{p}_l^2,(\epsilon,\breve{t}_i^2)) &= Pre(\breve{p}_l^2,\breve{t}_i^2),\\
				Post'(\breve{p}_l^1,(\breve{t}_i^1,\epsilon)) &= Post(\breve{p}_l^1,\breve{t}_i^1),\\
				Post'(\breve{p}_l^2,(\epsilon,\breve{t}_i^2)) &= Post(\breve{p}_l^2,\breve{t}_i^2);\\
			\end{align*}
		\item $M_0'(\breve{p}_l^k)=M_0(\breve{p}_l)$ for any $k$ in $[1,2]$ and any $l$ in $[1,|P|]$.
	\end{enumerate}

A labeled Petri net and its concurrent composition are shown in Fig. \ref{fig6_Det_PN}
and Fig. \ref{fig7_Det_PN}, respectively.

\begin{figure}
	\begin{center}
	\begin{minipage}[t]{0.29\linewidth}
		\begin{center}
    		\begin{tikzpicture}[
				>=stealth',shorten >=1pt,thick,auto,node distance=2.0 cm, scale = 0.8, transform shape,
	->,>=stealth,inner sep=2pt,
				every transition/.style={draw=red,fill=red,minimum width=1.0mm,minimum height=3.5mm},
				every place/.style={draw=blue,fill=blue!20,minimum size=7mm}]
				\tikzstyle{emptynode}=[inner sep=0,outer sep=0]
				\node[place, tokens=1,label=above:$p_1$] (p1) {};
				\node[transition, label=above:$b(b)$,below right of = p1] (t2){}
				edge[post] (p1);
				\node[transition, label=above:$a(\epsilon)$,above right of = p1] (t1) {}
				edge[pre] (p1);
				\node[place, below right of = t1,label=above:$p_2$] (p2) {}
				edge[pre] (t1)
				edge[post] (t2); 
			\end{tikzpicture}
		\caption{A labeled Petri net $G$, where event $a$ is unobservable, but $b$ can be directly observed.}
		\label{fig6_Det_PN}
      \end{center}
  \end{minipage}
  \begin{minipage}[t]{0.29\linewidth}
	  \begin{center}
        \begin{tikzpicture}[->,>=stealth,node distance=1.5cm
				>=stealth',shorten >=1pt,thick,auto,node distance=2.0 cm, scale = 0.8, transform shape,
	->,>=stealth,inner sep=2pt,
				every transition/.style={draw=red,fill=red,minimum width=1.0mm,minimum height=3.5mm},
				every place/.style={draw=blue,fill=blue!20,minimum size=7mm}]
				\tikzstyle{emptynode}=[inner sep=0,outer sep=0]
				\node[place, tokens=1,label=above:$p_1'$] (p11) {};
				\node[place, tokens=1,label=above:$p_1''$, below of = p11] (p12) {};
				\node[transition, label=below:{$(b,b)$},below right of = p12] (t2) {}
				edge[post] (p11)
				edge[post] (p12);
				\node[transition, label=above:{$(a,\epsilon)$},above right of = p11] (t1') {}
				edge[pre] (p11);
				\node[transition, label=above:{$(\epsilon,a)$},above right of = p12] (t1'') {}
				edge[pre] (p12);
				\node[place, below right of = t1',label=above:$p_2'$] (p21) {}
				edge[pre] (t1')
				edge[post] (t2);
				\node[place, label=above:{$p_2''$},below of = p21] (p22) {}
				edge[pre] (t1'')
				edge[post] (t2);
       \end{tikzpicture}
	   \caption{Concurrent composition of the net in Fig. \ref{fig6_Det_PN}.}
	   \label{fig7_Det_PN}
\end{center}
\end{minipage}
\begin{minipage}[t]{0.33\linewidth}
	  \begin{tikzpicture}[->,>=stealth,node distance=1.5cm
				>=stealth',shorten >=1pt,thick,auto,node distance=2.0 cm, scale = 0.8, transform shape,
	->,>=stealth,inner sep=2pt,
				every transition/.style={draw=red,fill=red,minimum width=1.0mm,minimum height=3.5mm},
				every place/.style={draw=blue,fill=blue!20,minimum size=7mm}]
				\tikzstyle{emptynode}=[inner sep=0,outer sep=0]
				\node[place, tokens=1,label=left:$p_1'$] (p11) {};
				\node[place, tokens=1,label=left:$p_1''$, below of = p11] (p12) {};
				\node[transition, label=below:{$(b,b)$},below right of = p12] (t2) {}
				edge[post] (p11)
				edge[post] (p12);
				\node[transition, label=above:{$(a,\epsilon)$},above right of = p11] (t1') {}
				edge[pre] (p11);
				\node[transition, label=above:{$(\epsilon,a)$},above right of = p12] (t1'') {}
				edge[pre] (p12);
				\node[place, below right of = t1',label=right:$p_2'$] (p21) {}
				edge[pre] (t1')
				edge[post] (t2);
				\node[place, label=right:{$p_2''$},below of = p21] (p22) {}
				edge[pre] (t1'')
				edge[post] (t2);
				\node[transition, label=left:{$(a,\phi)$},above left of = t1'] (aphi) {}
				edge[pre] (p11)
				edge[post, bend left] (p21);
				\node[transition, label=right:{$(b,\phi)$},above right of = t1'] (bphi) {}
				edge[post, bend right] (p11)
				edge[pre] (p21);
				\node[transition, label=left:{$(\phi,a)$},below left of = t2] (phia) {}
				edge[pre] (p12)
				edge[post, bend right] (p22);
				\node[transition, label=right:{$(\phi,b)$},below right of = t2] (phib) {}
				edge[post, bend left] (p12)
				edge[pre] (p22);
       \end{tikzpicture}
	   \caption{Extended concurrent composition of the net in Fig. \ref{fig6_Det_PN}.}
	   \label{fig14_Det_PN}
\end{minipage}
\end{center}
\end{figure}

	Assume that there exists a label sequence $\s\in\LM(G)$ such that $|\Mt(G,\s)|>1$, then
	there exist transitions $t_{\mu_1},\dots,t_{\mu_n},t_{\omega_1},\dots,t_{\omega_n}\in T\cup\{\epsilon\}$,
	where $n\ge 1$,
	such that $\ell(t_{\mu_i})=\ell(t_{\omega_i})$ for all $i\in[1,n]$,
	$\ell(t_{\mu_1}\dots t_{\mu_n})=\ell(t_{\omega_1}\dots t_{\omega_n})=\s$,
	$M_0[t_{\mu_1}\dots t_{\mu_n}\rangle M_1$ and $M_0[t_{\omega_1}\dots t_{\omega_n}\rangle M_2$
	for different $M_1$ and $M_2$ both in $\N^{P}$. Then for $\CCn(G)$, we have
	$M_0'[(t_{\mu_1}^1,t_{\omega_1}^2)\dots(t_{\mu_n}^1,t_{\omega_n}^2)\rangle M'$, where
	$M'(\breve{p}_l^k)=M_k(\breve{p}_l)$, $k\in[1,2]$, $l\in[1,|P|]$, and
	$M'(\breve{p}^1_{l'})\ne M'(\breve{p}^2_{l'})$ for some $l'\in[1,|P|]$
	(briefly denoted by $M'|_{P_1}\ne M'|_{P_2}$).

	Assume that for each label sequence $\s\in\LM(G)$, we have $|\Mt(G,\s)|=1$.
	Then for all $M'\in\Rt(N',M_0')$, $M'(\breve{p}^1_l)=M'(\breve{p}^2_l)$ for each $l$ in $[1,|P|]$
	(briefly denoted by $M'|_{P_1}=M'|_{P_2}$).

\begin{theorem}\label{thm3_Det_PN}
	\begin{enumerate}[(1)]
		\item\label{item9_Det_PN} It is decidable to verify if a labeled P/T net $G$ is instantly strongly detectable.
		\item\label{item10_Det_PN} It is EXPSPACE-hard to check if a labeled P/T net $G$ with $\LM^{\omega}(G)\ne\emptyset$
			is instantly strongly detectable.
	\end{enumerate}
\end{theorem}

\begin{proof}
	\eqref{item9_Det_PN} Proof of the decidability result: 

	By Proposition \ref{prop7_Det_PN}, we first verify whether $G$ satisfies $\LM^{\omega}(G)\ne
	\emptyset$ in EXPSPACE. If $\LM^{\omega}(G)=\emptyset$, then $G$ is instantly strongly detectable. Otherwise, 
	continue the following procedure.

	Next we reduce the non-instant strong detectability problem to the satisfiability of a Yen's path formula.
	Then by Proposition \ref{prop6_Det_PN}, the instant strong detectability of labeled Petri nets is
	decidable.

	It can be seen that a labeled Petri net $G=(N=(P,T,Pre,Post),M_0,\Sig,\ell)$ with $\LM^{\omega}(G)\ne\emptyset$ 
	is not instantly strongly
	detectable if and only if there is an infinite 
	label sequence $\s\in\LM^{\omega}(G)$ such that for some prefix $\s'\sqsubset\s$, one has
	$|\mathcal{M}(G,\s')|>1$. By this observation, we claim that $G$ is not instantly strongly detectable
	if and only if there exists a firing sequence
	\begin{equation}\label{eqn47_Det_PN}
		M_0[s_1\rangle M_1[s_2\rangle M_2[s_s\rangle M_3
	\end{equation}
	such that $|\mathcal{M}(G,\ell(s_1))|>1$, $M_2\le M_3$, and $\ell(s_3)\in\Sig^+$.

	The sufficiency follows from $M_2[s_3\rangle M_3$ is a repetitive firing sequence and can fire
	for infinitely many times, that is, the infinite firing sequence 
	$$M_0[s_1\rangle M_1[s_2\rangle M_2[s_3\rangle M_3[s_3\rangle\cdots[s_3\rangle\cdots$$ 
	satisfies $\ell(s_1s_2s_3s_3\dots)\in\LM^{\omega}(G)$.

	To prove the necessity, we assume that $G$ is not instantly strongly detectable and choose an arbitrary infinite
	firing sequence
	\begin{equation}\label{eqn41_Det_PN}
		M_0[\overline{t}_1\rangle \overline{M}_1[\overline{t}_2\rangle\cdots[\overline{t}_i\rangle 
		\overline{M}_i[\overline{t}_{i+1}\rangle\cdots
	\end{equation}
	satisfying $\overline{t}_i\in T$ for all $i\in\Z_{+}$,
	$\ell(\overline{t}_1\overline{t}_2\dots)\in\LM^{\omega}(G)$, and there exists $l\in\Z_{+}$ such that
	$|\mathcal{M}(G,\ell(\overline{t}_1\dots \overline{t}_l))|>1$.
	
	By Dickson's Lemma, in \eqref{eqn41_Det_PN}, there are totally finitely many distinct minimal
	markings. Choose an arbitrary number $k>l$ such that $\{M_0,\overline{M}_1\dots,\overline{M}_k\}$ 
	contains the largest number of
	distinct minimal markings of \eqref{eqn41_Det_PN}. Choose $k'>k$ such that at least one of 
	$\overline{t}_{k+1},\dots,\overline{t}_{k'}$
	is observable and $\overline{M}_{k'}\ge \overline{M}_{l'}$ for some 
	$0\le l'\le k$. Then $\overline{t}_{l'+1}\dots \overline{t}_{k'}$ is enabled at $\overline{M}_{k'}$.
	Consider the newly obtained firing sequence
	\begin{align}\label{eqn27_Det_PN}
		M_0[\overline{t}_1\dots \overline{t}_l\rangle \overline{M}_l[\overline{t}_{l+1}\dots \overline{t}_{k'}
		\rangle \overline{M}_{k'}[\overline{t}_{l'+1}\dots \overline{t}_{k'}\rangle \overline{M}'_{k'},
	\end{align}
	where $\overline{M}_{k'}\le \overline{M}_{k'}'$, $|\ell(\overline{t}_{l'+1}\dots \overline{t}_{k'})|>0$.
	Also by $|\mathcal{M}(G,\ell(\overline{t}_1\dots \overline{t}_l))|>1$,
	\eqref{eqn27_Det_PN} satisfies \eqref{eqn47_Det_PN}.

	Consider $G$ and its concurrent composition $\CCn(G)=(N'=(P',T',Pre',\\Post'),M_0')$ 
	shown in \eqref{eqn6_Det_PN}. Add a new set
	$$T_{\phi}=T_{\phi}^1\cup T_{\phi}^2$$ of transitions into $\CCn(G)$,
	where $\phi\notin T_1\cup T_2$,
	$T_{\phi}^1=\{(\breve{t}_1,\phi)|\breve{t}_1\in T_1\}$, $T_{\phi}^2=\{(\phi,\breve{t}_2)|\breve{t}_2\in T_2\}$.
	Add the following rules: 
	for all $l\in[1,|P|]$, all $i\in[1,|T|]$,
			\begin{align*}
				Pre'(\breve{p}_l^1,(\breve{t}_i^1,\phi)) &= Pre(\breve{p}_l^1,\breve{t}_i^1),\\
				Pre'(\breve{p}_l^2,(\phi,\breve{t}_i^2)) &= Pre(\breve{p}_l^2,\breve{t}_i^2),\\
				Post'(\breve{p}_l^1,(\breve{t}_i^1,\phi)) &= Post(\breve{p}_l^1,\breve{t}_i^1),\\
				Post'(\breve{p}_l^2,(\phi,\breve{t}_i^2)) &= Post(\breve{p}_l^2,\breve{t}_i^2).
			\end{align*}
	
	The newly obtained extended concurrent composition is denoted by
	\begin{equation}\label{eqn28_Det_PN}
		\CCne(G)=(N''=(P'',T'',Pre'',Post''),M_0''),
	\end{equation}
	where $P''=P'$, $T''=T'\cup T_{\phi}$, $M_0''=M_0'$.
	For example, the corresponding extended concurrent composition of the net in Fig. \ref{fig6_Det_PN} is
	shown in Fig. \ref{fig14_Det_PN}.

	By direct observation, one sees that there exists a firing sequence \eqref{eqn47_Det_PN} in $G$ 
	if and only if in the extended concurrent composition $\CCne(G)$, there is a firing sequence 
	\begin{equation}\label{eqn29_Det_PN}
		M''_0[s_1''\rangle M''_1[s_2''\rangle M''_2[s_3''\rangle M''_3
	\end{equation}
	such that either
	\begin{equation}\label{eqn31_Det_PN}\begin{split}
		&s_1''\in (T')^*; M_1''|_{P_1}\ne M_1''|_{P_2}; s_2'',s_3''\in(T'\cup T_{\phi}^1)^*;\\
		&\exists\text{ transition }(t_1,*)\text{ in }s_3''\text{ such that }\ell(t_1)\in\Sig;
		M_2''|_{P_1}\le M_3''|_{P_1};
	\end{split}\end{equation}
	or
	\begin{equation}\label{eqn32_Det_PN}\begin{split}
		&s_1''\in (T')^*; M_1''|_{P_1}\ne M_1''|_{P_2}; s_2'',s_3''\in(T'\cup T_{\phi}^2)^*;\\
		&\exists\text{ transition }(*,t_2)\text{ in }s_3''\text{ such that }\ell(t_2)\in\Sig;
		M_2''|_{P_2}\le M_3''|_{P_2}.
	\end{split}\end{equation}

	Consequently we have for net $G$, a firing sequence \eqref{eqn47_Det_PN} exists if and only if in $\CCne(G)$ either 
	\begin{equation}\label{eqn33_Det_PN}
		\text{there exists a firing sequence } \eqref{eqn29_Det_PN}\text{ satisfying } \eqref{eqn31_Det_PN}
	\end{equation}or
	\begin{equation}\label{eqn34_Det_PN}
		\text{there exists a firing sequence } \eqref{eqn29_Det_PN}\text{ satisfying } \eqref{eqn32_Det_PN}.
	\end{equation}

	Apparently, \eqref{eqn33_Det_PN} holds if and only if \eqref{eqn34_Det_PN} holds by symmetry of $\CCne(G)$.
	Hence we only need to consider \eqref{eqn33_Det_PN}.

	One directly sees that the necessity holds. One also sees that $M_0''[s_1''\rangle M_1''$ in \eqref{eqn29_Det_PN}
	is a firing sequence of $\CCn(G)$, the left component of $M_1''[s_2''s_3''\rangle M_3''$ is a firing 
	sequence of $G$ and the right component of $s_2''s_3''$ may contain several copies of $\phi$'s.

	For the sufficiency, if \eqref{eqn33_Det_PN} holds,
	then the left component of \eqref{eqn29_Det_PN} can be extended to an infinite firing sequence of 
	$G$, and its label sequence is of length $\infty$,
	because the left component of $M''_2[s_3''\rangle M''_3$ is a repetitive firing sequence of $G$ containing 
	an observable transition. In addition, when the label sequence of the left component of $s_1''$ 
	is observed (no matter whether $\ell(s_1'')=\epsilon$), net $G$ can reach at least two different markings,
	including $M_1''|_{P_1}$ and $M_1''|_{P_2}$. 

	Based on the above discussion, we have $G$ is not instantly strongly detectable if and only if 
	in $\CCne(G)$, \eqref{eqn33_Det_PN} holds.

	Now consider $\CCne(G)$ and whether \eqref{eqn33_Det_PN} is a Yen's path formula. 
	In \eqref{eqn31_Det_PN}, ``$s_1''\in (T')^*$'' and ``$s_2'',s_3''\in(T'\cup T_{\phi}^1)^*$''
	are transition predicates; 
	``$\exists$ transition $(t_1,*)$ in $s_3''$ such that $\ell(t_1)\in\Sig$'' is also a transition 
	predicate; ``$M_2''|_{P_1}\le M_3''|_{P_1}$'' can be expressed as 
	combination of marking predicates; only ``$M_1''|_{P_1}\ne M_1''|_{P_2}$'' is not a predicate.

	Next we reduce the satisfiability of \eqref{eqn33_Det_PN} to the satisfiability of a Yen's path formula
	of a new Petri net $\CCne(G)'$, completing the proof of the decidability result.

	Add two new places $p_0'''$ and $p_1'''$ into $\CCne(G)$, where initially $p_0'''$ contains exactly $1$ token,
	but $p_1'''$ contains no token; add one new transition $r_1'''$, and arcs $p_0'''\to r_1'''\to p_1'''$,
	both with weight $1$. Also, for each transition $t$ in $\CCne(G)$, add arcs $p_1'''\to t\to p_1'''$, both with
	weight $1$. Then we obtain a new Petri net $$\CCne(G)'=(N'''=(P''',T''',Pre''',Post'''),M_0''').$$
	We then have for $\CCne(G)$, \eqref{eqn33_Det_PN} holds if and only if $\CCne(G)'$ satisfies the Yen's path formula
	\begin{equation}\label{eqn26_Det_PN}
		\begin{split}
		&(\exists M_1''',M_2''',M_3''',M_4''')(\exists s_1''',s_2''',s_3''',s_4''')[\\
		&(M_0'''[s_1'''\rangle M_1'''[s_2'''\rangle M_2'''[s_3'''\rangle M_3'''[s_4'''\rangle M_4''')\wedge\\
		&(s_1'''=r_1''')\wedge(s_2'''\in (T')^*)\wedge((M_2'''-M_1''')|_{P_1}\ne(M_2'''-M_1''')|_{P_2})\wedge\\
		&(s_3''',s_4'''\in(T'\cup T_{\phi}^1)^*)\wedge
		(\exists\text{ transition }(t_1,*)\text{ in }s_4'''\text{ such that }\ell(t_1)\in\Sig)\wedge\\
		&\left.(M_3'''|_{P_1}\le M_4'''|_{P_1})\right].
		\end{split}
	\end{equation}

	\eqref{item10_Det_PN} Proof of the hardness result:

	Next we prove the hardness result by reducing the coverability problem to the non-instant strong detectability
	problem in polynomial time.

	We are given a Petri net $G=(N=(P,T,Pre,Post),M_0)$ and a destination marking $M\in\N^{P}$, and construct
	a labeled P/T net
	\begin{equation}\label{eqn9_Det_PN}
		G'=(N'=(P',T',Pre',Post'),M_0',T\cup\{\s_{G}\},\ell)
	\end{equation}as follows
	(see Fig. \ref{fig4_Det_PN} as a sketch):
	\begin{enumerate}
		\item Add three places $p_0,p_1,p_2$, where initially $p_0$ contains exactly one token, but $p_1$
			and $p_2$ contains no token;
		\item add three transitions $t_0,t_1,t_2$, and arcs $p_0\to t_0\to p_0$, $t_1\to p_1$, $t_2\to p_2$,
			all with weight $1$; for every $p\in P$, add arcs $p\to t_1$ and $p\to t_2$, both with weight 
			$M(p)$;
		\item add label $\s_G\notin T\cup\{t_0,t_1,t_2\}$, $\ell(t)=t$ for each $t\in T\cup\{t_0\}$, $\ell(t)=\s_G$ for each $t\in \{t_1,t_2\}$.
	\end{enumerate}

	It is clear that if $M$ is not covered by $G$ then $G'$ shown in \eqref{eqn9_Det_PN} is
	instantly strongly detectable. If $M$ is covered by $G$, then there exists a firing sequence
	$M_0[\s_1\rangle M_1$ with $M_1\ge M$. Furthermore, there exist two infinite firing sequences
	\begin{align*}
		&M_0'[\s_1\rangle M_1'[t_1\rangle M_2'[t_0\rangle M_2'[t_0\rangle \cdots,\\
		&M_0'[\s_1\rangle M_1'[t_2\rangle M_2''[t_0\rangle M_2''[t_0\rangle\cdots,
	\end{align*}
	where $M_2'\ne M_2''$ since $M_2'(p_1)>0$, $M_2'(p_2)=0$, $M_2''(p_2)>0$, $M_2''(p_1)=0$;
	in both sequences, after $t_1$, all firing transitions are $t_0$. Also by $\ell(t_1)=\ell(t_2)$,
	we have $G'$ is not instantly strongly detectable.
	This reduction runs in time linear of
	the number of places of $G$ and the number of tokens of the destination marking $M$.
	Since the coverability problem is EXPSPACE-hard in the number of transitions of $G$,
	deciding non-instant strong detectability is EXPSPACE-hard in the numbers of places and transitions
	of $G'$ and the number of tokens of $M$, hence deciding instant strong detectability is also
	EXPSPACE-hard, which completes the proof.

		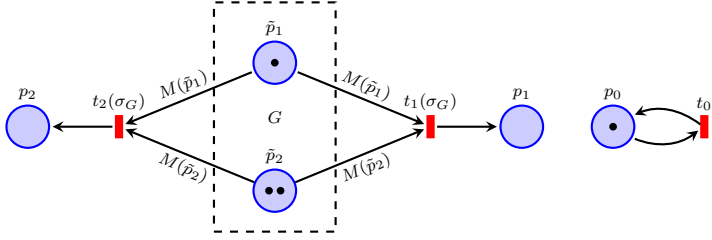
\begin{figure}[htbp]
		\tikzset{global scale/.style={
    scale=#1,
    every node/.append style={scale=#1}}}
		\begin{center}
			\begin{tikzpicture}[global scale = 1.0,
				>=stealth',shorten >=1pt,thick,auto,node distance=1.5 cm, scale = 0.8, transform shape,
	->,>=stealth,inner sep=2pt,
				every transition/.style={draw=red,fill=red,minimum width=1mm,minimum height=3.5mm},
				every place/.style={draw=blue,fill=blue!20,minimum size=7mm}]
				\tikzstyle{emptynode}=[inner sep=0,outer sep=0]
				\node[place,label=above:$\tilde p_1$,tokens=1] (p1t) {};
				\node[emptynode,below right of = p1t] (empty1) {};
				\node[transition,label=above:$t_1(\sigma_G)$,right of =empty1] (t1) {}
				edge[pre] node[above, sloped] {$M(\tilde p_1)$} (p1t);
				\node[emptynode,left of = p1t] (empty2) {};
				\node[transition,label=above:$t_2(\sigma_G)$,below left of =empty2] (t2) {}
				edge[pre] node[above, sloped] {$M(\tilde p_1)$} (p1t);
				\node[place,label=above:$\tilde p_2$,below left of =empty1,tokens=2] (p2t) {}
				edge [post] node[below, sloped] {$M(\tilde p_2)$} (t2)
				edge [post] node[below, sloped] {$M(\tilde p_2)$} (t1);
				\node[place,label=above:$p_1$,right of =t1] (p1) {}
				edge [pre] (t1);
				\node[place,label=above:$p_2$,left of =t2] (p2) {}
				edge [pre] (t2);
				\node[place,tokens=1,label=above:$p_0$,right of =p1] (p0) {};
				\node[transition,label=above:$t_0$,right of =p0] {}
				edge [pre, bend left] (p0)
				edge [post, bend right] (p0);

				\draw[dashed] (-1.0,-2.8) rectangle (1,1.0);
				\node at (0,-0.9) {$G$};

			\end{tikzpicture}
			\caption{Sketch for the reduction in the hardness proof of Theorem \ref{thm3_Det_PN}.}
			\label{fig4_Det_PN}
		\end{center}
	\end{figure}
	
\end{proof}

\begin{remark}
	By using the extended concurrent composition and a similar procedure as the proof of Theorem 
	\ref{thm3_Det_PN}, the decidability result for strong detectability
	of labeled Petri nets proved in \cite{Masopust2018DetectabilityPetriNet} can be strengthened
	to hold only based on the promptness assumption.
\end{remark}


\begin{remark}
	The concept of instant strong detectability of labeled Petri nets is a uniform concept.
	That is, a labeled Petri net is instantly strongly detectable if and only if it is
	instantly strongly detectable when its initial marking is replaced by each of its reachable markings.
	Formally, for a labeled Petri net $G=(N,M_0,\Sig,\ell)$, $G$ is instantly strongly detectable
	if and only if $G'=(N,M,\Sig,\ell)$ is instantly strongly detectable for each $M\in\Rt(N,M_0)$.
	The sufficiency naturally holds since $M_0\in\Rt(N,M_0)$. For the necessity, if there exists
	$M_1\in\Rt(N,M_{0})$ such that labeled Petri net $G_1=(N,M_1,\Sig,\ell)$ is not instantly strongly detectable,
	then there exist $\s_1\sqsubset\s_2\in\LM^{\omega}(G_1)$ satisfying $|\Mt(G_1,\s_1)|>1$.
	Since there exists
	$\s_0\in\LM(G)$ satisfying $M_1\in\Mt(G,\s_0)$, we have $\Mt(G,\s_0\s_1)\supset\Mt(G_1,\s_1)$ and
	$|\Mt(G,\s_0\s_1)|>1$, i.e., $G$ is not instantly strongly detectable. Hence if a labeled Petri net is
	instantly strongly detectable, in order to determine the current marking,
	one does not need to care about when the net started to run.
\end{remark}

\subsection{Eventual strong detectability}

The concepts of  strong detectability and eventual strong detectability are
given as follows. The former implies there exists a positive integer $k$ such that
for each infinite label sequence generated by a
system, each prefix of the label sequence of length greater than $k$ allows reconstructing
the current state.
The latter implies that for each infinite label sequence generated by a
system, there exists a positive integer $k$ (depending on the label sequence) such that
each prefix of the label sequence of length greater than $k$ allows doing that.
Hence the former is stronger than the latter.

\begin{definition}[SD]\label{def4_Det_PN}
	Consider an LSTS ${\mathcal S}=(X,T,X_0,\to,\Sig,\ell)$.
	System $\cal S$ is called {\it strongly detectable} if there exists a positive
	integer $k$ such that for each label sequence $\s\in \LM^{\omega}({\cal S})$,
	$|\Mt({\cal S},\s')|=1$ for every prefix $\s'$ of $\s$ satisfying $|\s'|>k$.
\end{definition}

\begin{definition}[ESD]\label{def5_Det_PN}
	Consider an LSTS ${\mathcal S}=(X,T,X_0,\to,\Sig,\ell)$.
	System ${\cal S}$ is called {\it eventually strongly detectable} if
	for each label sequence $\s\in \LM^{\omega}({\cal S})$,
	there exists a positive integer $k_{\s}$ such that
	$|\Mt({\cal S},\s')|=1$ for every prefix $\s'$ of $\s$ satisfying $|\s'|>k_{\s}$.
\end{definition}

By definition, strong detectability implies eventual strong detectability.
The following Proposition \ref{prop5_Det_PN} shows that they are not equivalent.

\begin{proposition}\label{prop5_Det_PN}
	Strong detectability strictly implies eventual strong detectability
	for labeled P/T nets and finite automata.
\end{proposition}

\begin{proof}
	Consider the labeled P/T net $G$ in Fig. \ref{fig5_Det_PN}, where $a$ and $b$
	are labels of transitions.
	It can be seen that $\LM^{\omega}(G)=a^{\omega}+a^*b^{\omega}+a^*ba^{\omega}:=
	\{a^{\omega}\}\cup\{a^nb^{\omega}|n\in\N\}\cup\{a^nba^{\omega}|n\in\N\}$.
	One also has that $\Mt(G,a^n)=\{(1,0,0)\}$, $\Mt(G,a^nb)=\{(0,1,0),(0,0,1)\}$,
	$\Mt(G,a^nbb^{m+1})=\{(0,1,0)\}$, $\Mt(G,a^nba^{m+1})=\{(0,0,1)\}$
	for all $m,n\in\N$. Hence $G$ is eventually strongly detectable,
	but not strongly detectable.

	The net can be regarded as a deterministic finite automaton satisfying Assumption 
	\ref{assum1_Det_PN} when $a$ and $b$ are
	regarded as labels of events.
	By a direct observation, it is also eventually strongly detectable, but not strongly detectable.
	
		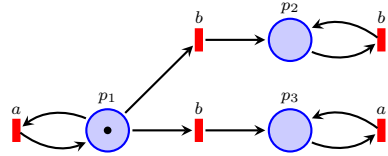
\begin{figure}[htbp]
		\tikzset{global scale/.style={
    scale=#1,
    every node/.append style={scale=#1}}}
		\begin{center}
			\begin{tikzpicture}[global scale = 1.0,
				>=stealth',shorten >=1pt,thick,auto,node distance=1.5 cm, scale = 0.8, transform shape,
	->,>=stealth,inner sep=2pt,
				every transition/.style={draw=red,fill=red,minimum width=1mm,minimum height=3.5mm},
				every place/.style={draw=blue,fill=blue!20,minimum size=7mm}]
				\tikzstyle{emptynode}=[inner sep=0,outer sep=0]
				\node[place, label=above:$p_1$,tokens=1] (p1) {};
				\node[transition, label=above:$a$,left of = p1] (t1) {}
				edge[pre, bend left] (p1)
				edge[post, bend right] (p1);
				\node[transition, label=above:$b$,right of = p1] (t3){}
				edge[pre] (p1);
				\node[transition, label=above:$b$,above of = t3] (t2) {}
				edge[pre] (p1);
				\node[place, label=above:$p_2$,right of = t2] (p2) {}
				edge[pre] (t2);
				\node[place, label=above:$p_3$,right of = t3] (p3) {}
				edge[pre] (t3);
				\node[transition, label=above:$b$,right of = p2] (t4) {}
				edge[pre, bend left] (p2)
				edge[post, bend right] (p2);
				\node[transition, label=above:$a$,right of = p3] (t5) {}
				edge[pre, bend left] (p3)
				edge[post, bend right] (p3);
			\end{tikzpicture}
			\caption{A labeled P/T net $G$ that is eventually strongly detectable, but not strongly detectable.}
			\label{fig5_Det_PN}
		\end{center}
	\end{figure}
\end{proof}

\subsubsection{Finite automata} 

We next use the concurrent composition, the observation automaton, and the bifurcation
automaton of a finite automaton $\Scal$
defined by \eqref{eqn48_Det_PN}, \eqref{eqn70_Det_PN}, and \eqref{eqn71_Det_PN}
to verify its strong detectability 
and eventual strong detectability without any assumption.
These results extend the related 
results given in \cite{Shu2007Detectability_DES,Shu2011GDetectabilityDES}, since the verification methods 
for strong detectability
in these papers generally do not apply to finite automata that do not satisfy Assumption
\ref{assum1_Det_PN}.

\begin{theorem}\label{thm7_Det_PN}
	The strong detectability of finite automata 
	can be verified in polynomial time.
\end{theorem}

\begin{proof}
	Consider a finite automaton ${\mathcal S}=(X,T,X_0,\to,\Sig,\ell)$ and another finite
	automaton $\Acc(\CCa(\Acc(\Scal)))=(X',T',X_0',\to')$ that is the accessible part of the concurrent
	composition of $\Acc(\Scal)$.
	We claim that $\Scal$ is not strongly detectable if and only if in $\Acc(\CCa(\Acc(\Scal)))$, 
	\begin{subequations}\label{eqn55_Det_PN}
		\begin{align}
			&\text{there exists a transition sequence }\nonumber\\
			&x_0'\xrightarrow[]{s_1'}x_1'\xrightarrow[]{s_2'}x_1'\xrightarrow[]{s_3'}
			x_2'
			\text{ satisfying}\label{eqn55_1_Det_PN}\\
			&x_0'\in X_0';x_1',x_2'\in X'; s_1',s_2',s_3'\in(T')^*;\ell(s_2')\in\Sig^+;x_2'(L)\ne x_2'(R);\\
			&\text{and in }\Scal,\text{ there exists a cycle with nonempty label sequence }
			\text{reachable from }x_2'(L).
		\end{align}
	\end{subequations}

	If \eqref{eqn55_Det_PN} holds, then 
	in $\Scal$, for every $n\in\Z_{+}$, there exists a transition sequence 
	\begin{equation}\label{eqn56_Det_PN}
		x_0'(L)\xrightarrow[]{s_1'(L)}x_1'(L)\xrightarrow[]{(s_2'(L))^n}x_1'(L)\xrightarrow[]{s_3'(L)}
		x_2'(L)
	\end{equation}
	such that $|\Mt(\Scal,\ell(s_1'(L)(s_2'(L))^ns_3'(L)))|>1$ and at $x_2'(L)$ there is an 
	infinite-length transition sequence with infinite-length label sequence.
	Hence $\Scal$ is 
	not strongly detectable. 

	If $\Scal$ is not strongly detectable, then for every $n\in\Z_{+}$, there exists a transition
	sequence $x_0\xrightarrow[]{s_1}x_1\xrightarrow[]{s_2}$ such that $x_0\in X_0$, $x_1\in X$,
	$s_1\in T^*$, $s_2\in T^{\omega}$, $|\ell(s_1)|>n$, $\ell(s_2)\in\Sig^{\omega}$, and
	$|\Mt(\Scal,\ell(s_1))|>1$. Then there is a transition sequence such that the sequence and
	$x_0\xrightarrow[]{s_1}x_1$ combine to \eqref{eqn55_1_Det_PN}
	if $n$ is sufficiently large by the finiteness of $X$. 
	Also by the finiteness of $X$, there exists a cycle with nonempty label sequence
	reachable from $x_1$. Hence \eqref{eqn55_Det_PN} holds.
	
	Next we show that \eqref{eqn55_Det_PN} can be verified in polynomial time. See Fig. \ref{fig20_Det_PN}
	for a sketch. 
	\begin{enumerate}
		\item Compute $\Acc(\CCa(\Acc(\Scal)))=(X',T',X_0',\to')$.
		\item Compute the set $X_e'\subset X'$ of states  $(x_2,\bar x_2)$ with $x_2\ne\bar x_2$ that are reachable from a cycle of $\Acc(\CCa(\Acc(\Scal)))$ 
			with positive-length label sequence. 
		\item Compute $$\overline X_0:=\{x\in X|(\exists x'\in X)[(x\ne x')\wedge(\text{either }(x,x')\text{ or }(x',x) 
			\text{ belongs to\\}X_e')]\}.$$
		\item Construct finite automaton $\Scal'$ from $\Acc(\Scal)$
			by replacing all initial states with all states of $\overline X_0$, then check whether $\LM^{\omega}
			(\Acc(\Scal'))=\emptyset$. 
	\end{enumerate}
	
	For the second step, one can firstly compute all strongly connected components of
	$\Obs(\Acc(\CCa(\Acc(\Scal))))$; the in 
	each component that contains an observable transition,
	choose an arbitrary state, and put all these states into a set $X_{oc}'$;
	thirdly, by searching starting from $X_{oc}'$ all reachable states,
	one can find all states $(x_2',\bar x_2')$ with $x_2'\ne\bar x_2'$ that are reachable 
	from $X_{oc}'$. The set of all these states $(x_2',\bar x_2')$ is exactly $X_e'$. 
	Hence the second step costs linear time of $\CCa(\Scal)$.
	
	The third step costs time $O(|X|^2)$.
	
	Note that the set $X_e$ in Fig. \ref{fig20_Det_PN} denotes the set of 
	states of $X$ that are reachable from $\overline X_0$ and belong to a cycle with positive-length label 
	sequence. Hence $X_e\ne\emptyset$ if and only if $\LM^{\omega}(\Acc(\Scal'))\ne\emptyset$.
	By Proposition \ref{prop8_Det_PN}, it takes linear time of $\Scal$ to check
	this condition by computing $\Obs(\Acc(\Scal'))$. One sees that $\LM^{\omega}(\Acc(\Scal'))\ne\emptyset$
	if and only if \eqref{eqn55_Det_PN} holds. Hence the fourth step consumes linear time of 
	$\Scal$.
	
	Then verifying strong detectability for $\Scal$
	takes linear time of $\CCa(\Scal)$, i.e., at most $O(|X|^4|T|^2)$.

		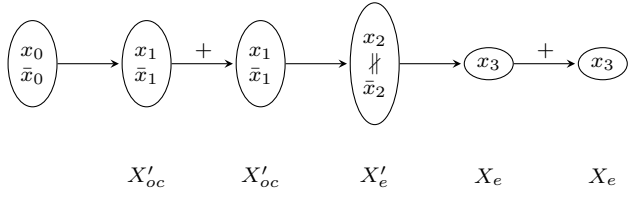
\begin{figure}
        \centering
\begin{tikzpicture}[>=stealth',shorten >=1pt,auto,node distance=1.5 cm, scale = 1.0, transform shape,
	->,>=stealth,inner sep=2pt,state/.style={shape=circle,draw,top color=red!10,bottom color=blue!30}]

	\node[elliptic state] (x0) {$\begin{matrix}x_0\\\bar x_0\end{matrix}$};
	\node[elliptic state] (x1) [right of = x0] {$\begin{matrix}x_1\\\bar x_1\end{matrix}$};
	\node[elliptic state] (x1') [right of = x1] {$\begin{matrix}x_1\\\bar x_1\end{matrix}$};
	\node[elliptic state] (x2) [right of = x1'] {$\begin{matrix}x_2\\\nparallel\\\bar x_2\end{matrix}$};
	\node[elliptic state] (x3) [right of = x2] {$x_3$};
	\node[elliptic state] (x3') [right of = x3] {$x_3$};

	\node(empty) (x0_)  [below of = x0] {}; 
	\node(empty) (x1_) [right of = x0_] {$X_{oc}'$};
	\node(empty) (x1'_) [right of = x1_] {$X_{oc}'$};
	\node(empty) (x2_) [right of = x1'_] {$X_{e}'$};
	\node(empty) (x3_) [right of = x2_] {$X_{e}$};
	\node(empty) (x3'_) [right of = x3_] {$X_{e}$};
	
	\path [->]
	(x0) edge (x1)
	(x1) edge node {$+$} (x1')
	(x1') edge (x2)
	(x2) edge (x3)
	(x3) edge node {$+$} (x3')
	;

        \end{tikzpicture}
		\caption{A sketch for verifying \eqref{eqn55_Det_PN}.}
		\label{fig20_Det_PN}
	\end{figure}

\end{proof}

\begin{example}\label{exam5_Det_PN}
	Recall the finite automaton $\Scal$ in Example \ref{exam4_Det_PN} (in the left part of Fig. 
	\ref{fig19_Det_PN}). 
	Following the procedure in the proof of Theorem \ref{thm7_Det_PN}, we have
	$X'_{oc}=\{(s_0,s_0),(s_1,s_1)\}$, $X_e'=\{(s_1,s_2),(s_2,s_1)\}$, $\overline X_0=
	\{s_1,s_2\}$. Replace all initial states of its observation automaton (shown in the left part of
	Fig. \ref{fig24_Det_PN}) by states of $\overline X_0$, one then has the corresponding automaton
	$\Scal'$ generates a nonempty $\omega$-language, since a cycle $s_1\xrightarrow[]{\bar\ep}s_1$
	is reachable from $s_1$ in $\Obs(\Acc(\Scal'))$. Then $\Scal$ is not strongly detectable.
\end{example}

%
%
%
%
%

\begin{theorem}\label{thm8_Det_PN}
	The eventual strong detectability of finite automata 
	can be verified in polynomial time.
\end{theorem}

\begin{proof}
	Consider a finite automaton ${\mathcal S}=(X,T,X_0,\to,\Sig,\ell)$ and another finite
	automaton $\Acc(\CCa(\Scal))=(X',T',X_0',\to')$.
	Similarly to Theorem \ref{thm7_Det_PN}, we use $\Acc(\CCa(\Scal))$, $\Obs(\Acc(\Scal))$,
	and $\Bifur(\Acc(\Scal))$ to verify its eventual strong detectability.
	
	One observes by definition that $\Scal$ is not eventually strongly detectable if and only if
	\begin{subequations}\label{eqn57_Det_PN}
		\begin{align}
			&\text{there is an infinite transition sequence }x_0\xrightarrow[]{s_1}\text{ such that}\label{eqn57_1_Det_PN}\\
			&x_0\in X_0,\text{ }\ell(s_1)\in\Sig^{\omega}\text{ and for every }n\in\Z_{+},
			\text{ there is a prefix }\label{eqn57_2_Det_PN}\\
			&s_1'\text{ of }s_1\text{ satisfying }
			|\ell(s_1')|>n\text{ and }|\Mt(\Scal,\ell(s_1'))|>1.\label{eqn57_3_Det_PN}
		\end{align}
	\end{subequations}

	We claim that \eqref{eqn57_Det_PN} holds if and only if one of the following items holds:
	\begin{enumerate}[(1)]
		\item\label{item13_Det_PN} In $\Acc(\CCa(\Scal))$, there exists an infinite transition sequence
			\begin{align}\label{eqn58_Det_PN}
				x_0'\xrightarrow[]{s_1'}x_1'\xrightarrow[]{s_2'}\cdots
			\end{align}
			such that $x_0'\in X_0'$, for every $i\in\Z_{+}$, $s_i'\in(T')^*$,
			$\ell(s_i')\in\Sig^+$, and $s_i'(L)\ne s_i'(R)$.
		\item\label{item14_Det_PN} In $\Scal$, there exists an infinite transition sequence  
			\begin{align}\label{eqn59_Det_PN}
				x_0\xrightarrow[]{s_1}x_1\xrightarrow[]{s_2}x_2\xrightarrow[]{s_3}\cdots
			\end{align}
			such that $x_0\in X_0$, for all $i\in\Z_{+}$, $s_i\in T^*$, $\ell(s_{i+1})\in\Sig^+$,
			and $|\Mt(\{x_i\},\sigma)|>1$ for some $\sigma\sqsubset \ell(s_{i+1})$.
	\end{enumerate}

	It is trivial to see that either Item \eqref{item13_Det_PN} or Item \eqref{item14_Det_PN}
	implies \eqref{eqn57_Det_PN}.
	
	Conversely suppose that \eqref{eqn57_Det_PN} holds
	but Item \eqref{item14_Det_PN} does not hold. Then for $\Scal$,
	there is an infinite transition sequence
	\begin{align}\label{eqn60_Det_PN}
		\bar x_0\xrightarrow[]{\bar s_1}\bar x_1\xrightarrow[]{\bar s_2}\bar x_2\xrightarrow[]{\bar s_3}\cdots
	\end{align}
	satisfying \eqref{eqn57_2_Det_PN} and \eqref{eqn57_3_Det_PN} such that for every $i\in\Z_{+}$, 
	$\bar s_i\in T^*$, $\ell(\bar s_{i+1})\in\Sig^+$, and
	$|\Mt(\{\bar x_i\},\bar\sigma)|=1$ for all $\bar\sigma\sqsubset\ell(\bar s_{i+1})$.
	Fix such a sequence \eqref{eqn60_Det_PN}. Then for every $i\in\Z_{+}$, there exists a finite 
	transition sequence
	\begin{equation}\label{eqn61_Det_PN}
		\bar x_0^i\xrightarrow[]{\bar s_1^i}\cdots\xrightarrow[]{\bar s_{i}^i}\bar x_i^i
	\end{equation}
	such that $\bar x_0^i\in X_0$, for all $j\in[1,i]$, one has 
	$\ell(\bar s_j^i)=\ell(\bar s_j)$, $\bar x_j^i\ne\bar x_j$.
	Choose $k$ sufficiently large, by the finiteness of $X$, we obtain a transition sequence
	\begin{equation}\label{eqn62_Det_PN}
		\bar x_0'\xrightarrow[]{\bar s_1'}\cdots\xrightarrow[]{\bar s_{k}'}\bar x_k'
	\end{equation}
	of $\Acc(\CCa(\Scal))$ such that $\bar x_0'\in X_0'$, the left component and the right component of
	\eqref{eqn62_Det_PN} are a prefix of \eqref{eqn60_Det_PN} and \eqref{eqn61_Det_PN} with $i=k$;
	for all $i\in[1,k]$, $\bar x_i'(L)\ne \bar x_i'(R)$, and $\bar x_{l'}'=\bar x_{l''}'$
	for some $0< l'<l''\le k$. Then the prefix $\bar x_0'\xrightarrow[]{\bar s_1'}\cdots
	\xrightarrow[]{\bar s_{l'}'}\bar x_{l'}'\xrightarrow[]{\bar s_{l'+1}'}\cdots
	\xrightarrow[]{\bar s_{l''}'}\bar x_{l''}'$ of \eqref{eqn62_Det_PN} can be extended to
	an infinite transition sequence of the form \eqref{eqn58_Det_PN} by repeating $\bar x_{l'}'
	\xrightarrow[]{\bar s_{l'+1}'}\cdots\xrightarrow[]{\bar s_{l''}'}\bar x_{l''}'$
	for infinitely many times, i.e., Item \eqref{item13_Det_PN} holds.

	Next we show that both Item \eqref{item13_Det_PN} and Item \eqref{item14_Det_PN} can be verified 
	in polynomial time.

	Observe that Item \eqref{item13_Det_PN} holds if and only if in $\Acc(\CCa(\Scal))$, there is a finite transition
	sequence
	\begin{equation}\label{eqn63_Det_PN}
		\tilde x_0'\xrightarrow[]{\tilde s_1'}\tilde x_1'\xrightarrow[]{\tilde s_2'}\tilde x_1'
	\end{equation}
	with $\tilde x_0'\in X_0'$, $\tilde s_1',\tilde s_2'\in(T')^*$ such that $\ell(\tilde s_2')\in\Sig^+$
	and $\tilde x_1'(L)\ne\tilde x_1'(R)$.
	Next we verify \eqref{eqn63_Det_PN} in polynomial time.  See Fig. \ref{fig21_Det_PN} for a sketch.
	\begin{enumerate}
		\item Compute $\Obs(\Acc(\CCa(\Scal)))$.
		\item Compute all strongly connected components of $\Obs(\Acc(\CCa(\Scal)))$.
		\item Denote the set of states $(x,\bar x)$ of $\Obs(\Acc(\CCa(\Scal)))$
			with $x\ne \bar x$ that belong to a cycle with nonempty label sequence by 
			$X_c'$, check whether $X_{c}'\ne\emptyset$.
	\end{enumerate}
	Each of the first two steps costs linear time of $\CCa(\Scal)$.
	Note that $X_{c}'\ne\emptyset$
	if and only if \eqref{eqn63_Det_PN} holds.
	Observe that $X_c'\ne\emptyset$ if and only if in one of the obtained strongly connected
	components, there is an observable transition and a state $(x',\bar x')$ with $x'\ne\bar x'$.
	Hence the third step also costs linear time. Overall, verifying Item 
	\eqref{item13_Det_PN} costs linear time of $\CCa(\Scal)$, at most $O(|X|^4|T|^2)$.

		\begin{figure}
        \centering
\begin{tikzpicture}[>=stealth',shorten >=1pt,auto,node distance=1.5 cm, scale = 1.0, transform shape,
	->,>=stealth,inner sep=2pt,state/.style={shape=circle,draw,top color=red!10,bottom color=blue!30}]

	\node[elliptic state] (x0) {$\begin{matrix}x_0\\\bar x_0\end{matrix}$};
	\node[elliptic state] (x1) [right of = x0] {$\begin{matrix}x_1\\\nparallel\\\bar x_1\end{matrix}$};
	\node[elliptic state] (x1') [right of = x1] {$\begin{matrix}x_1\\\nparallel\\\bar x_1\end{matrix}$};

	\node(empty) (x0_)  [below of = x0] {}; 
	\node(empty) (x1_) [right of = x0_] {$X_{c}'$};
	\node(empty) (x1'_) [right of = x1_] {$X_{c}'$};
	
	\path [->]
	(x0) edge (x1)
	(x1) edge node {$+$} (x1')
	;

        \end{tikzpicture}
		\caption{A sketch for verifying \eqref{eqn63_Det_PN}.}
		\label{fig21_Det_PN}
	\end{figure}
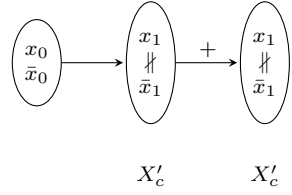

	Also observe that Item \eqref{item14_Det_PN} holds if and only if in $\Scal$,  
	there exists a finite transition sequence  
	\begin{align}\label{eqn64_Det_PN}
		\tilde x_0\xrightarrow[]{\tilde s_1}\tilde x_1\xrightarrow[]{\tilde s_2}\tilde x_1
	\end{align}
	such that $\tilde x_0\in X_0$, $\tilde s_1,\tilde s_2\in T^*$, $\ell(\tilde s_2)\in\Sig^+$,
	and $|\Mt(\{\tilde x_1\},\sigma)|>1$ for some $\sigma\sqsubset\ell(\tilde s_{2})$.

	Next we show that \eqref{eqn64_Det_PN} can be verified in polynomial time.
	See Fig. \ref{fig22_Det_PN} for a sketch. 
	\begin{enumerate}
		\item Compute $\Obs(\Acc(\Scal))$ and $\Bifur(\Acc(\Scal))$.
		\item Compute $X_{oc}$ and $X_{bc}$, where $X_{oc}$ (resp. $X_{bc}$) is the set of 
			states of $\Acc(\Scal)$ that belong to a cycle containing an observable transition
			(resp. a bifurcation transition).
		\item Check whether $X_{oc}\cap X_{bc}=\emptyset$.
	\end{enumerate}
	Note that
	a state $x$ of $\Acc(\Scal)$ belongs to a cycle containing an observable transition 
	(resp. a bifurcation transition) if and only if $x$ is any state of any strongly connected component
	of $\Obs(\Acc(\Scal))$ (resp. $\Bifur(\Acc(\Scal))$) that contains an observable transition 
	(resp. a bifurcation transition).
	Then one has $X_{oc}\cap X_{bc}\ne\emptyset$ if and only if 
	\eqref{eqn64_Det_PN} holds. Hence it takes linear time of $\Scal$ to check whether 
	Item \eqref{item14_Det_PN} holds.

			\begin{figure}
        \centering
\begin{tikzpicture}[>=stealth',shorten >=1pt,auto,node distance=1.9 cm, scale = 1.0, transform shape,
	->,>=stealth,inner sep=2pt,state/.style={shape=circle,draw,top color=red!10,bottom color=blue!30}]

 	\node[elliptic state] (x0) {$x_0$};
	\node[elliptic state] (x1) [right of = x0] {$x_1$}; 
	\node[elliptic state] (x1') [right of = x1] {$x_1$};

	\node(empty) (x0_)  [below of = x0] {}; 
	\node(empty) (x1_) [right of = x0_] {$X_{oc}\cap X_{bc}$};
	\node(empty) (x1'_) [right of = x1_] {$X_{oc}\cap X_{bc}$};
	
	\path [->]
	(x0) edge (x1)
	(x1) edge node {$+$} node [below, sloped] {bifurcation} (x1')
	;

       \end{tikzpicture}
		\caption{A sketch for verifying \eqref{eqn64_Det_PN}.}
		\label{fig22_Det_PN}
	\end{figure}
\end{proof}

\begin{example}\label{exam6_Det_PN}
	Recall the finite automaton $\Scal$ in Example \ref{exam5_Det_PN} (in the left part of Fig. 
	\ref{fig19_Det_PN}). 
	Following the procedure in the proof of Theorem \ref{thm8_Det_PN}, by Figs. \ref{fig19_Det_PN}
	and \ref{fig24_Det_PN}, we have $X_{oc}=\{s_0,s_1\}$, $X_{bc}=\emptyset$,
	$X_{oc}\cap X_{bc}=\emptyset$ (implying that Item \eqref{item14_Det_PN} does not hold), and
	$X_{c}'=\emptyset$ (implying that Item \eqref{item13_Det_PN} 
	does not hold either), then $\Scal$ is eventually strongly detectable. 
\end{example}

\begin{remark}\label{rem1_Det_PN}
	Let us analyse the computational complexity of using \cite[Theorem 5]{Shu2011GDetectabilityDES} to
	verify strong detectability of finite automata satisfying Assumption \ref{assum1_Det_PN}. In 
	\cite{Shu2011GDetectabilityDES}, for a finite automaton $\Scal$ (satisfying Assumption \ref{assum1_Det_PN}),
	a nondeterministic finite automaton $G_{det}$ with at most $|X|^2/2+|X|/2+1$ states and
	at most $(|X|^2/2+|X|/2+1)^2|T|$ transitions is constructed to verify its strong detectability,
	where every state of $G_{det}$ is a subset of states of $\Scal$ with cardinality $1$ or $2$, except 
	for the initial state of $G_{det}$ being a superset of $X_0$.
	The time consumption for computing $G_{det}$ is as follows:
	\begin{align}
		&|X|\underbrace{|T_{\ep}||X|}_{\subalign{&\text{compute initial state $Q_0$ of $G_{det}$ by traversing all $\ep$-transition}\\&\text{sequences from $X_0$, where $X_0\subset Q_0\subset X$}}}+\\
			&|Q_0|\underbrace{|T_o||X|}_{\subalign{&\text{compute}\\&Q'\subset X\\&\text{by traversing}\\&\text{all observable}\\&\text{transitions}\\&\text{from $Q_0$}}} +
			\underbrace{|\Sig||X||T_\ep||X|}_{\subalign{&\text{compute $Q''$}\\&\text{by traversing}\\&\text{all $\ep$-transition}\\&\text{sequences}\\&\text{from $Q'$}}} +
			\underbrace{|\Sig||X|^2}_{\subalign{&\text{split $Q''$ into}\\&\text{subsets of $X$}\\&\text{of cardinality $2$}\\&\text{to obtain}\\&\text{non-initial states}\\&\text{of $G_{det}$}}}+\label{eqn73_Det_PN}\\
			&\underbrace{|X|^2(2|T_o||X|+|\Sig||X||T_\ep||X|+|\Sig||X|^2)}_{\text{Repeat \eqref{eqn73_Det_PN}
			from states of $G_{det}$ of cardinality $2$}}+\\
			&\underbrace{|X|(|T_o||X|+|\Sig||X||T_\ep||X|+|\Sig||X|^2)}_{\text{Repeat \eqref{eqn73_Det_PN}
			from states of $G_{det}$ of cardinality $1$}},
	\end{align}
	i.e., at most $O(2|X|^3|T_o|+|X|^4|\Sig||T_\ep|+|X|^4|\Sig|)$. For the special case when all observable
	events can be directly observed studied in \cite{Shu2011GDetectabilityDES},
	the complexity is $O(2|X|^3|T_o|+|X|^4|T_o||T_\ep|+|X|^4|T_o|)$.
	Actually, this construction tracks sets of states of $\Scal$ with consistent observations,
	which is similar to the powerset construction that is of exponential size of $\Scal$.
	It is proved that $\Scal$ is strongly detectable if and only if every state of $G_{det}$
	reachable from a cycle is a singleton. This condition can be check in linear time of $G_{det}$
	by computing strongly connected components of $G_{det}$.

	However, this method generally does not apply to a finite automaton that does not satisfy Assumption
	\ref{assum1_Det_PN}. For example, let us consider the finite automaton $\Scal$ in the left part of
	Fig. \ref{fig19_Det_PN}. Remove the self-loop on $s_1$, and denote the new automaton 
	by $\bar\Scal$. Then one directly sees that $\LM^{\omega}
	(\bar\Scal)=\{a^{\omega}\}$, and $\bar\Scal$ is strongly detectable. However, in the corresponding $G_{det}$,
	which consists of a self-loop with label $a$ on $\{s_0\}$ and a transition from $\{s_0\}$ to $\{s_1,s_2\}$
	with label $b$, there is a state $\{s_1,s_2\}$ with cardinality $2$ reachable from a cycle,
	hence $\bar\Scal$ is not strongly detectable by \cite[Theorem 5]{Shu2011GDetectabilityDES}.
	Actually, the verification method does not apply to this example because, two deadlock states 
	$s_1$ and $s_2$ are not in any infinite-length transition sequence, but reachable from a state
	$s_0$ that belongs to an infinite-length transition sequence with infinite-length label sequence.
\end{remark}

By Proposition \ref{prop5_Det_PN}, we have shown that generally strong detectability is not equivalent 
to eventual strong detectability even for deterministic finite automata satisfying
Assumption \ref{assum1_Det_PN}.
However, using the method in \cite{Shu2011GDetectabilityDES},
we can prove that these two notions
are equivalent for deterministic finite automata satisfying Assumption \ref{assum1_Det_PN}
each of whose events can be directly observed (see Proposition \ref{cor2_Det_PN}).
Furthermore, by using the proofs of Theorems \ref{thm7_Det_PN} and \ref{thm8_Det_PN}, we can 
prove an even stronger result: these two notions
are equivalent for a deterministic finite automaton such that
each of its events can be directly observed (see Proposition \ref{cor3_Det_PN}).

\begin{proposition}\label{cor2_Det_PN}
	Strong detectability is equivalent to eventual strong detectability for deterministic finite
	automata satisfying Assumption \ref{assum1_Det_PN} each of whose events can be directly observed.
\end{proposition}

\begin{proof}
	Consider a deterministic finite automaton ${\cal S}=(X,T,X_0,\to,\Sig,\ell)$ satisfying Assumption 
	\ref{assum1_Det_PN} and each of its events can be directly observed. Construct the corresponding 
	nondeterministic finite automaton $G_{det}$ as in \cite[Theorem 5]{Shu2011GDetectabilityDES},
	which shows that $\Scal$ is strongly detectable if and only if in $G_{det}$,
	every state reachable from a cycle is a singleton. Actually, using similar procedure,
	one can prove that $\Scal$ is eventually strongly detectable if and only if in $G_{det}$,
	each state of each cycle is a singleton. Since 
	$\cal S$ is deterministic and each of its events can be directly observed,
	for each transition $q\xrightarrow[]{\s}q'$ in $G_{det}$,
	we have $|q'|\le|q|$. Hence if each state of each cycle
	in $G_{det}$ is a singleton, then each state reachable from a cycle
	is also a singleton. That is, eventual strong detectability is stronger than
	strong detectability, and hence they are equivalent.
\end{proof}

\begin{proposition}\label{cor3_Det_PN}
	Strong detectability is equivalent to eventual strong detectability for a deterministic finite automaton
	such that each of its events can be directly observed.
\end{proposition}

\begin{proof}
	Consider a deterministic finite automaton ${\cal S}=(X,T,X_0,\to,\Sig,\ell)$ such that
	each of its events can be directly observed. 
	Construct another finite automaton $\Acc(\CCa(\Scal))=(X',T',X_0',\to')$, i.e.,
	the accessible part of the concurrent composition of $\Scal$.

	Since by definition strong detectability is stronger than eventual strong detectability,
	we only need to prove that if such an $\Scal$ is not strongly detectable, then it is not
	eventually strongly detectable either.

	Assume that $\Scal$ is not strongly detectable. Then by the proof of Theorem \ref{thm7_Det_PN}, one has
	$X'_{oc}$, $X'_e$, and $X_e$ (also see Fig. \ref{fig20_Det_PN}) are all nonempty,
	where $X_e$ is the set of states of $\Scal$
	to which some cycle with nonempty label sequence is reachable; $X'_e$ is the set of states $(x,x')$
	of $\Acc(\CCa(\Scal))$ such that $x\ne x'$ and at least one of $x$ and $x'$, say $x$, belong to $X_e$ or $X_e$ 
	is reachable 
	from $x$; $X'_{oc}$ is the set of states $(x'',x''')$ of $\Acc(\CCa(\Scal))$ such that 
	$(x'',x''')$ belong to a cycle with nonempty label sequence, and, either $(x'',x''')$ belong
	to $X'_e$ or $X'_e$ is reachable from $(x'',x''')$.

	Then in $\Acc(\CCa(\Scal))$, there is a transition sequence
	\begin{equation}\label{eqn69_Det_PN}
		x_0'\xrightarrow[]{s_1'}x_1'\xrightarrow[]{s_2'}x_1'\xrightarrow[]{s_3'}x_2'
	\end{equation} such that 
	$x_0'\in X_0'$, $s_1',s_2',s_3'\in (T')^*$, $\ell(s_2')\in\Sig^+$, and $x_2'(L)\ne x_2'(R)$.
	Since $\Scal$ is deterministic and each of its events can be directly observed, we have $T'=\{(t,t)|t\in T\}$,
	and for each state of \eqref{eqn69_Det_PN}, its left component differs from its right component.
	Then $x_1'(L)\ne x_1'(R)$, and the corresponding set $X_c'$ in the proof of Theorem 
	\ref{thm8_Det_PN} (also see Fig. \ref{fig21_Det_PN}) is nonempty,
	where $X_c'$ is the set of states $(x,x')$ of $\Acc(\CCa(\Scal))$ belonging
	to a cycle with nonempty label sequence and satisfying $x\ne x'$.
	That is, Item \eqref{item13_Det_PN} in the proof of Theorem \ref{thm8_Det_PN} holds. Hence $\Scal$
	is not eventually strongly detectable.

%
\end{proof}


\begin{remark}
	Similar to instant strong detectability, eventual strong detectability is also
	a uniform concept.
	That is, a labeled Petri net is eventually strongly detectable if and only if it is
	eventually strongly detectable when its initial marking is replaced by any of its reachable markings.
	Formally, for a labeled Petri net $G=(N,M_0,\Sig,\ell)$, $G$ is eventually strongly detectable
	if and only if $G'=(N,M,\Sig,\ell)$ is eventually strongly detectable for each $M\in\Rt(N,M_0)$.
\end{remark}

\begin{example}
	Let us consider a labeled P/T net $G$ shown in Fig. \ref{fig8_Det_PN}, where $a$ is the label
	of all transitions.
	We have $\LM^{\omega}(G)=a^{\omega}$,
	$|\Mt(G,a)|=2$, $|\Mt(G,a^{n+2})|=1$ for all $n\in\N$.
	Hence the net is strongly detectable, but not instantly strongly detectable.
	A deterministic finite automaton satisfying Assumption \ref{assum1_Det_PN}
	can be obtained from the net when $a$ is regarded as labels of all events.
	The obtained automaton is also strongly detectable, but not instantly strongly detectable.

		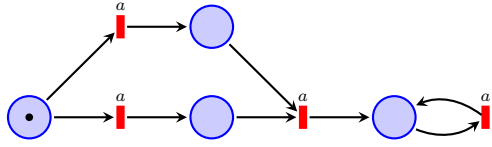
\begin{figure}[htbp]
		\tikzset{global scale/.style={
    scale=#1,
    every node/.append style={scale=#1}}}
		\begin{center}
			\begin{tikzpicture}[global scale = 1.0,
				>=stealth',shorten >=1pt,thick,auto,node distance=1.5 cm, scale = 0.8, transform shape,
	->,>=stealth,inner sep=2pt,
				every transition/.style={draw=red,fill=red,minimum width=1mm,minimum height=3.5mm},
				every place/.style={draw=blue,fill=blue!20,minimum size=7mm}]
				\tikzstyle{emptynode}=[inner sep=0,outer sep=0]
				\node[place, tokens=1] (p1) {};
				\node[transition, label=above:$a$,right of = p1] (t3){}
				edge[pre] (p1);
				\node[transition, label=above:$a$,above of = t3] (t2) {}
				edge[pre] (p1);
				\node[place, right of = t2] (p2) {}
				edge[pre] (t2);
				\node[place, right of = t3] (p3) {}
				edge[pre] (t3);
				\node[transition, label=above:$a$,right of = p3] (t4) {}
				edge[pre] (p2)
				edge[pre] (p3);
				\node[place, right of = t4] (p4) {}
				edge[pre] (t4);
				\node[transition, label=above:$a$, right of = p4] (t5) {}
				edge[pre,bend left] (p4)
				edge[post,bend right] (p4);
			\end{tikzpicture}
			\caption{A labeled P/T net $G$ that is strongly detectable, but not instantly strongly detectable.}
			\label{fig8_Det_PN}
		\end{center}
	\end{figure}

\end{example}

\begin{example}
	Consider a labeled P/T net $G$ as shown in Fig. \ref{fig9_Det_PN}, where $a,b$ are labels.
	We have $\LM^{\omega}(G)=a^{\omega}+
	a^*b^{\omega}$, $|\Mt(G,a^n)|=1$, $|\Mt(G,a^nb^m)|=2$ for all $m,n\in\Z_{+}$.
	Hence the net is weakly detectable, but not eventually strongly detectable.
	The deterministic finite automaton obtained from the net when $a$ and $b$ are regarded as
	labels of events is also weekly detectable, but not eventually strongly detectable.

		\begin{figure}[htbp]
		\tikzset{global scale/.style={
    scale=#1,
    every node/.append style={scale=#1}}}
		\begin{center}
			\begin{tikzpicture}[global scale = 1.0,
				>=stealth',shorten >=1pt,thick,auto,node distance=1.5 cm, scale = 0.8, transform shape,
	->,>=stealth,inner sep=2pt,
				every transition/.style={draw=red,fill=red,minimum width=1mm,minimum height=3.5mm},
				every place/.style={draw=blue,fill=blue!20,minimum size=7mm}]
				\tikzstyle{emptynode}=[inner sep=0,outer sep=0]
				\node[place, tokens=1] (p1) {};
				\node[transition, label=above:$a$,left of = p1] (t1) {}
				edge[pre, bend left] (p1)
				edge[post, bend right] (p1);
				\node[transition, label=above:$b$,right of = p1] (t3){}
				edge[pre] (p1);
				\node[transition, label=above:$b$,above of = t3] (t2) {}
				edge[pre] (p1);
				\node[place, right of = t2] (p2) {}
				edge[pre] (t2);
				\node[place, right of = t3] (p3) {}
				edge[pre] (t3);
				\node[transition, label=above:$b$,right of = p2] (t4) {}
				edge[pre, bend left] (p2)
				edge[post, bend right] (p2);
				\node[transition, label=above:$b$,right of = p3] (t5) {}
				edge[pre, bend left] (p3)
				edge[post, bend right] (p3);
			\end{tikzpicture}
			\caption{A labeled P/T net $G$ that is weakly detectable, but not eventually strongly detectable.}
			\label{fig9_Det_PN}
		\end{center}
	\end{figure}
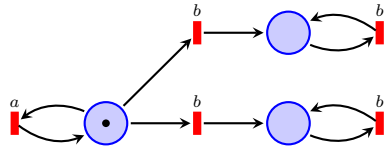

\end{example}

\subsubsection{Labeled Petri nets}

We next characterize eventual strong detectability for labeled P/T nets.
Similar to instant strong detectability,
if a labeled Petri net $G$ satisfies $\LM^{\omega}(G)=\emptyset$, then it is 
eventually strongly detectable. Different from giving the decidability result of instant strong detectability
without any assumption (Theorem \ref{thm3_Det_PN}),
we will prove the decidability result of eventual strong detectability under 
\eqref{item6_Det_PN} of Assumption \ref{assum2_Det_PN}.

Checking strong detectability for labeled P/T nets is proved to be decidable
and EXPSPACE-hard in the size of a labeled P/T net
\cite{Masopust2018DetectabilityPetriNet} under Assumption \ref{assum1_Det_PN}
(it is not difficult to see that the assumption ``there does not exist an infinite unobservable 
sequence'' used in \cite{Masopust2018DetectabilityPetriNet} is equivalent to promptness by 
Dickson's lemma).
Here the size of a P/T net $G=(N=(P,T,Pre,Post),M_0)$ is $\ceil*{\log|P|}+\ceil*{\log|T|}+$
the size of $\{Pre(p,t)|p\in P,t\in T\}\cup\{Post(p,t)|p\in P,t\in T\}\cup\{M_0(p)|p\in P\}$,
where the last term means the sum of the lengths of the binary representations of the elements
of $\{Pre(p,t)|p\in P,t\in T\}\cup\{Post(p,t)|p\in P,t\in T\}\cup\{M_0(p)|p\in P\}$
\cite{Atig2009YenPathLogicPetriNet,Yen1992YenPathLogicPetriNet}.
Hence the size of a labeled P/T net can be defined as the sum of the size of its underlying P/T net
and that of its labeling function $\ell:T\to\Sig\cup\{\epsilon\}$, where the latter
is actually no greater than $|T|$.

Consider a labeled Petri net $G$. Consider a reachable marking $M_1$ 
of $G$ and a firing sequence $\psi=M_1[t_2\rangle M_2[t_3\rangle \cdots[t_l\rangle M_l$,
where $l>1$, $t_i$ is a transition of $G$ for every $i\in[2,l]$.
We say that {\it $\psi$ has a bifurcation} if there exists $k\in[2,l]$ such that 
in the concurrent composition $\CCn(G)$ of $G$,
there is a firing sequence $M_1'[t_2'\rangle M_2'[ t_3'\rangle\cdots[t_n'\rangle M_n'$ for some
$n>1$ and with all $t_2',\dots,t_n'$ being transitions of $\CCn(G)$ such that $M_1'|_{P_1}=M_1'|_{P_2}=M_1$,
$M_n'|_{P_1}=M_k$, the left component of $t_2'\dots t_n'$ equals $t_2\dots t_k$, and 
$M_{k'}'|_{P_1}\ne M_{k'}'|_{P_2}$ for some $k'\in[2,n]$.

For $G$, for two infinite firing sequences
\begin{subequations}
\begin{align}
	&M_0[\widetilde{t}_1\rangle \widetilde{M}_1[\widetilde{t}_2\rangle\cdots\text{ and }\label{eqn42_1_Det_PN}\\
	&M_0[\widehat{t}_1\rangle \widehat{M}_1[\widehat{t}_2\rangle\cdots,\label{eqn42_2_Det_PN}
\end{align}
\end{subequations}
where $\widetilde{t}_i,\widehat{t}_i$ are transitions of $G$ for all $i\in\Z_{+}$,
we call {\it they merge after a finite time} if in $\CCn(G)$, there is an infinite firing sequence
$M_0'[t_1'\rangle M_1'[t_2'\rangle\cdots$ with $t_1',t_2',\dots$ all being transitions of $\CCn(G)$
such that the left component and right component of $t_1't_2'\dots$ equal $\widetilde{t}_1
\widetilde{t}_2\dots$ and $\widehat{t}_1\widehat{t}_2\dots$, respectively, and 
there exists $k\in\Z_{+}$ such that $M_j'|_{P_1}={M}_j'|_{P_2}$ for all $j>k$.

\begin{theorem}\label{thm4_Det_PN}
	\begin{enumerate}[(1)]
		\item\label{item3_Det_PN}
			The eventual strong detectability of a labeled P/T net $G$ under \eqref{item6_Det_PN}
			of Assumption \ref{assum2_Det_PN}
			is decidable.
		\item\label{item4_Det_PN}
			Deciding whether a labeled P/T net $G$ with $\LM^{\omega}(G)\ne\emptyset$
			is eventually strongly detectable is
			EXPSPACE-hard. 
	\end{enumerate}
\end{theorem}

\begin{proof}
	\eqref{item3_Det_PN} Proof of the decidability result: 

	By Proposition \ref{prop7_Det_PN}, we first verify whether $G$ satisfies $\LM^{\omega}(G)\ne
	\emptyset$ in EXPSPACE. If no, then $G$ is eventually strongly detectable.
	Otherwise, continue the following procedure.

	Consider a labeled Petri net $G=(N=(P,T,Pre,Post),M_0,\Sig,\ell)$ with $\LM^{\omega}(G)$
	being nonempty.
	By definition, $G$ is not eventually strongly detectable if and only if
	there exists $\s\in \LM^{\omega}(G)$ such that for all $k\in\N$ there exists a prefix
	$\bar\s$ of $\s$ satisfying $|\bar\s|>k$ and $|\Mt(G,\bar\s)|>1$.
	We construct the concurrent composition $\CCn(G)=(N'=(P',T',Pre',Post'),M_0')$ of $G$
	as in \eqref{eqn6_Det_PN}.

	We claim that $G$ is not eventually strongly detectable if and only if one of the following
	two items holds (see Examples \ref{exam1_Det_PN} and \ref{exam2_Det_PN}):  
	\begin{enumerate}[(1)]
		\item\label{item7_Det_PN} In $\CCn(G)$, there exists an infinite firing sequence
			\begin{align}\label{eqn43_Det_PN}
				M_0'[s_1'\rangle M_1'[s_2'\rangle\cdots,
			\end{align}
			where for every $i\in\Z_{+}$, $s_i'$ contains a transition of $T_o'$, and $M_i'|_{P_1}\ne M_i'|_{P_2}$.
		\item\label{item8_Det_PN} In $G$, there exists an infinite firing sequence 
			\begin{align}\label{eqn36_Det_PN}
				M_0[s_1\rangle M_1[s_2\rangle M_2[s_3\rangle \cdots
			\end{align}
			such that $M_0[s_1\rangle M_1$ has a bifurcation,
			for each $i\in\Z_{+}$, $\ell(s_i)\in\Sig^+$,
			and $M_i[s_{i+1}\rangle M_{i+1}$ also has a bifurcation.
	\end{enumerate}

	Apparently if Item \eqref{item7_Det_PN} or Item \eqref{item8_Det_PN} holds,
	then $G$ is not eventually strongly detectable.

	Suppose that $G$ is not eventually strongly detectable. Then there exists an infinite firing sequence
	\begin{align}\label{eqn35_Det_PN}
		M_0[\bar s_1\rangle \overline{M}_1[\bar s_2\rangle \overline{M}_2[\bar s_3\rangle \cdots
	\end{align}
	such that $\ell(\bar s_i)\in\Sig^+$ and $|\Mt(G,\ell(\bar s_1\dots\bar s_i))|>1$ for all $i\in\Z_{+}$.
	Next we fix such a sequence \eqref{eqn35_Det_PN}.
	
	Furthermore, suppose that Item \eqref{item7_Det_PN} does not hold. Then \eqref{eqn35_Det_PN}
	and each infinite firing sequence of
	$G$ staring at $M_0$ and having the same label sequence as \eqref{eqn35_Det_PN} has will merge after a finite
	time, since the label sequence of \eqref{eqn35_Det_PN} is of infinite length.
	Next we prove that Item \eqref{item8_Det_PN} holds. If in \eqref{eqn35_Det_PN}, infinitely many of 
	$M_0[\bar s_1\rangle \overline{M}_1$, $\overline{M}_1[\bar s_2\rangle \overline{M}_2$,
	$\dots$ have bifurcations, then \eqref{eqn35_Det_PN}
	is a firing sequence satisfying the requirement in Item \eqref{item8_Det_PN}. Next we assume that there are
	only finitely many of them having bifurcations, and reach a contradiction.
	Without loss of generality, we assume that only $M_0[\bar s_1\rangle \overline{M}_1$ has a bifurcation.
	Then for each $k\in\Z_{+}$, 
	there exists a firing sequence $M_0[\widetilde{s}_k\rangle \widetilde{ M}_k$ such that
	$\ell(\widetilde{s}_k)\sqsubset
	\ell(\bar s_1\bar s_2\dots)$, $|\ell(\widetilde{s}_k)|>k$, and some prefix of \eqref{eqn35_Det_PN}
	and $M_0[\widetilde{s}_k\rangle \widetilde{ M}_k$
	can be combined to obtain a firing sequence $M_0'[\overline{s}_k'\rangle \overline{M}_k'$
	of $\CCn(G)$ such that the label sequence of the right component of $\overline{s}_k'$ equals $\ell(\widetilde{s}_k)$,
	$\overline{M}_k'|_{P_2}=\widetilde{M}_k$, and $\overline{M}_k'|_{P_1}\ne\overline{M}_k'|_{P_2}$.
	Collecting all such firing sequences
	$M_0[\widetilde{s}_k\rangle \widetilde{M}_k$, $k\in\Z_{+}$, we obtain a locally finite,
	infinite tree $\mathbb T$ with ${M}_0$ the root. Also collect all such markings $\widetilde{M}_k$, $k\in\Z_{+}$,
	to obtain a set $\mathbb{M}$. Observe that in $\mathbb{T}$, $M_0$ has infinitely many descendants of
	$\mathbb{M}$.
	Also observe in $\mathbb{T}$ that one of the finitely many children of $M_0$ also has
	infinitely many descendants of $\mathbb{M}$, denote such a child of $M_0$ by $\widehat{M}_1$, 
	then we obtain a firing 
	sequence $M_0[\widehat{t}_1\rangle \widehat{M}_1$ of $G$, where $\widehat{t}_1\in T$.
	Since $\mathbb{T}$ is locally finite,
	repeating the process of looking for $M_0[\widehat{t}_1\rangle \widehat{M}_1$,
	we can obtain an infinite firing sequence  
	\begin{equation}\label{eqn44_Det_PN}  
		M_0[\widehat{t}_1\rangle \widehat{M}_1[\widehat{t}_2\rangle\cdots
	\end{equation}
	of $G$ such that for each $i\in\Z_{+}$, $\widehat{M}_i$ has infinitely many descendants of $\mathbb{M}$ in $\mathbb{T}$.
	By \eqref{item6_Det_PN} of Assumption \ref{assum2_Det_PN}, we have \eqref{eqn44_Det_PN} is labeled by an 
	infinite-length label sequence. Also, since for each $i\in\Z_{+}$, $M_0[\widehat{t}_1\dots
	\widehat{t}_i\rangle\widehat{M}_i$ is a prefix of some path of $\mathbb T$, we have
	$\ell(\widehat{t}_1\widehat{t}_2\dots)=\ell(\bar s_1\bar s_2\dots)$. Then it is not difficult to see 
	that \eqref{eqn44_Det_PN} and \eqref{eqn35_Det_PN} can be combined into an infinite firing sequence of $\CCn(G)$
	satisfying the requirement in Item \eqref{item7_Det_PN}, which is a contradiction.

	
	
	Next we prove that the satisfiability of Item \eqref{item7_Det_PN} or Item \eqref{item8_Det_PN} are both decidable,
	completing the proof of the decidability result of eventual strong detectability.

	For Item \eqref{item7_Det_PN}:
	
	We claim that Item \eqref{item7_Det_PN} holds if and only if
	there exists a firing sequence
	\begin{equation}\label{eqn7_Det_PN}
		M_0'[s_1'\rangle M_1'[s_2'\rangle M_2'
	\end{equation}
	in $\CCn(G)$ satisfying
	\begin{equation}\label{eqn8_Det_PN}
		\begin{split}
			&(M_2'\ge M_1')\wedge(s_2'\text{ contains a transition in }T_o')
		\wedge
		(M_2'|_{P_1}\ne M_2'|_{P_2}),
		\end{split}
	\end{equation}
	where $T_o'\subset T'$ is shown in \eqref{eqn6_Det_PN}.
	That is, we next prove that Item \eqref{item7_Det_PN} holds if and only if
	\begin{equation}\label{eqn20_Det_PN}
		(\exists M_1',M_2')(\exists s_1',s_2')[\eqref{eqn7_Det_PN}\wedge\eqref{eqn8_Det_PN}]
	\end{equation} 
	is satisfied.

	``if'': Assume that for $\CCn(G)$, Eqn. \eqref{eqn20_Det_PN} holds. Then Item \eqref{item7_Det_PN} holds,
	because $M_2'|_{P_1}\ne M_2'|_{P_2}$, $s_2'$ contains a transition in $T_0$ (hence $\ell(s_2')$
	is of positive length),
	and $M_1'[s_2'\rangle M_2'$ is a repetitive firing sequence and can fire consecutively for infinitely many times.



	``only if'': Assume that Item \eqref{item7_Det_PN} holds, and fix a sequence \eqref{eqn43_Det_PN}.

	By Dickson's lemma, the set $\{M_0',M_1',\dots\}$ contains at most finitely many distinct minimal elements.
	Then there exists $k\in\Z_{+}$ such that $\{M_0',\dots,M_k'\}$ contains the maximal number of distinct 
	minimal elements of $\{M_0',M_1',\dots\}$. Hence there exists $0\le l\le k$ such that $M_l'\le M_{k+1}'$.
	Then the firing sequence $$M_0'[s_1'\dots s_l'\rangle M_l'[s_{l+1}'\dots s_{k+1}'\rangle M'_{k+1}$$ satisfies that
	$M'_{k+1}\ge M'_l$, $s_{l+1}'\dots s_{k+1}'$ contains at least one transition of $T_o'$, and $M'_{k+1}|_{P_1}\ne 
	M'_{k+1}|_{P_2}$, i.e., \eqref{eqn20_Det_PN} holds.

	In \eqref{eqn8_Det_PN}, ``$M_2'\ge M_1'$'' can be expressed as combination of marking predicates,
	``$s_2'\text{ contains a transition in }T_o'$'' is a transition predicate,
	only ``$M_2'|_{P_1}\ne M_2'|_{P_2}$'' is not a predicate. 

	Similarly to the proof of Theorem \ref{thm3_Det_PN},
    add two new places $p_0''$ and $p_1''$ into $\CCn(G)$, where initially $p_0''$ contains exactly $1$ token,
	but $p_1''$ contains no token; add one new transition $r_1''$, and arcs $p_0''\to r_1''\to p_1''$,
	both with weight $1$. Also, for each transition $t$ in $\CCn(G)$, add arcs $p_1''\to t\to p_1''$, both with
	weight $1$. Then we obtain a new Petri net $\CCn(G)'$.
	We have $\CCn(G)$ satisfies \eqref{eqn20_Det_PN} if and only if $\CCn(G)'$ satisfies the Yen's path formula
	\begin{equation}\label{eqn30_Det_PN}
		\begin{split}
		&(\exists M_1'',M_2'',M_3'')(\exists s_1'',s_2'',s_3'')[\\
		&(M_0''[s_1''\rangle M_1''[s_2''\rangle M_2''[s_3''\rangle M_3'')\wedge\\
		&(s_1''=r_1'')\wedge(M_3''\ge M_2'')\wedge (s_3''\text{ contains a transition of }T_o')\wedge\\
		&((M_3''-M_1'')|_{P_1}\ne (M_3''-M_1'')|_{P_2})],
		\end{split}
	\end{equation}
	where note that one always has $M_1''|_{P_1}=M_1''|_{P_2}$.

	Then by Proposition \ref{prop6_Det_PN},
	the satisfiability of \eqref{eqn20_Det_PN} is decidable, implying that the satisfiability of Item
	\eqref{item7_Det_PN} is decidable.

	Next we prove that the satisfiability of Item \eqref{item8_Det_PN} is decidable.

	We claim that for $G$, Item \eqref{item8_Det_PN} holds if and only if 
	\begin{subequations}\label{eqn39_Det_PN}
		\begin{align}
			&\text{there exists a firing sequence }M_0[\underline{s}_1\rangle \underline{M}_1[\underline{s}_2\rangle 
			\underline{M}_2\text{ satisfying}\label{eqn39_1_Det_PN}\\
			&\underline{M}_1\le \underline{M}_2,\label{eqn39_2_Det_PN}\\
			&\underline{s}_2\text{ contains an observable transition, and }\label{eqn39_3_Det_PN}\\
			&\underline{M}_1[\underline{s}_2\rangle \underline{M}_2\text{ contains a bifurcation}\label{eqn39_4_Det_PN}.
		\end{align}
	\end{subequations}

	Assume that for $G$, Item \eqref{item8_Det_PN} holds. Again by Dickson's lemma, 
	there exist $0\le l<k$ such that
	the firing sequence $M_0[s_1\dots s_l\rangle M_l[s_{l+1}\dots s_k\rangle M_k$ satisfies that
	$M_l\le M_k$, $s_{l+1}\dots s_k$ contains an observable transition, and $M_l[s_{l+1}\dots 
	s_k\rangle M_k$ has a bifurcation. That is, \eqref{eqn39_Det_PN} holds.

	Assume that \eqref{eqn39_Det_PN} holds. By \eqref{eqn39_2_Det_PN}, \eqref{eqn39_3_Det_PN},
	and \eqref{eqn39_4_Det_PN}, the sequence in  \eqref{eqn39_1_Det_PN}
	can be extended to an infinite firing sequence
	$$M_0[\underline{s}_1\rangle \underline{M}_1[\underline{s}_2\rangle \underline{M}_2 [\underline{s}_2\rangle
	(\underline{M}_2+(\underline{M}_2-\underline{M}_1))[\underline{s}_2\rangle\cdots[\underline{s}_2\rangle
	(\underline{M}_2+k(\underline{M}_2-\underline{M}_1))[\underline{s}_2\rangle\cdots$$ satisfying for each
	$l\in\Z_{+}$, one has $(\underline{M}_2+l(\underline{M}_2-\underline{M}_1))[\underline{s}_2\rangle
	(\underline{M}_2+(l+1)(\underline{M}_2-\underline{M}_1))$ has a bifurcation.
	That is, Item \eqref{item8_Det_PN} holds.

	Construct extended concurrent composition $$\CCne(G)=(N'''=(P''',T''',Pre''',Post'''),M_0''')$$ from $\CCn(G)$ as in 
	\eqref{eqn28_Det_PN}.
	
	Then for $G$, \eqref{eqn39_Det_PN} holds if and only if for $\CCne(G)$, 
	\begin{subequations}\label{eqn40_Det_PN}
		\begin{align}
			&\text{there exists a firing sequence }M_0'''[s_1'''\rangle M_1'''[s_2'''\rangle M_2'''[s_3'''\rangle M_3'''
			\text{ such that}\label{eqn40_1_Det_PN}\\
			&M_1'''\le M_3''',\label{eqn40_2_Det_PN}\\
			&s_3'''\text{ contains a transition }(t_1,*)\text{ with }\ell(t_1)\in\Sig,\label{eqn40_3_Det_PN}\\
			&M_2'''|_{P_1}\ne M_2'''|_{P_2},\label{eqn40_4_Det_PN}\\
			&s_1''',s_2'''\in(T')^*,\label{eqn40_5_Det_PN}\\
			&s_3'''\in(T'\cup T_{\phi}^1)^*,\label{eqn40_6_Det_PN}
		\end{align}
	\end{subequations}
	where we omit a similar proof for the equivalence compared to the previous claim.
	Among \eqref{eqn40_2_Det_PN}-\eqref{eqn40_6_Det_PN},
	only \eqref{eqn40_4_Det_PN} is not a predicate. Using a similar construction to the one that is used to 
	reduce the satisfiability of \eqref{eqn20_Det_PN} for $\CCn(G)$ to the satisfiability of a Yen's path formula
	for $\CCn(G)'$, we can reduce the satisfiability of \eqref{eqn40_Det_PN} to the satisfiability of a Yen's path formula
	for a new Petri net. Hence, the satisfiability of Item \eqref{item8_Det_PN} for $G$ is decidable.

	\eqref{item4_Det_PN} Proof of the hardness result:

	To prove conclusion \ref{item4_Det_PN} of Theorem \ref{thm4_Det_PN},
	we are given a Petri net
	$G=(N=(P,T,Pre,Post),M_0)$ and a destination marking $M\in\N^{P}$, and construct
	a labeled P/T net $G'=(N'=(P',T',Pre',Post'),M_0',T\cup\{\s_{G}\},\ell)$ as in
	\eqref{eqn9_Det_PN}. 
	Then similar to the proof of Theorem \ref{thm3_Det_PN},
	it is also clear that $G$ does not cover $M$ if and only if
	$G'$ is eventually strongly detectable. Hence conclusion \ref{item4_Det_PN} holds.
\end{proof}

\begin{example}\label{exam1_Det_PN}
	Consider a labeled Petri net $G$ shown in Fig. \ref{fig15_Det_PN}, where event $a$ can be directly 
	observed, but $b$ and $c$ share the same label $b$.
	One directly sees that 
	$\LM^{\omega}(G)=\{(ab)^{\omega}\}$, and $\Mt(G,(ab)^n)=\{(1,0),(0,0)\}$ for all $n\in\Z_{+}$.
	Hence $G$ is not weakly detectable, and hence not eventually strongly detectable. By its reachability graph 
	shown in Fig. \ref{fig16_Det_PN}, one sees that this net 
	satisfies Item \eqref{item8_Det_PN} in the proof of Theorem \ref{thm4_Det_PN}, but not 
	Item \eqref{item7_Det_PN} in the proof. However, the net in Fig. \ref{fig9_Det_PN}
	satisfies Item \eqref{item7_Det_PN} but not Item \eqref{item8_Det_PN}.

		\begin{figure}[htbp]
		\tikzset{global scale/.style={
    scale=#1,
    every node/.append style={scale=#1}}}
		\begin{center}
			\begin{tikzpicture}[global scale = 1.0,
				>=stealth',shorten >=1pt,thick,auto,node distance=1.5 cm, scale = 0.8, transform shape,
	->,>=stealth,inner sep=2pt,
				every transition/.style={draw=red,fill=red,minimum width=1mm,minimum height=3.5mm},
				every place/.style={draw=blue,fill=blue!20,minimum size=7mm}]
				\tikzstyle{emptynode}=[inner sep=0,outer sep=0]
				\node[place, tokens=1, label=above:$p_1$] (p1) {};
				\node[transition, label=above:$b(b)$,right of = p1] (t3){}
				edge[post] (p1);
				\node[transition, label=above:$a(a)$,above of = t3] (t2) {}
				edge[pre] (p1);
				\node[place, right of = t3,label=above:$p_2$] (p3) {}
				edge[post] (t3)
				edge[pre] (t2);
				\node[transition, label=above:$c(b)$,right of = p3] (t4) {}
				edge[pre] (p3);
			\end{tikzpicture}
			\caption{A labeled P/T net $G$.}
			\label{fig15_Det_PN}
		\end{center}
	\end{figure}
	
				\begin{figure}[htbp]
		\begin{center}
			\begin{tikzpicture}[->,>=stealth,node distance=2.5cm,
		point/.style={circle,inner sep=0pt,minimum size=2pt,fill=red},
		skip loop/.style={to path={-- ++(0,#1) -| (\tikztotarget)}},
		state/.style={
		rectangle,minimum size=6mm,rounded corners=3mm,
		very thin,draw=black!50,
		top color=white,bottom color=black!50,font=\ttfamily\bf},
		input/.style={
		rectangle,
		minimum size=6mm,
		very thin,
		font=\itshape\bf},
		output/.style={
		rectangle,rounded corners=5.5mm,drop shadow={opacity=0.5},
		minimum size=6mm,
		very thin,draw=black!50,
		top color=white, 
		bottom color=black!50,
		font=\itshape\bf}]
				\matrix[row sep=6mm,column sep=5mm] {
				  \node (11) [input] {$(1,0)$}; & \node (12) [input] {$(0,1)$};
				& \node (13) [input] {$(1,0)$}; & \node (14) [input] {$(0,1)$};
				& \node (15) [input] {$(1,0)$}; & \node (16) [input] {$\cdots$};\\
				& & \node (23) [input] {$(0,0)$}; & & \node (25) [input] {$(0,0)$};\\
		};

		\path 
		(11) edge node [above, sloped]  {$a(a)$} (12)
		(12) edge node [above, sloped]  {$b(b)$} (13)
		(13) edge node [above, sloped]  {$a(a)$} (14)
		(14) edge node [above, sloped]  {$b(b)$} (15)
		(15) edge node [above, sloped]  {$a(a)$} (16)
		(12) edge node [above, sloped]  {$c(b)$} (23)
		(14) edge node [above, sloped]  {$c(b)$} (25)
		;

			\end{tikzpicture}
			\caption{Reachability graph of the labeled Petri net in Fig \ref{fig15_Det_PN}.}
			\label{fig16_Det_PN}
		\end{center}
	\end{figure}
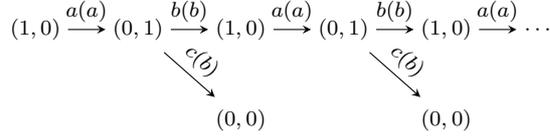

\end{example}

\begin{example}\label{exam2_Det_PN}
	Consider a labeled Petri net $G$ shown in Fig. \ref{fig17_Det_PN}. Its reachability graph is shown in Fig.
	\ref{fig18_Det_PN}, one has $\LM^{\omega}(G)=\{ab^{\omega}\}$.
	By the reachability graph, one sees that the net is not prompt, since there is a
	repetitive firing sequence in $(1,0,0,0,0)\xrightarrow[]{t_2(a)}(0,0,1,0,0)\xrightarrow[]{t_4(\epsilon)}
	(0,0,1,1,0)$
	labeled by the empty string.
	This net is not eventually strongly detectable, since for each $n\in\Z_{+}$, $|\Mt(G,ab^n)|=\infty>1$.
	However, the net does not satisfy Item \eqref{item7_Det_PN} or Item \eqref{item8_Det_PN}
	in the proof of Theorem \ref{thm4_Det_PN}.

		\begin{figure}[htbp]
		\tikzset{global scale/.style={
    scale=#1,
    every node/.append style={scale=#1}}}
		\begin{center}
			\begin{tikzpicture}[global scale = 1.0,
				>=stealth',shorten >=1pt,thick,auto,node distance=1.5 cm, scale = 0.8, transform shape,
	->,>=stealth,inner sep=2pt,
				every transition/.style={draw=red,fill=red,minimum width=1mm,minimum height=3.5mm},
				every place/.style={draw=blue,fill=blue!20,minimum size=7mm}]
				\tikzstyle{emptynode}=[inner sep=0,outer sep=0]
				\node[place, tokens=1,label=above:$p_1$] (p1) {};
				\node[transition, label=above:$t_1(a)$,right of = p1] (t1) {}
				edge[pre] (p1);
				\node[transition, label=above:$t_2(a)$,left of = p1] (t2) {}
				edge[pre] (p1);
				\node[place, right of = t1, label=above:$p_2$] (p2) {}
				edge[pre] (t1);
				\node[transition, label=above:$t_3(b)$, right of = p2] (t3) {}
				edge[pre, bend left] (p2)
				edge[post, bend right] (p2);
				\node[place, left of = t2, label=above:$p_3$] (p3) {}
				edge[pre] (t2);
				\node[transition, label=above:$t_4(\epsilon)$, left of = p3] (t4) {}
				edge[pre, bend left] (p3)
				edge[post, bend right] (p3);
				\node[place, left of = t4, label=above:$p_4$] (p4) {}
				edge[pre] (t4);
				\node[transition, label=above:$t_6(b)$, above of = t4] (t6) {}
				edge[pre] (p3);
				\node[place, left of = t6, label=above:$p_5$] (p5) {}
				edge[pre] (t6);
				\node[transition, label=below:$t_5(b)$, left of = p5] (t5) {}
				edge[pre] (p4)
				edge[post, bend left] (p5)
				edge[pre, bend right] (p5);
			\end{tikzpicture}
			\caption{A labeled P/T net $G$.}
			\label{fig17_Det_PN}
		\end{center}
	\end{figure}

			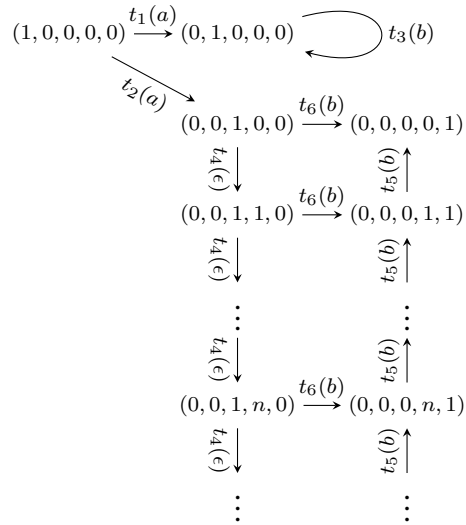
\begin{figure}[htbp]
		\begin{center}
			\begin{tikzpicture}[->,>=stealth,node distance=2.0cm,
		point/.style={circle,inner sep=0pt,minimum size=2pt,fill=red},
		skip loop/.style={to path={-- ++(0,#1) -| (\tikztotarget)}},
		state/.style={
		rectangle,minimum size=6mm,rounded corners=3mm,
		very thin,draw=black!50,
		top color=white,bottom color=black!50,font=\ttfamily\bf},
		input/.style={
		rectangle,
		minimum size=6mm,
		very thin,
		font=\itshape\bf},
		output/.style={
		rectangle,rounded corners=5.5mm,drop shadow={opacity=0.5},
		minimum size=6mm,
		very thin,draw=black!50,
		top color=white, 
		bottom color=black!50,
		font=\itshape\bf}]
				\matrix[row sep=6mm,column sep=5mm] {
				  \node (11) [input] {$(1,0,0,0,0)$}; & \node (12) [input] {$(0,1,0,0,0)$};\\
				& \node (22) [input] {$(0,0,1,0,0)$}; & \node (23) [input] {$(0,0,0,0,1)$};\\
				& \node (32) [input] {$(0,0,1,1,0)$}; & \node (33) [input] {$(0,0,0,1,1)$};\\
				& \node (42) [input] {$\vdots$}; & \node (43) [input] {$\vdots$};\\
				& \node (52) [input] {$(0,0,1,n,0)$}; & \node (53) [input] {$(0,0,0,n,1)$};\\
				& \node (62) [input] {$\vdots$}; & \node (63) [input] {$\vdots$};\\
		};

		\path 
		(11) edge node [above, sloped]  {$t_1(a)$} (12)
		(11) edge node [below, sloped]  {$t_2(a)$} (22)
		(22) edge node [below, sloped]  {$t_4(\epsilon)$} (32)
		(32) edge node [below, sloped]  {$t_4(\epsilon)$} (42)
		(42) edge node [below, sloped]  {$t_4(\epsilon)$} (52)
		(52) edge node [below, sloped]  {$t_4(\epsilon)$} (62)
		(12) edge [loop right] node {$t_3(b)$} (12)
		(22) edge node [above, sloped] {$t_6(b)$} (23)
		(32) edge node [above, sloped] {$t_6(b)$} (33)
		(52) edge node [above, sloped] {$t_6(b)$} (53)
		(63) edge node [above, sloped] {$t_5(b)$} (53)
		(53) edge node [above, sloped] {$t_5(b)$} (43)
		(43) edge node [above, sloped] {$t_5(b)$} (33)
		(33) edge node [above, sloped] {$t_5(b)$} (23)
		;
		
%
			\end{tikzpicture}
			\caption{Reachability graph of the labeled Petri net in Fig \ref{fig17_Det_PN}.}
			\label{fig18_Det_PN}
		\end{center}
	\end{figure}


\end{example}

\section{Conclusion}\label{sec:conc}

In this paper, we obtained a series of results on detectability of discrete-event systems.
We proposed one new notion of weak detectability and two new notions of strong detectability.
We proved that (1) the problem of
verifying weak approximate detectability of labeled Petri nets is undecidable;
(2) the problem of verifying instant strong detectability of labeled Petri nets is
decidable and EXPSPACE-hard;
(3) the problem of verifying eventual strong
detectability of labeled Petri nets is decidable and EXPSPACE-hard under the promptness 
assumption;
(4) for finite automata, the problem of verifying weak approximate detectability is PSPACE-complete,
and the other two properties can be verified in polynomial time.
(5) The relationships between thse notions of detectability were also characterized, and it was proved
that no two of them are equivalent.

Among the relationship between these notions,
the open question whether there exists a reduction from weak detectability to
weak approximate detectability is an important one. It is because the latter is a natural generalization of
the former. If the answer is yes, then the undecidability result of weak approximate detectability
for labeled Petri nets (Theorem \ref{thm2_Det_PN}) immediately follows from the undecidability result
of weak detectability of labeled Petri nets proved in \cite{Masopust2018DetectabilityPetriNet}.
Other variants of notions of detectability, e.g., instant weak detectability, different notions
of approximate detectability are left for further study.
Uniform versions of these notions of detectability are left for further study.
It is also an interesting topic to look for fast algorithms for verifying these notions for
(bounded) labeled Petri nets.





%
%

\end{document}